\documentclass[reqno,a4paper,twoside]{amsart}

\usepackage[a4paper,margin=0.8in,footskip=0.25in]{geometry}

\usepackage{enumitem}

\setenumerate{label=\textnormal{(\arabic*)}}

\allowdisplaybreaks[4]

\numberwithin{equation}{section}

\usepackage{amsmath,amssymb,dsfont,verbatim,bm,mathrsfs}

\usepackage{mathtools}
\usepackage[raggedright]{titlesec}
\usepackage{autobreak}

\titleformat{\chapter}[display]
{\normalfont\huge\bfseries}{\chaptertitlename\\thechapter}{20pt}{\Huge}
\titleformat{\section}
{\normalfont\Large\bfseries\center}{\thesection}{1em}{}
\titleformat{\subsection}
{\normalfont\large\bfseries}{\thesubsection}{1em}{}
\titleformat{\subsubsection}
{\normalfont\normalsize\bfseries}{\thesubsubsection}{1em}{}
\titleformat{\paragraph}[runin]
{\normalfont\normalsize\bfseries}{\theparagraph}{1em}{}
\titleformat{\subparagraph}[runin]
{\normalfont\normalsize\bfseries}{\thesubparagraph}{1em}{}

\titlespacing*{\chapter} {0pt}{50pt}{40pt}
\titlespacing*{\section} {0pt}{3.5ex plus 1ex minus .2ex}{2.3ex plus .2ex}
\titlespacing*{\subsection} {0pt}{3.25ex plus 1ex minus .2ex}{1.5ex plus .2ex}
\titlespacing*{\subsubsection}{0pt}{3.25ex plus 1ex minus .2ex}{1.5ex plus .2ex}
\titlespacing*{\paragraph} {0pt}{3.25ex plus 1ex minus .2ex}{1em}
\titlespacing*{\subparagraph} {\parindent}{3.25ex plus 1ex minus .2ex}{1em}

\usepackage{tikz}
\usetikzlibrary{matrix}

\usepackage{amsrefs}

\usepackage{titletoc}
\usepackage{stackengine}
\usepackage{scalerel}

\contentsmargin{2.55em}
\dottedcontents{section}[1.5em]{}{2em}{1pc}
\dottedcontents{subsection}[4.35em]{}{2.8em}{1pc}
\dottedcontents{subsubsection}[7.6em]{}{3.2em}{1pc}
\dottedcontents{paragraph}[10.3em]{}{3.2em}{1pc}

\usepackage[noautomatic,nonewpage]{imakeidx}
\usepackage[columns=3,rule=0pt]{idxlayout}
\setindexprenote{For convenience of the reader we list below almost all symbols used in this work, indicating the page in which each of them is introduced. We only except those that are standard.}

\makeindex

\newtheorem{theorem}{Theorem}[section]
\newtheorem{lemma}[theorem]{Lemma}
\newtheorem{proposition}[theorem]{Proposition}
\newtheorem{corollary}[theorem]{Corollary}

\theoremstyle{definition}
\newtheorem{definition}[theorem]{Definition}
\newtheorem{notations}[theorem]{Notations}
\newtheorem{notation}[theorem]{Notation}
\newtheorem{example}[theorem]{Example}

\theoremstyle{remark}
\newtheorem{remark}[theorem]{Remark}

\DeclareMathOperator{\codom}{codom}
\DeclareMathOperator{\dom}{dom}

\DeclareMathOperator{\ide}{id}
\DeclareMathOperator{\End}{End}
\DeclareMathOperator{\Ext}{Ext}

\DeclareMathOperator{\Ob}{Ob}

\DeclareMathOperator{\ima}{Im}
\DeclareMathOperator{\BC}{BC}
\DeclareMathOperator{\BN}{BN}
\DeclareMathOperator{\BP}{BP}
\DeclareMathOperator{\Ho}{H}
\DeclareMathOperator{\HH}{HH}
\DeclareMathOperator{\HC}{HC}
\DeclareMathOperator{\HP}{HP}
\DeclareMathOperator{\HN}{HN}
\DeclareMathOperator{\Hom}{Hom}

\DeclareMathOperator{\Tor}{Tor}

\newcommand{\ov}{\overline}
\newcommand{\ot}{\otimes}
\newcommand{\wh}{\widehat}
\newcommand{\wt}{\widetilde}

\newcommand{\ba}{\mathbf a}
\newcommand{\bx}{\mathbf x}

\newcommand{\bh}{\mathbf h}
\newcommand{\brh}{\mathrm h}

\newcommand{\byy}{\mathbf y}

\newcommand{\hs}{\hspace{-0.5pt}}
\newcommand{\hsm}{\hspace{-0.5pt}}
\newcommand{\xcdot}{\cdot}
\newcommand{\xcirc}{\hsm\circ\hsm}

\DeclareMathAlphabet{\mathpzc}{OT1}{pzc}{m}{it}

\usepackage[pdftex,breaklinks]{hyperref}

\begin{document}

\title{(Co)homology of Crossed Products by Weak Hopf Algebras}

\author{Jorge A. Guccione}
\address{Departamento de Matem\'atica\\ Facultad de Ciencias Exactas y Naturales-UBA, Pabell\'on~1-Ciudad Universitaria\\ Intendente Guiraldes 2160 (C1428EGA) Buenos Aires, Argentina.}
\address{Instituto de Investigaciones Matem\'aticas ``Luis A. Santal\'o"\\ Facultad de Ciencias Exactas y Natu\-ra\-les-UBA, Pabell\'on~1-Ciudad Universitaria\\ Intendente Guiraldes 2160 (C1428EGA) Buenos Aires, Argentina.}
\email{vander@dm.uba.ar}

\author{Juan J. Guccione}
\address{Departamento de Matem\'atica\\ Facultad de Ciencias Exactas y Naturales-UBA\\ Pabell\'on~1-Ciudad Universitaria\\ Intendente Guiraldes 2160 (C1428EGA) Buenos Aires, Argentina.}
\address{Instituto Argentino de Matem\'atica-CONICET\\ Savedra 15 3er piso\\ (C1083ACA) Buenos Aires, Argentina.}
\email{jjgucci@dm.uba.ar}

\thanks{Jorge A. Guccione and Juan J. Guccione were supported by CONICET PIP 2021-2023 GI,11220200100423CO and CONCYTEC-FONDECYT within the framework of the contest ``Proyectos de Investigaci\'on B\'asica 2020-01'' [contract number 120-2020-FONDECYT]}

\author{Christian Valqui}
\address{Pontificia Universidad Cat\'olica del Per\'u, Secci\'on Matem\'aticas, PUCP, Av. Universitaria 1801, San Miguel, Lima 32, Per\'u.}
\address{Instituto de Matem\'atica y Ciencias Afines (IMCA) Calle Los Bi\'ologos 245. Urb San C\'esar.
La Molina, Lima 12, Per\'u.}
\email{cvalqui@pucp.edu.pe}

\thanks{Christian Valqui was supported by CONCYTEC-FONDECYT within the framework of the contest ``Proyectos de Investigaci\'on B\'asica 2020-01" [contract number 120-2020-FONDECYT]}

\subjclass[2010]{primary 16E40; secondary 16T05}
\keywords{Crossed products, Hochschild (co)homology, Cyclic homology, Weak Hopf algebras}

\begin{abstract} We obtain a mixed complex simpler than the canonical one that computes the cyclic type homologies of a crossed product with invertible cocycle $A\times_{\rho}^f H$, of a weak module algebra $A$ by a weak Hopf algebra $H$. This complex is endowed with a filtration. The spectral sequence of this filtration generalizes the spectral sequence obtained in \cite{CGG}. When $f$ takes its values in a separable subalgebra of $A$ that satisfies suitable conditions, the above mentioned mixed complex is provided with another filtration, whose spectral sequence generalize the Feigin-Tsygan spectral sequence.
\end{abstract}

\maketitle

\tableofcontents

\section*{Introduction}

Given a differential or algebraic manifold $M$, each group $G$ acting on $M$ acts in a natural way on the ring $A$ of regular functions of $M$, and the algebra ${}^{G\!} A$ of invariants of this action consists of the functions that are constants on each of the orbits of $M$. This suggest to consider ${}^{G\!} A$ as a replacement for $M/G$ in noncommutative geometry. Under suitable conditions the invariant algebra ${}^{G\!} A$ and the smash product $A\# k[G]$, associated with the action of $G$ on $A$, are Morita equivalent. Since $K$-theory, Hochschild homology and cyclic homology are Morita invariant, there is no loss of information if ${}^{G\!} A$ is replaced by $A\# k[G]$. In the general case the experience has shown that smash products are better choices than invariant rings for algebras playing the role of noncommutative quotients. In fact, except when the invariant algebra and the smash product are Morita equivalent, the first one never is considered in noncommutative geometry. The problem of developing tools to compute the cyclic homology of smash products algebras $A\# k[G]$, where $A$ is an algebra and $G$ is a group, was considered in \cites{FT, GJ, N}. For instance, in the first paper the authors obtained a spectral sequence converging to the cyclic homology of $A\# k[G]$, and in \cite{GJ} this result was derived from the theory of paracyclic modules and cylindrical modules developed by the authors. The main tool for this computation is a version for cylindrical modules of the Eilenberg-Zilber theorem. More recently, and also due to its connections with noncommutative geometry, the cyclic homology of algebras obtained through more general constructions involving Ore extensions, twisted tensor products or Hopf algebras (Hopf crossed products, Hopf Galois extensions, Braided Hopf crossed products, etcetera) has been extensively studied. See for instance \cites{AK, CGG, CGG1, CGG2, CGGV, FGG, GG2, GG3, JS, KR, KS, R, V, W, ZH}. Weak Hopf algebras (also called quantum groupoids) are an important generalization of Hopf algebras in which the counit is not required to be an algebra homomorphism and the unit is not required to be a coalgebra homomorphism, these properties being replaced by weaker axioms. Examples of weak Hopf algebras are groupoid algebras and their duals, face algebras \cite{Ha}, quantum groupoids constructed from subfactors \cite{NV} and generalized Kac algebras of Yamanouchi \cite{Y}. It is natural to try to extend the results of \cite{FT} to noncommutative quotients $A\# H$ of algebras $A$ by actions of weak Hopf algebras $H$. More generally, in this paper we use the results obtained in \cite{GGV2} to study the Hochschild (co)homology and the cyclic type homologies of weak crossed products with invertible cocycle. Specifically, for a unitary crossed product $E\coloneqq A\times_{\rho}^f H$, of an algebra $A$ by a weak Hopf al\-ge\-bra $H$ and a subalgebra $K$ of $A$ satisfying suitable conditions, we construct a chain complex, a cochain complex and a mixed complex, simpler than the canonical ones (and also simpler that the ones constructed in \cite{GGV2}), that compute the Hochschild (co)homology of $E$ with coefficients in an $E$-bimodule $M$, relative to $K$, and the cyclic, negative and periodic homologies of $E$, relative to $K$ (see \cite{GS}). It is well known that when $K$ is separable, then these relative groups coincide with the absolute groups. These complexes are endowed with canonical filtrations whose spectral sequences generalize the Hochschild-Serre spectral sequence and the Feigin and Tsygan spectral sequence. By example, applying these results, we obtain that when $A=K$, the Hochschild homology of $E$ with coefficients in $M$, relative to $K$, coincide with the homology of $H$ with coefficients in the group $M/[M,K]$, considered as a left $H$-module via conjugation. A similar result is obtained for the Hochschild cohomology (see Subsection~\ref{(Co)homology of weak Hopf algebras} and Examples~\ref{A es K homologia} and~\ref{A es K cohomologia}).

\smallskip

The paper is organized as follows:

\smallskip

In Section~1 we review the notions of weak Hopf algebra and of crossed products of algebras by weak Hopf algebras, and we recall the concept of mixed complex and the perturbation lemma. In sections~2 and~3 we obtain complexes that compute the Hochschild homology and the Hochschild cohomology of a weak crossed product with invertible cocycle $E\coloneqq A\times_{\rho}^f H$ with coefficients in an $E$-bimodule $M$. Then, in Section~4 we study the cup and the cap products of $E$, and in Section~\ref{cyclic homology of cleft extensions} we obtain a mixed complex that computes the cyclic type homologies of $E$.

\smallskip

\noindent\textsl{Remark.} In this paper we consider the notion of weak crossed products introduced in~\cite{AG}, but this is not the unique concept of weak crossed product of algebras by weak Hopf algebras in the literature. There is a notion of crossed product of an algebra $A$ with a Hopf algebroid introduced by B\"ohm and Brzezi\'nski in \cite{BB}. It is well known that weak Hopf algebras $H$ provide examples of Hopf algebroids. The crossed products considered by us in this paper are canonically isomorphic to the B\"ohm-Brzezi\'nski crossed products $A\#_f H$ with invertible cocycle whose action satisfies $h\cdot (l\cdot 1_A) = hl\cdot 1_A$ for all $h\in H$ and $l\in H^L$ (see \cite{GG1}). So, our results also apply to these algebras.

\smallskip

We thank the referee for a careful reading of the paper and for his indications that helped improve the writing. We also thank the referee for pointing out the paper \cite{KS}, in which the Hochschild homology of the twisted tensor product of algebras is studied, obtaining complete calculations for several examples. There is some intersec\-tion between the results of \cite{KS} and some results obtained in \cite{GG2} and \cite{GG3}. More precisely, \cite{KS}*{Proposition 2.1} is \cite{GG2}*{Theorem~1.7} and the results in \cite{KS}*{Subsections~3.3.1 to~3.3.3} are the cases $r=n=2$, $r=1, n=2$ and~$r=0, n=2$ of \cite{GG3}*{Proposition~1.9}.

\section{Preliminaries}
In this article we work in the category of vector spaces over a field $k$. Hence we assume implicitly that all the maps are $k$-linear maps. The tensor product over $k$ is denoted by $\ot_k$. Given an ar\-bitrary algebra $K$, a $K$-bimodule $V$ and an $n\ge 0$, we let $V^{\ot_{\hs K}^n}$ denote the $n$-fold tensor product $V\ot_{\hs K}\cdots\ot_{\hs K} V$, which is considered as a $K$-bimodule via
$$
\lambda\xcdot (v_1\ot_{\hs K}\cdots \ot_{\hs K} v_n)\xcdot \lambda' \coloneqq \lambda\xcdot v_1\ot_{\hs K}\cdots \ot_{\hs K} v_n\xcdot \lambda'.
$$
Given $k$-vector spaces $U$, $V$, $W$ and a map $g\colon V\to W$ we write $U\ot_k g$ for $\ide_U\ot_k g$ and $g\ot_k U$ for $g\ot_k\ide_U$. We assume that the reader is familiar with the notions of weak Hopf algebra introduced in~\cites{BNS1,BNS2} and of weak crossed products introduced in~\cite{AG} and studied in a series of papers (see for instance \cites{AFGR1, AFGR2, AFGR3, FGR, GGV1, Ra}). We are specifically interested in the case in which $A$ is a weak module algebra and the cocycle of the crossed product is convolution invertible (see \cite{GGV1}*{Sections~4--6}).

\subsection{Weak Hopf algebras}\label{subsection: Weak Hopf algebras}

Weak bialgebras and weak Hopf algebras are generalizations of bialgebras and Hopf algebras, introduced in~\cites{BNS1,BNS2}, in which the axioms about the unit, the counit and the antipode are replaced by weaker properties. Next we give a brief review of the basic properties of these structures.

\smallskip

\begin{definition} Let $k$ be a field. A weak {\em bialgebra} is a $k$-vector space $H$, endowed with an algebra structure and a~co\-algebra structure, such that $\Delta(hl) =\Delta(h)\Delta(l)$ for all $h,l\in H$, and the equalities
\begin{align}
\label{propiedad de 1}
& \Delta^2(1)= 1^{(1)}\ot_k 1^{(2)} 1^{(1')}\ot_k 1^{(2')}= 1^{(1)}\ot_k 1^{(1')} 1^{(2)}\ot_k 1^{(2')}
\shortintertext{and}
\label{propiedad de epsilon}
& \epsilon(hlm)=\epsilon(hl^{(1)})\epsilon(l^{(2)}m)=\epsilon(hl^{(2)})\epsilon(l^{(1)}m)\quad\text{for all $h,l,m\in H$,}
\end{align}
are fulfilled, where we are using the Sweedler notation for the coproduct, with the summation symbol omitted. A {\em weak bialgebra morphism} is a function $g\colon H\to L$ that is an algebra and a coalgebra map.
\end{definition}

In the rest of this subsection we assume that $H$ is a weak bialgebra.

It is well known that the maps $\Pi^L,\Pi^R, \overline{\Pi}^L, \overline{\Pi}^R\in \End_k(H)$,\index{zqa@$\Pi^L$|dotfillboldidx}\index{zqb@$\Pi^R$|dotfillboldidx}\index{zqc@$\overline{\Pi}^L$|dotfillboldidx}
\index{zqd@$\overline{\Pi}^R$|dotfillboldidx}
defined by
$$
\Pi^L(h)\hs\coloneqq\hs \epsilon(1^{(1)}h)1^{(2)},\hs\quad \Pi^R(h)\hs\coloneqq\hs 1^{(1)}\epsilon(h1^{(2)}),\hs\quad \overline{\Pi}^L(h)\hs\coloneqq\hs 1^{(1)}\epsilon(1^{(2)}h),\hs\quad\overline{\Pi}^R(h)\hs\coloneqq\hs \epsilon(h1^{(1)})1^{(2)},
$$
respectively, are idempotent (for a proof see~\cites{BNS1, CDG}). We set $H^{\!L}\coloneqq \ima(\Pi^L)\index{ha@$H^L$|dotfillboldidx}$ and $H^{\!R}\coloneqq \ima(\Pi^R)\index{hb@$H^R$|dotfillboldidx}$. In~\cite{CDG} it was also proven that $\ima(\overline{\Pi}^R)=H^{\!L}$ and $\ima(\overline{\Pi}^L) = H^{\!R}$.

\begin{proposition}\label{para buena def} For all $h,l\in H$ and $m\in H^{\!L}$, we have
$$
\ov{\Pi}^R(hl) = \ov{\Pi}^R(\ov{\Pi}^R(h)l),\quad \Pi^R(hl) = \Pi^R(\ov{\Pi}^R(h)l)\quad\text{and}\quad \ov{\Pi}^R(hm) = \ov{\Pi}^R(h)m.
$$
\end{proposition}

\begin{proof} Left to the reader.
\end{proof}

An {\em antipode} of $H$ is a map $S\colon H\to H$\index{sb@$S$|dotfillboldidx} (or $S_H$ if necessary), such that
$$
h^{(1)}S(h^{(2)})=\Pi^L(h),\quad S(h^{(1)})h^{(2)}=\Pi^R(h)\quad\text{and}\quad S(h^{(1)})h^{(2)}S(h^{(3)})=S(h),
$$
for all $h\in H$. As it was shown in~\cite{BNS1}, if an antipode $S$ exists, then it is unique. It was also shown in~\cite{BNS1} that $S$ is antimultiplicative, anticomultiplicative and leaves the unit and counit invariant. 
A {\em weak Hopf algebra} is a weak bialgebra that has an antipode. By \cite{BNS1}*{equalities~(2.24a) and~(2.24b)},
\begin{equation}\label{pepe}
\Pi^L = S\xcirc \ov{\Pi}^L\qquad\text{and}\qquad \Pi^R = S\xcirc \ov{\Pi}^R.
\end{equation}
A {\em morphism of weak Hopf algebras} $g\colon H\to L$ is simply a bialgebra morphism from $H$ to $L$. In~\cite{AFGLV}*{Pro\-position~1.4} it was proven that if $g\colon H\to L$ is a weak Hopf algebra morphism, then $g\xcirc S_H= S_L\xcirc g$.

\subsection{(Co)homology of weak Hopf algebras}\label{(Co)homology of weak Hopf algebras}

Consider $H^{\!R}$ as a right $H$-module via $l\xcdot h\coloneqq \Pi^R(lh)$ (by~\cite{BNS1}*{equality~(2.5b)} this is an action and the map $\Pi^R\colon H\to H^R$ is a morphism of right $H$-modules). By definition, the {\em homology of $H$ with coefficients in a left $H$-module $N$} is $\Ho_*(H,N)\coloneqq \Tor^H_*(H^{\!R},N)$, while the {\em cohomology of $H$ with coefficients in a right $H$-module $N$} is $\Ho^*(H,N)\coloneqq \Ext_H^*(H^{\!R},N)$.

\begin{notations}\label{not 7.1'} We will use the following notations:

\begin{enumerate}[itemsep=0.7ex, topsep=1.0ex, label=(\arabic*)]

\item We set $\ov{H}\coloneqq H/H^{\!L}$\index{hc@$\ov{H}$|dotfillboldidx}. Moreover, given $h\in H$ we let $\ov{h}$\index{hi@$\ov{h}$|dotfillboldidx} denote its class in $\ov{H}$.

\item Given $h_1,\dots, h_s\in H$ we set $\ov{\bh}_{1s}\coloneqq \ov{h_1}\ot_{H^{\!L}} \cdots \ot_{H^{\!L}} \ov{h_s}$.\index{hiab@$\ov{\bh}_{1s}$|dotfillboldidx}

\end{enumerate}
\end{notations}

Clearly $\ov{H}^{\ot_{\! H^{\!L}}^s}$ is an $H^{\!L}$-bimodule via $l\xcdot \ov{\bh}_{1s}\xcdot l' \coloneqq \ov{lh_1} \ot_{\! H^{\!L}} \ov{\bh}_{2,s-1}\ot_{\! H^{\!L}} \ov{h_sl}$.

\begin{proposition}\label{hom} The following facts hold:

\begin{enumerate}[itemsep=0.7ex, topsep=1.0ex, label=(\arabic*)]

\item The homology of $H$ with coefficients in a left $H$-module $N$ is the ho\-mology of the chain complex
\begin{equation*}
\begin{tikzpicture}
\begin{scope}[yshift=0cm,xshift=0cm]
\matrix(BP2complex) [matrix of math nodes,row sep=0em, text height=1.5ex, text
depth=0.25ex, column sep=2.5em, inner sep=0pt, minimum height=5mm,minimum width =6mm]
{N & \ov{H}\ot_{H^{\!L}} N & \ov{H}^{\ot_{\!H^{\!L}}^2}\ot_{H^{\!L}} N & \ov{H}^{\ot_{\!H^{\!L}}^3}\ot_{H^{\!L}} N & \ov{H}^{\ot_{\!H^{\!L}}^4}\ot_{H^{\!L}} N &\cdots,\\};
\draw[<-] (BP2complex-1-1) -- node[above=1pt,font=\scriptsize] {$d_1$} (BP2complex-1-2);
\draw[<-] (BP2complex-1-2) -- node[above=1pt,font=\scriptsize] {$d_2$} (BP2complex-1-3);
\draw[<-] (BP2complex-1-3) -- node[above=1pt,font=\scriptsize] {$d_3$} (BP2complex-1-4);
\draw[<-] (BP2complex-1-4) -- node[above=1pt,font=\scriptsize] {$d_4$} (BP2complex-1-5);
\draw[<-] (BP2complex-1-5) -- node[above=1pt,font=\scriptsize] {$d_5$} (BP2complex-1-6);
\end{scope}
\end{tikzpicture}
\end{equation*}
where $d_1\bigl(h_1\ot_{H^{\!L}} n\bigr)\coloneqq \ov{\Pi}^R(h_1)\xcdot n - h_1\xcdot n$ and, for $s>1$,
\begin{align*}
%
%
d_s\bigl(\ov{\bh}_{1s}\ot_{H^{\!L}} n\bigr)& \coloneqq \ov{\Pi}^R(h_1)\xcdot \ov{\bh}_{2s}\ot_{H^{\!L}} n\\
& + \sum_{i=1}^{s-1} (-1)^i \ov{\bh}_{1,i-1}\ot_{H^{\!L}}\ov{h_ih_{i+1}}\ot_{H^{\!L}} \ov{\bh}_{i+2,s} \ot_{H^{\!L}} n + (-1)^s \ov{\bh}_{1,s-1}\ot_{H^{\!L}} h_s\xcdot n.
\end{align*}

\item The cohomology of $H$ with coefficients in a right $H$-module $N$ is the cohomology of the cochain complex
\begin{equation*}
\begin{tikzpicture}
\begin{scope}[yshift=0cm,xshift=0cm]
\matrix(BP2complex) [matrix of math nodes,row sep=0em, text height=1.5ex, text
depth=0.25ex, column sep=2.5em, inner sep=0pt, minimum height=5mm,minimum width =6mm]
{N & \Hom_{\!H^{\!L}} (\ov{H},N) & \Hom_{\!H^{\!L}} (\ov{H}^{\ot_{\!H^{\!L}}^2},N) & \Hom_{\!H^{\!L}} (\ov{H}^{\ot_{\!H^{\!L}}^3},N)  &\cdots,\\};
\draw[->] (BP2complex-1-1) -- node[above=1pt,font=\scriptsize] {$d^1$} (BP2complex-1-2);
\draw[->] (BP2complex-1-2) -- node[above=1pt,font=\scriptsize] {$d^2$} (BP2complex-1-3);
\draw[->] (BP2complex-1-3) -- node[above=1pt,font=\scriptsize] {$d^3$} (BP2complex-1-4);
\draw[->] (BP2complex-1-4) -- node[above=1pt,font=\scriptsize] {$d^4$} (BP2complex-1-5);
%
\end{scope}
\end{tikzpicture}
\end{equation*}
where $d^1(n)(h_1)\coloneqq n\xcdot \ov{\Pi}^R(h_1) - n\xcdot h_1$ and, for $s>1$,
\begin{equation*}
\qquad\quad d_s(\beta)(\ov{\bh}_{1s}) \coloneqq \beta\bigl(\ov{\Pi}^R(h_1)\xcdot\ov{\bh}_{2s}\bigr) + \sum_{i=1}^{s-1} (-1)^i \beta\bigl(\ov{\bh}_{1,i-1}\ot_{H^{\!L}}\ov{h_ih_{i+1}}\ot_{H^{\!L}}\ov{\bh}_{i+2,s} \bigr) + (-1)^s \beta\bigl(\ov{\bh}_{1,s-1}\bigr) \xcdot h_s.
\end{equation*}
\end{enumerate}
\end{proposition}

\begin{proof} For $s\in \mathds{N}$ and $0\le i\le s$, let $d''_{si}\colon H^{\ot_k^{s+1}}\to H^{\ot_k^{s+1}}$ be the map defined by
$$
d''_{si}(\bx) \coloneqq \begin{cases} \ov{\Pi}^R(h_1)h_2\ot_k h_3\ot_k\ot\cdots\ot_k h_{s+1} & \text{if $i=0$,}\\
h_1\ot_k\cdots\ot_k h_{i-1}\ot_k h_ih_{i+1}\ot_k h_{i+2}\ot_k \cdots \ot_k h_{s+1} & \text{if $1\le i\le s$,}
\end{cases}
$$
where $\bx\coloneqq  h_1\ot_k\cdots\ot_k h_{s+1}$. For each $s>1$ and $0\le i\le s$, we let $d'_{si}\colon H^{\ot_k^{s+1}}\to \ov{H}^{\ot_{\!H^{\!L}}^{s-1}}\ot_{H^{\!L}} H$ denote the morphism induced by $d''_{si}$. Moreover we set $d'_{10} = d''_{10}$ and $d'_{11} = d''_{11}$.
We claim that

\begin{enumerate}[itemsep=0.7ex, topsep=1.0ex, label=(\arabic*)]

\item If $s=1$ and $h_1\in H^{\!L}$, then $d'_{10}(\bx) = d'_{11}(\bx)$,

\item If $s>1$ and $h_1\in H^{\!L}$, then $d'_{s0}(\bx) = d'_{s1}(\bx)$ and $d'_{si}(\bx) = 0$ for $1<i$,

\item If $s>1$ and $h_j\in H^{\!L}$ for some $j>1$, then $d'_{s,j-1}(\bx) = d'_{sj}(\bx)$ and $d'_{si}(\bx) = 0$ for $i\notin\{j-1,j\}$,

\item $d'_{si}(h_1 l\ot_k h_2\ot_k\cdots\ot_k h_{s+1}) = d'_{si}(h_1 \ot_k lh_2\ot_k\cdots\ot_k h_{s+1})$ for $0\le i\le s$ and $l\in H^{\!L}$,

\item For $0\le i\le s$, $1\le j\le s$ and $l\in H^{\!L}$, we have
\begin{equation}\label{aux3}
\qquad d'_{si}(h_1\ot_k\cdots\ot_k h_j l\ot_k h_{j+1}\ot_k\cdots\ot_k h_{s+1}) = d'_{si}(h_1\ot_k\cdots\ot_k h_j\ot_k l h_{j+1}\ot_k\cdots\ot_k h_{s+1}).
\end{equation}
\end{enumerate}
Items~(1) and the first assertion in item~(2) hold since $\ov{\Pi}^R(h) = h$, for all $h\in H^{\!L}$; while the second assertion in item~(2) is clear. We next prove item~(3). The fact that $d'_{sj}(\bx) = 0$, for $i>0$ and $i\notin\{j-1,j\}$, is trivial; while the equality $d'_{s,j-1}(\bx) = d'_{sj}(\bx)$ is straightforward. Finally if $j=2$, then $d'_{s0}(\bx) = 0$, since $\ov{\Pi}^R(h_1)h_2 \in H^{\!L}$. The case $i=0$ in item~(4) holds, since $\ov{\Pi}^R(h_1l)h_2 = \ov{\Pi}^R(h_1)lh_2$, by the third identity in Proposition~\ref{para buena def}; while the cases $i>0$ are straightforward. It remains to check item~(4). Equality~\eqref{aux3} for $i=0$ and $j=1$ holds, since $\ov{\Pi}^R(h_1l)h_2 = \ov{\Pi}^R(h_1)lh_2$, by the third identity in Proposition~\ref{para buena def}; while the other cases are trivial. This ends the proof of the claim. Consequently, $\sum_{i=0}^s (-1)^1 d'_{si}$ induces a morphism $d'_s\colon \ov{H}^{\ot_{\!H^{\!L}}^s}\ot_{H^{\!L}} H \to \ov{H}^{\ot_{\!H^{\!L}}^{s-1}}\ot_{H^{\!L}} H$,~for each $s\ge 1$. Consider the diagram
\begin{equation}\label{resol}
\begin{tikzpicture}[baseline=(current bounding box.center)]
\begin{scope}[yshift=0cm,xshift=0cm]
\matrix(BP1complex) [matrix of math nodes,row sep=0em, text height=1.5ex, text
depth=0.25ex, column sep=2.5em, inner sep=0pt, minimum height=5mm,minimum width =6mm]
{H^{\!R} & H & \ov{H}\ot_{H^{\!L}} H & \ov{H}^{\ot_{\!H^{\!L}}^2}\ot_{H^{\!L}} H & \ov{H}^{\ot_{\!H^{\!L}}^3}\ot_{H^{\!L}} H &\cdots,\\};
\draw[<-] (BP1complex-1-1) -- node[above=1pt,font=\scriptsize] {$\Pi^R$} (BP1complex-1-2);
\draw[<-] (BP1complex-1-2) -- node[above=1pt,font=\scriptsize] {$d'_1$} (BP1complex-1-3);
\draw[<-] (BP1complex-1-3) -- node[above=1pt,font=\scriptsize] {$d'_2$} (BP1complex-1-4);
\draw[<-] (BP1complex-1-4) -- node[above=1pt,font=\scriptsize] {$d'_3$} (BP1complex-1-5);
\draw[<-] (BP1complex-1-5) -- node[above=1pt,font=\scriptsize] {$d'_4$} (BP1complex-1-6);
\end{scope}
\end{tikzpicture}
\end{equation}
where $\ov{H}^{\ot_{\!H^{\!L}}^s}\ot_{H^{\!L}} H$ is a right $H$-module via the canonical action. Clearly the $d'_s$'s are right $H$-linear maps. Next we will prove that~\eqref{resol} is a chain complex. Assume that $s>2$. Since
\begin{align*}
&d''_{s-1,i}\xcirc d''_{sj} = d''_{s-1,j-1}\xcirc d''_{si} && \text{for $1\le i<j$,}\\
& d''_{s-1,0}\xcirc d''_{sj} = d''_{s-1,j-1}\xcirc d''_{s0} && \text{for $1<j$,}\\
\end{align*}
the composition $d'_{s-1}\xcirc d'_s$ is the map induced by $d''_{s-1,0}\xcirc d''_{s0} - d''_{s-1,0}\xcirc d''_{s1}$. Thus,
$$
d'_{s-1}\xcirc d'_s\bigl(\ov{\bh}_{1s}\ot_{H^{\!L}} h_{s+1}\bigr) = \ov{\ov{\Pi}^R\bigl(\ov{\Pi}^R(h_1)h_2\bigr)h_3} \ot_{H^{\!L}} \ov{\bh}_{4s}\ot_{H^{\!L}} h_{s+1} - \ov{\ov{\Pi}^R(h_1h_2)h_3} \ot_{H^{\!L}} \ov{\bh}_{4s}\ot_{H^{\!L}} h_{s+1} = 0,
$$
where the last equality follows from the first identity in~Propo\-sition~\ref{para buena def}. In order to finish the proof that~\eqref{resol} is a chain complex, we must check that $\Pi^R\xcirc d'_1 = 0$ and $d'_1\xcirc d'_2 = 0$. But, this follows from the first two identities in Proposition~\ref{para buena def}.
%
%
We claim that the family of morphisms
$$
\hbar_0\colon H^{\!R}\longrightarrow H,\qquad  \hbar_{s+1}\colon \ov{H}^{\ot_{\!H^{\!L}}^s}\ot_{H^{\!L}} H\longrightarrow \ov{H}^{\ot_{\!H^{\!L}}^{s+1}}\ot_{H^{\!L}} H\quad\text{($s\ge 0$),}
$$
given by $\hbar_0(h)\coloneqq h$ and $\hbar_{s+1}(\ov{\bh}_{1s}\ot_{H^{\!L}} h_{s+1})\coloneqq (-1)^{s+1} \ov{\bh}_{1,s+1}\ot_{H^{\!L}} 1$ for $s\ge 0$,
is a contracting homotopy of~\eqref{resol} as a complex of right $H^{\!L}$-modules. In fact, this follows immediately from the equalities
\begin{align*}
& \Pi^R\xcirc \hbar_0(h) = h, \qquad \hbar_0\xcirc \Pi^R(h) = \Pi^R(h),\qquad d'_1\xcirc \hbar_1(h) = -\ov{\Pi}^R(h) + h,\\
& \hbar_s \xcirc d'_s(\ov{\bh}_{1s}\ot_{H^{\!L}} h_{s+1}) = (-1)^s \ov{\ov{\Pi}^R(h_1)h_2}\ot_{H^{\!L}} \ov{\bh}_{3,s+1}\ot_{H^{\!L}} 1\\
& \phantom{\hbar_s \xcirc d'_s(\ov{\bh}_{1s}\ot_{H^{\!L}} h_{s+1})} + \sum_{i=1}^s (-1)^{s+i} \ov{\bh}_{1,i-1}\ot_{H^{\!L}}\ov{h_ih_{i+1}} \ot_{H^{\!L}}\ov{\bh}_{i+2,s+1} \ot_{H^{\!L}} 1
\intertext{and}
& d'_{s+1}\xcirc \hbar_{s+1}(\ov{\bh}_{1s}\ot_{H^{\!L}} h_{s+1}) = (-1)^{s+1}\ov{\ov{\Pi}^R(h_1)h_2}\ot_{H^{\!L}} \ov{\bh}_{3,s+1}\ot_{H^{\!L}} 1 \\
&\phantom{d'_{s+1}\xcirc \hbar_{s+1}(\ov{\bh}_{1s}\ot_{H^{\!L}} h_{s+1})} + \sum_{i=1}^s (-1)^{s+i+1} \ov{\bh}_{1,i-1}\ot_{H^{\!L}}\ov{h_ih_{i+1}} \ot_{H^{\!L}}\ov{\bh}_{i+2,s+1} \ot_{H^{\!L}} 1\\
&\phantom{d'_{s+1}\xcirc \hbar_{s+1}(\ov{\bh}_{1s}\ot_{H^{\!L}} h_{s+1})} + \ov{\bh}_{1s}\ot_{H^{\!L}} h_{s+1}.
\end{align*}
Note that, since $H^{\!L}$ is separable, the right $H^{\!L}$-modules $\ov{H}^{\ot_{\!H^{\!L}}^s}$ are projective . Hence, $\ov{H}^{\ot_{\!H^{\!L}}^s} \ot_{H^{\!L}} H$ is a right projective $H$-module for each $s\ge 0$ and, consequently,~\eqref{resol} is a projective resolution of $H^{\!R}$ as a right $H$-module. Now items~(1) and~(2) follows immediately from this and the fact that $\bigl(\ov{H}^{\ot_{\!H^{\!L}}^s}\ot_{H^{\!L}} H\bigr)\ot_H N\simeq \ov{H}^{\ot_{\!H^{\!L}}^s}\ot_{H^{\!L}} N$ and $\Hom_H\bigl(\ov{H}^{\ot_{\!H^{\!L}}^s}\ot_{H^{\!L}} H, N\bigr)\simeq \Hom_{\!H^{\!L}}\bigl(\ov{H}^{\ot_{\!H^{\!L}}^s},N\bigr)$.
\end{proof}

\subsection{Crossed products by weak Hopf algebras}\label{subsection: Crossed products by weak Hopf algebras}
Let $H$ be a weak-Hopf algebra, $A$ an algebra and $\rho\colon H\ot_k A\to A$\index{zr@$\rho$|dotfillboldidx} a linear map. For $h\in H$ and $a\in A$, we set $h\xcdot a\coloneqq \rho(h\ot_k a)$. We say that $\rho$ is a {\em weak measure of $H$ on $A$} if
\begin{align}\label{equal}
h\xcdot (aa') = (h^{(1)}\xcdot a)(h^{(2)}\xcdot a')\quad\text{for all $h\in H$ and $a,a'\in A$.}
\end{align}
From now on $\rho$ denotes a weak measure of $H$ on $A$. Let $\chi_{\rho}\colon H\ot_k A\longrightarrow A\ot_k H$\index{zr@$\rho$!zwa@$\chi_{\rho}$|dotfillboldidx} be the map defined by $\chi_{\rho}(h\ot_k a) \coloneqq h^{(1)}\xcdot a\ot_k h^{(2)}$. By equality~\eqref{equal} the triple $(A,H,\chi_{\rho})$ is a twisted space (see \cite{GGV1}*{Defi\-nition~1.6}). By \cite{GGV1}*{Sub\-section~1.2} we know that $A\ot_k H$ is a non unitary $A$-bimodule~via
$$
a'\xcdot (a\ot_k h) \coloneqq a'a\ot_k h\qquad\text{and}\qquad (a\ot_k h)\xcdot a' \coloneqq a(h^{(1)}\cdot a')\ot_k h^{(2)},
$$
and that  the map $\nabla_{\!\rho}\colon A\ot_k H\longrightarrow A\ot_k H$,\index{zr@$\rho$!zwaa@$\nabla_{\rho}$|dotfillboldidx} defined by
$\nabla_{\!\rho}(a\ot_k h) \coloneqq a\cdot \chi_{\rho}(h\ot_k 1_A)$, is a left and right $A$-linear idempotent. In the sequel we will write $a\times h\coloneqq \nabla_{\!\rho}(a\ot_k h)$\index{zzma@$a\times h$|dotfillboldidx}. It is easy to check that the $A$-subbimodule $A\times H\coloneqq \nabla_{\!\rho}(A\ot_k H)$ of $A\ot_k H$ is unitary. Let $\gamma\colon H\to A\times H$\index{zc@$\gamma$|dotfillboldidx}, $\nu\colon k\to A\ot_k H$\index{zn@$\nu$|dotfillboldidx} and $\jmath_{\nu}\colon A\to A\times H$,\index{zn@$\nu$!znc@$\jmath_{\nu}$|dotfillboldidx}
be the maps defined by $\gamma(h)\coloneqq 1_A\times h$, $\nu(\lambda)\coloneqq \lambda 1_A\times 1$ and $\jmath_{\nu}(a)\coloneqq a\times 1$, respectively. Given a morphism $f\colon H\ot_k H\to A$\index{fb@$f$|dotfillboldidx}, we define $\mathcal{F}_f\colon H\ot_k H\longrightarrow A\ot_k H\index{fb@$f$!fc@$\mathcal{F}_f$|dotfillboldidx}$ by $\mathcal{F}_f(h\ot_k l) \coloneqq  f(h^{(1)}\ot_k l^{(1)})\ot_k h^{(2)}l^{(2)}$. Assume that
$$
f(h\ot_k l) = f(h^{(1)}\ot_k l^{(1)})(h^{(2)}l^{(2)}\xcdot 1_A)\quad\text{for all $h,l\in H$.}
$$
By \cite{GGV1}*{Proposition~2.4} we know that $(A,H,\chi_{\rho},\mathcal{F}_{\hs f})$ is a crossed product system (see \cite{GGV1}*{Def\-i\-nition~1.7}). We say that $f$ satisfis the {\em twisted module condition} if
$$
f(h^{(1)}\ot_k l^{(1)})(h^{(2)}l^{(2)}\xcdot a)=(h^{(1)}\xcdot (l^{(1)}\xcdot a)) f(h^{(2)}\ot_k l^{(2)})\quad\text{for all $h,l\in H$ and $a\in A$,}
$$
and that $f$ {\em is a cocycle} if
$$
f(h^{(1)}\ot_k l^{(1)})f(h^{(2)}l^{(2)}\ot_k m) = (h^{(1)}\xcdot f(l^{(1)}\ot_k m^{(1)})) f(h^{(2)}\ot_k l^{(2)}m^{(2)})\quad\text{for all $h,l,m\in H$.}
$$
Let $E$ be $A\times H$ endowed with the multiplication map $\mu_E$ introduced in \cite{GGV1}*{Notation~1.9}. A direct computation using the twisted module condition shows that $(a\times h)(b\times l) = a(h^{(1)}\cdot b) f(h^{(2)}\ot_k l^{(1)})\times h^{(3)}l^{(2)}$.

\begin{theorem}\label{weak crossed prod} Assume that
\begin{enumerate}[itemsep=0.7ex,  topsep=1.0ex, label=\emph{(\arabic*)}]

\item $f(h\ot_k l) = f(h^{(1)}\ot_k l^{(1)})(h^{(2)}l^{(2)}\xcdot 1_A)$ for all $h,l\in H$,

\item $h\xcdot 1_A = (h^{(1)}\xcdot (1^{(1)}\cdot 1_A))f(h^{(2)} \ot_k 1^{(2)})$ for all $h\in H$,

\item $h\xcdot 1_A = (1^{(1)}\cdot 1_A)f(1^{(2)}\ot_k h)$ for all $h\in H$,

\item $a\times 1 = 1^{(1)}\xcdot a\ot_k 1^{(2)}$ for all $a\in A$,

\item $f$ is a cocycle that satisfies the twisted module condition.

\end{enumerate}
Then,
\begin{enumerate}[itemsep=0.7ex,  topsep=1.0ex, label=\emph{(\alph*)}]

\item $(A,H,\chi_{\rho},\mathcal{F}_{\hs f},\nu)$ is a crossed product system with preunit (see \cite{GGV1}*{Definition~1.11}),

\item $\mathcal{F}_f$ is a cocycle that satisfies the twisted module condition (see \cite{GGV1}*{Definitions~1.8}),

\item $\mu_E$ is left and right $A$-linear, associative and has unit $1_A\times 1$,

\item The morphism $\jmath_{\nu}\colon A\to E$ is left and right $A$-linear, multiplicative and unitary,

\item $\jmath_{\nu}(a)\bx = a\cdot \bx$ and $\bx\jmath_{\nu}(a) = \bx\cdot a$, for all $a\in A$ and $\bx\in E$,

\item $\chi_{\rho}(h\ot_k a) = \gamma(h)\jmath_{\nu}(a)$ and $\mathcal{F}_f(h\ot_k l) = \gamma(h)\gamma(l)$, for all $h,l\in H$ and $a\in A$.

\end{enumerate}

\end{theorem}

\begin{proof} By~\cite{GGV1}*{Propositions~2.4, 2.6 and~2.10, and Theorems 1.12(7) and~2.11}.
\end{proof}

In the rest of this subsection we assume that the hypotheses of Theorem~\ref{weak crossed prod} are fulfilled, and we say that $E$ is {\em the unitary crossed product of $A$ with $H$ associated with $\rho$ and $f$}. Note that by item~(10) of that theorem we have
\begin{equation}\label{equacion1}
\jmath_{\nu}(a)\gamma(h) = a\cdot \nabla_{\!\rho}(1_A\ot_k h) = \nabla_{\!\rho}(a\ot_k h) = a\times h.
\end{equation}
Moreover, by~\cite{GGV2}*{Equalities~(1.6) and~(1.7)} and the definitions of $\chi_{\rho}$ and $\mathcal{F}_f$, we have
\begin{equation}\label{gama iota y gama gama}
\gamma(h)\jmath_{\nu}(a) = \jmath_{\nu}(h^{(1)}\xcdot a) \gamma(h^{(2)}) \quad\text{and}\quad \gamma(h)\gamma(l) = \jmath_{\nu}\bigl(f(h^{(1)}\ot_k l^{(1)})\bigr) \gamma(h^{(2)}l^{(2)}).
\end{equation}

\subsubsection[The weak comodule structure of $E$]{The weak comodule structure of $\bm{E}$}
Let $B$ be a right $H$-comodule. The tensor product $B\ot_k B$ is a (not necessarily counitary) right $H$-comodule, via $\delta_{B\ot_k B}(b\ot_k c)\coloneqq b^{(0)}\ot_k c^{(0)}\ot_k b^{(1)}c^{(1)}$.

\begin{definition} A unitary associative algebra $B$, which is also a counitary right $H$-comodule, is a {\em right $H$-comodule algebra} if $\mu_B$ is right $H$-colinear and one of the following equivalent conditions is satisfied:
\begin{enumerate}[itemsep=1.0ex, topsep=1.0ex, label=(\arabic*)]

\item $1_B^{(0)}\ot_k 1_B^{(1)}\ot_k 1_B^{(2)} = 1_B^{(0)}\ot_k 1_B^{(1)}1^{(1)}\ot_k 1^{(2)}$.

\item $1_B^{(0)}\ot_k 1_B^{(1)}\ot_k 1_B^{(2)} = 1_B^{(0)}\ot_k 1^{(1)}1_B^{(1)}\ot_k 1^{(2)}$.

\item $b^{(0)}\ot_k \ov{\Pi}^{\hs R}(b^{(1)}) = b1_B^{(0)}\ot_k 1_B^{(1)}$ for all $b\in B$.

\item $b^{(0)}\ot_k \Pi^L(b^{(1)}) = 1_B^{(0)}b\ot_k 1_B^{(1)}$ for all $b\in B$.

\item $1_B^{(0)}\ot_k \ov{\Pi}^{\hs R}(1_B^{(1)}) = 1_B^{(0)}\ot_k 1_B^{(1)}$.

\item $1_B^{(0)}\ot_k \Pi^L(1_B^{(1)}) = 1_B^{(0)}\ot_k 1_B^{(1)}$.

\end{enumerate}
\end{definition}

\begin{proposition}[Comodule algebra structure on $E$]\label{E es H-comodulo algebra debil} Assume that the hypotheses of Theorem~\ref{weak crossed prod} are satisfied. Then $E$ is a weak $H$-comodule algebra via the map $\delta_E\colon E\to E\ot_k H$ de\-fined by
$$
\delta_E\Bigl(\sum a_i\ot_k h_i\Bigr)\coloneqq \sum \nabla_{\!\rho}\bigl(a_i\ot_k h_i^{(1)}\bigr)\ot_k h_i^{(2)}
= \sum (a_i\times h_i^{(1)}) \ot_k h_i^{(2)}.\index{zbbb@$\delta_E$|dotfillboldidx}
$$
\end{proposition}

\begin{proof} See~\cite{GGV1}*{Proposition~2.27}.
\end{proof}

\subsection{Weak crossed products of weak module algebras}\label{subsection: Weak crossed products of weak module algebras by weak bialgebras in which the unit cocommutes}
Let $H$ be a weak bialgebra, $A$ a $k$-algebra and $\rho\colon H\ot_k A\to A$ a map. In this subsection we study the weak crossed products of $A$ with $H$ in which $A$ is a left weak $H$-module algebra.

\begin{definition}\label{def modulo algebra debil} We say that $A$ is a {\em left weak $H$-module algebra} via $\rho$, if
\begin{enumerate}[itemsep=0.7ex, topsep=1.0ex, label=(\arabic*)]

\item $1\xcdot a = a$ for all $a\in A$,

\item $h\xcdot (aa') = (h^{(1)}\xcdot a)(h^{(2)}\xcdot a')$ for all $h\in H$ and $a,a'\in A$,

\item $h\xcdot (l\xcdot 1_A) = (hl)\xcdot 1_A$ for all $h,l\in H$ and $a\in A$,

\end{enumerate}
In this case we say that $\rho$ is a weak left action of $H$ on $A$. If we also have
\begin{equation}\label{es modulo algebra}
h\xcdot (l\xcdot a) = (hl)\xcdot a\qquad\text{for all $h,l\in H$ and all $a\in A$},
\end{equation}
then we say that $A$ is a {\em left $H$-module algebra}.
\end{definition}

\begin{proposition}\label{modulo algebra debil} For each weak $H$-module algebra $A$ the following assertions hold:

\begin{enumerate}[itemsep=0.7ex, topsep=1.0ex, label=\emph{(\arabic*)}]

\item $\Pi^L(h)\xcdot a = (h\xcdot 1_A)a$ for all $h\in H$ and $a\in A$.

\item $\ov{\Pi}^L(h)\xcdot a = a(h\xcdot 1_A)$ for all $h\in H$ and $a\in A$.

\item $\Pi^L(h)\xcdot 1_A = h\xcdot 1_A$ for all $h\in H$.

\item $\ov{\Pi}^L(h)\xcdot 1_A = h\xcdot 1_A$ for all $h\in H$.

\item $h\xcdot (l\xcdot 1_A) = (h^{(1)}\xcdot 1_A)\epsilon(h^{(2)}l)$ for all $h,l\in H$.

\item $h\xcdot (l\xcdot 1_A) = (h^{(2)}\xcdot 1_A)\epsilon(h^{(1)}l)$ for all $h,l\in H$.

\end{enumerate}

\end{proposition}

\begin{proof} In \cite{CDG} it was proven that these items are equivalent and the proof of item~(3) it was kindly communicated to us by Jos\'e Nicanor Alonso \'Alvarez and Ram\'on Gonz\'alez Rodr\'iguez (see \cite{GGV1}*{Proposition~4.2}).
\end{proof}

\begin{example} The algebra $H^{\!L}$ is a left $H$-module algebra via $h\xcdot l\coloneqq \Pi^L(hl)$. In fact item~(1) of Definition~\ref{def modulo algebra debil} is trivial. By~\cite{BNS1}*{equality~(2.7a)} we have
\begin{equation}\label{aux1}
h^{(1)}l^{(1)}\ot_k h^{(2)}l^{(2)}=h^{(1)}l\ot_k h^{(2)}\quad\text{for all $h\in H$ and $l\in H^{\!L}$} .
\end{equation}
Using this, equality~\cite{BNS1}*{equality~(2.5a)} and the fact that $\Pi^L*\ide = \ide$, we obtain
$$
\Pi^L(h^{(1)}l)\Pi^L(h^{(2)}m)\! =\! \Pi^L\bigl(\Pi^L(h^{(1)}l^{(1)})\Pi^L(h^{(2)}l^{(1)}m)\bigr)\! =\! \Pi^L\bigr(\Pi^L(h^{(1)}l^{(1)})h^{(2)}l^{(1)}m)\bigr)\! =\! \Pi^L(hlm),
$$
$h\in H$ and $l,m\in H^{\!L}$. This proves item~(2). Finally equality~\eqref{es modulo algebra} follows from~\cite{BNS1}*{equality~(2.5a)}. This example is known as the trivial representation.
\end{example}

\begin{example}\label{cociclo trivial} Let $A$ be a left $H$-module algebra. The map $f(h\ot_k l)\coloneqq hl\xcdot 1_A$, named the {\em triv\-ial cocycle of $A$}, satisfies the hypotheses of Theorem~\ref{weak crossed prod}.
\end{example}

\begin{definition}\label{estable bajo rho} A subalgebra $K$ of $A$ is {\em stable under $\rho$} if $h\xcdot \lambda\in K$ for all $h\in H$ and $\lambda\in K$.
\end{definition}

\begin{remark}\label{estable bajo rho = estable bajo chi} A subalgebra $K$ of $A$ is stable under $\rho$ if and only if $\chi_{\rho}(H\ot_k K)\subseteq K\ot_k H$.
\end{remark}

\begin{example}\label{ejemlo subalgebra estable} Let $K\coloneqq \{h\xcdot 1_A: h\in H\}$. By Definition~\ref{def modulo algebra debil} and Proposition~\ref{modulo algebra debil}, we know that $K$ is a subalgebra of $A$, which is stable under $\rho$, and $K = \{h\xcdot 1_A: h\in H^{\!L}\} = \{h\xcdot 1_A: h\in H^{\!R}\}$. Moreover, the map $\pi^L\colon H^{\!L}\to K$, defined by $\pi^L(h)\coloneqq h\xcdot 1_A$, is a surjective morphism of algebras, and the map $\pi^R\colon H^{\!R}\to K$, defined by the same formula, is a surjective anti-mor\-phism of algebras. By \cite{BNS1}*{Proposition~2.11}, this implies that $K$ is separable.
\end{example}

\begin{remark}\label{incluido} Each subalgebra of $A$, which is stable under $\rho$, includes $\{h\xcdot 1_A: h\in H\}$.
\end{remark}

From here until the end of this subsection $A$ is a left weak $H$-module algebra and $E$ is the unitary crossed product of $A$ by $H$ associated with $\rho$ and a map $f\colon H\ot_k H\to A$. Thus, we assume that the hypotheses of Theorem~\ref{weak crossed prod} are fulfilled. In particular $\nu$, $\jmath_{\nu}$ and $\gamma$ are as at the beginning of Subsection~\ref{subsection: Crossed products by weak Hopf algebras}. By~\cite{BNS1}*{Propo\-sition~2.4} we know that $H^{\!L}H^{\!R}$ is a subalgebra $H$.

\begin{proposition}\label{fundamental'} The map $f$ factorizes throughout $H\ot_{H^{\!L}H^{\!R}} H$.
\end{proposition}

\begin{proof} It suffices to prove that $f(hl\ot_k m) = f(hl\ot_k m)$ for all $h,m\in H$ and $l\in H^{\!L}\cup H^{\!R}$. For $l\in H^{\!R}$ this follows from~\cite{GGV1}*{Proposition~2.7}. Suppose now that $l\in H^{\!L}$. By Definition~\ref{def modulo algebra debil}(3), the maps $u_2,v_2\colon H\ot_k H\to A$,~de\-fined by $u_2(h\ot_k l)\coloneqq hl\cdot 1_A$ and $v_2(h\ot_k l)\coloneqq h\cdot (l\cdot 1_A)$ coincide. Consequently $f(hl\ot_k m) = f(hl\ot_k m)$ by~\cite{GGV1}*{Proposition~2.8 and Remarks~2.16 and~2.17}.
\end{proof}

\subsubsection{Invertible cocycles and cleft extensions}
Let $u_2\colon H\ot_k H\to A$ be the map defined by $u_2(h\ot_k l) \coloneqq hl\xcdot 1_A$. We say that the cocycle $f$ is invertible if there exists a (unique) map $f^{-1}\colon H\ot_k H\to A$\index{fbz@$f^{-1}$|dotfillboldidx} such that $u_2*f^{-1} = f^{-1}*u_2 = f^{-1}$ and $f^{-1}*f = f*f^{-1} = u_2$.





\begin{remark}\label{neutro del otro lado} Condition~(1) in Theorem~\ref{weak crossed prod} says that $f*u_2 = f$. By \cite{GGV1}*{Remark~2.17} and the comment above \cite{GGV1}*{Proposition~4.4}, this implies that $u_2*f = f$.
\end{remark}

\begin{example}\label{trivial implica inversible} Assume that $A$ is a left $H$-module algebra and that $f$ is the trivial cocycle. By the previous remark and Definition~\ref{def modulo algebra debil}(2), the cocycle $f$ is invertible and $f^{-1} = f$.
\end{example}

\begin{definition} A map $g\colon H\ot_kH\to A$ is {\em normal} if $g(1\ot_k h) = g(h\ot_k 1) = h\xcdot 1_A$ for all~$h$.
\end{definition}

By \cite{GGV1}*{Proposition~2.17 and Remark~2.19} we know that the cocycle $f$ is normal if and only if the~equal\-ities in items~(2) and~(3) of Theorem~\ref{weak crossed prod} hold. When $f$ is invertible, we define $\gamma^{-1}\colon H\to E$\index{zc@$\gamma$!zc6@$\gamma^{-1}$|dotfillboldidx} by
\begin{equation}\label{gamma{-1}}
\gamma^{-1}(h)\coloneqq \jmath_{\nu}\bigl(f^{-1}\bigl(S(h^{(2)})\ot_k h^{(3)}\bigr)\bigr)\gamma\bigl(S(h^{(1)})\bigr).
\end{equation}
Assume for instance that $A$ is a left $H$-module algebra and that $f$ is the trivial cocycle. By Example~\ref{trivial implica inversible} and the fact that $S*\ide = \Pi^R$, we have
$$
\gamma^{-1}(h) = \jmath_{\nu}\bigl(S(h^{(2)})h^{(3)}\xcdot 1_A\bigr)\gamma\bigl(S(h^{(1)})\bigr) = \jmath_{\nu}\bigl(\Pi^R(h^{(2)}) \xcdot 1_A\bigr)\gamma\bigl(S(h^{(1)})\bigr).
$$
Using this, Proposition~\ref{modulo algebra debil}(3) and the fact that $\Pi^L\xcirc \Pi^R = \Pi^L\xcirc \ov{\Pi}^L\xcirc S = \Pi^L\xcirc S$, we obtain
$$
\gamma^{-1}(h) =\jmath_{\nu}\bigl(\Pi^L(S(h^{(2)}))\xcdot 1_A\bigr)\gamma\bigl(S(h^{(1)})\bigr) =\gamma\bigl(\Pi^L(S(h^{(2)})) S(h^{(1)})\bigr) =\gamma\bigl(S(h)\bigr),
$$
where the second equality holds by~\cite{GGV1}*{Propositions~2.22 and~4.5}; and the last one, since $S$ is anticomultiplicative and $\Pi^L*\ide = \ide$.

\subsection{Mixed complexes}\label{subsection: Mixed complexes}
In this subsection we recall briefly the notion of mixed complex. For more details about this concept we refer to~\cite{B} and~\cite{K}.

\smallskip

A {\em mixed complex} $\mathcal{X}\coloneqq  (X,b,B)$\index{xz@$\mathcal{X}$|dotfillboldidx} is a graded $k$-module $(X_n)_{n\ge 0}$, endowed with morphisms
$$
b\colon X_n\longrightarrow X_{n-1}\qquad\text{and}\qquad B\colon X_n\longrightarrow X_{n+1},
\index{ba@$b$|dotfillboldidx}\index{bb@$B$|dotfillboldidx}
$$
such that $b\xcirc b = 0$, $B\xcirc B = 0$ and $B\xcirc b + b\xcirc B = 0$. A {\em morphism of mixed complexes} $g\colon (X,b,B)\longrightarrow (Y,d,D)$ is a family of maps $g\colon X_n\to Y_n$, such that $d\xcirc g = g\xcirc b$ and $D\xcirc g= g\xcirc B$. Let $u$ be a degree~$2$ variable. A mixed complex $\mathcal{X}\coloneqq  (X,b,B)$ determines a double complex
$$
\begin{tikzpicture}
\begin{scope}[yshift=-0.47cm,xshift=-6cm]
\draw (0.5,0.5) node {$\BP(\mathcal{X})=$};
\end{scope}
\begin{scope}
\matrix(BPcomplex) [matrix of math nodes,row sep=2.5em, text height=1.5ex, text
depth=0.25ex, column sep=2.5em, inner sep=0pt, minimum height=5mm,minimum width =9.5mm]
{&\vdots &\vdots &\vdots &\vdots\\
\cdots & X_3 u^{-1} & X_2 u^0 & X_1 u^{} & X_0 u^2\\
\cdots & X_2 u^{-1} & X_1 u^0 & X_0 u\\
\cdots & X_1 u^{-1} & X_0 u^0\\
\cdots & X_0 u^{-1}\\};
\draw[->] (BPcomplex-1-2) -- node[right=1pt,font=\scriptsize] {$b$} (BPcomplex-2-2);
\draw[->] (BPcomplex-1-3) -- node[right=1pt,font=\scriptsize] {$b$} (BPcomplex-2-3);
\draw[->] (BPcomplex-1-4) -- node[right=1pt, font=\scriptsize] {$b$} (BPcomplex-2-4);
\draw[->] (BPcomplex-1-5) -- node[right=1pt, font=\scriptsize] {$b$} (BPcomplex-2-5);
\draw[<-] (BPcomplex-2-1) -- node[above=1pt,font=\scriptsize] {$B$} (BPcomplex-2-2);
\draw[<-] (BPcomplex-2-2) -- node[above=1pt,font=\scriptsize] {$B$} (BPcomplex-2-3);
\draw[<-] (BPcomplex-2-3) -- node[above=1pt,font=\scriptsize] {$B$} (BPcomplex-2-4);
\draw[<-] (BPcomplex-2-4) -- node[above=1pt,font=\scriptsize] {$B$} (BPcomplex-2-5);
\draw[->] (BPcomplex-2-2) -- node[right=1pt,font=\scriptsize] {$b$} (BPcomplex-3-2);
\draw[->] (BPcomplex-2-3) -- node[right=1pt,font=\scriptsize] {$b$} (BPcomplex-3-3);
\draw[->] (BPcomplex-2-4) -- node[right=1pt, font=\scriptsize] {$b$} (BPcomplex-3-4);
\draw[<-] (BPcomplex-3-1) -- node[above=1pt,font=\scriptsize] {$B$} (BPcomplex-3-2);
\draw[<-] (BPcomplex-3-2) -- node[above=1pt,font=\scriptsize] {$B$} (BPcomplex-3-3);
\draw[<-] (BPcomplex-3-3) -- node[above=1pt,font=\scriptsize] {$B$} (BPcomplex-3-4);
\draw[->] (BPcomplex-3-2) -- node[right=1pt,font=\scriptsize] {$b$} (BPcomplex-4-2);
\draw[->] (BPcomplex-3-3) -- node[right=1pt,font=\scriptsize] {$b$} (BPcomplex-4-3);
\draw[<-] (BPcomplex-4-1) -- node[above=1pt,font=\scriptsize] {$B$} (BPcomplex-4-2);
\draw[<-] (BPcomplex-4-2) -- node[above=1pt,font=\scriptsize] {$B$} (BPcomplex-4-3);
\draw[->] (BPcomplex-4-2) -- node[right=1pt,font=\scriptsize] {$b$} (BPcomplex-5-2);
\draw[<-] (BPcomplex-5-1) -- node[above=1pt,font=\scriptsize] {$B$} (BPcomplex-5-2);
\end{scope}
\begin{scope}[yshift=-0.47cm,xshift=4cm]
\draw (0.5,0.5) node {,};
\end{scope}
\end{tikzpicture}
\index{xz@$\mathcal{X}$!bp@$\BP(\mathcal{X})$|dotfillboldidx}
$$
where $b(\bx u^i)\coloneqq  b(\bx)u^i$ and $B(\bx u^i)\coloneqq  B(\bx)u^{i-1}$. By deleting the positively numbered columns we obtain a subcomplex $\BN(\mathcal{X})$\index{xz@$\mathcal{X}$!bn@$\BN(\mathcal{X})$|dotfillboldidx} of $\BP(\mathcal{X})$. Let $\BN'(\mathcal{X})$\index{xz@$\mathcal{X}$!bn'@$\BN'(\mathcal{X})$|dotfillboldidx} be the kernel of the canonical surjection from $\BN(\mathcal{X})$ to $(X,b)$. The quotient double complex $\BP(\mathcal{X})/\BN'(\mathcal{X})$ is denoted by $\BC(\mathcal{X})$\index{xz@$\mathcal{X}$!bc@$\BC(\mathcal{X})$|dotfillboldidx}. The homology groups $\HC_*(\mathcal{X})$,\index{xz@$\mathcal{X}$!hc@$\HC_*(\mathcal{X})$|dotfillboldidx} $\HN_*(\mathcal{X})$\index{xz@$\mathcal{X}$!hn@$\HN_*(\mathcal{X})$|dotfillboldidx} and $\HP_*(\mathcal{X})$,\index{xz@$\mathcal{X}$!hb@$\HP_*(\mathcal{X})$|dotfillboldidx} of the total complexes of $\BC(\mathcal{X})$, $\BN(\mathcal{X})$ and $\BP(\mathcal{X})$, respectively, are called the {\em cyclic}, {\em negative} and {\em periodic homology groups} of $\mathcal{X}$. The homology $\HH_*(\mathcal{X})$,\index{xz@$\mathcal{X}$!hh@$\HH_*(\mathcal{X})$|dotfillboldidx} of $(X,b)$, is called the {\em Hochschild homology} of $\mathcal{X}$. Finally, it is clear that a morphism $f\colon\mathcal{X}\to\mathcal{Y}$ of mixed complexes induces a morphism from the double complex $\BP(\mathcal{X})$ to the double complex $\BP(\mathcal{Y})$.

\smallskip

\begin{notation}\label{tensor circular} Let $C$ be an algebra. If $K$ is a subalgebra of $C$ we will say that $C$ is a $K$-algebra. Given a $K$-bi\-module $M$, we let $M\ot_{\hs K}$\index{ma@$M\ot$|dotfillboldidx} denote the quotient $M/[M,K]$, where $[M,K]$ is the $k$-submodule of $M$ generated by all the commutators $m\lambda -\lambda m$, with $m\in M$ and $\lambda\in K$. Moreover, for $m\in M$, we let $[m]$\index{mb@$[m]$|dotfillboldidx} denote the class of $m$ in $M\ot_{\hs K}$. \end{notation}

By definition, the {\em normalized mixed complex of the $K$-algebra $C$} is $(C\ot_{\hs K}\ov{C}^{\ot_{\hs K}^*}\ot_{\hs K},b_*,B_*)$, where $\ov{C}\coloneqq C/K$, $b_*$ is the canonical Hochschild boundary map and the Connes operator $B_*$ is given by
$$
B\bigl([c_0\ot_{\hs K}\cdots \ot_{\hs K} c_r]\bigr)\coloneqq \sum_{i=0}^r (-1)^{ir} [1\ot_{\hs K} c_i\ot_{\hs K}\cdots \ot_{\hs K} c_r\ot_{\hs K} c_0\ot_{\hs K} c_1\ot_{\hs K} \cdots \ot_{\hs K} c_{i-1}].
$$
The {\em cyclic}, {\em negative}, {\em periodic} and {\em Hochschild homology groups} $\HC^K_*(C)$,\index{hcc@$\HC^K_*(C)$|dotfillboldidx} $\HN^K_*(C)$\index{hd@$\HN^K_*(C)$|dotfillboldidx}, $\HP^K_*(C)$\index{he@$\HP^K_*(C)$|dotfillboldidx} and $\HH^K_*(C)$\index{hf$@$\HH^K_*(C)$|dotfillboldidx} of $C$ are the respective homology groups of $(C\ot_{\hs K}\ov{C}^{\ot_{\hs K}^*}\ot_{\hs K},b_*,B_*)$.

\section{Hochschild homology of cleft extensions}\label{Hochschild homology of cleft extensions}
Let $H$ be a weak Hopf algebra, $A$ an algebra, $\rho\colon H\to A$ a weak left action and $f\colon H\ot_kH\to A$ a li\-near map. Let $\chi_{\rho}$, $\gamma$, $\nu$, $\jmath_{\nu}$ and $\mathcal{F}_f$ be as at the beginning of Subsection~\ref{subsection: Crossed products by weak Hopf algebras}. Assume that the hypotheses~of~The\-orem~\ref{weak crossed prod} are fulfilled. Let $E$ be the crossed product associated with $\rho$ and $f$ and let $K$ be a subalgebra of $A$, which is stable under $\rho$. For instance we can take~$K$ as the minimal subalgebra of $A$ that is stable under $\rho$ (see Ex\-am\-ple~\ref{ejemlo subalgebra estable}). By Theorem~\ref{weak crossed prod}, the tuple $(A,\chi_{\rho},\mathcal{F}_f,\nu)$ is a crossed product system with preunit, $\mathcal{F}_f$ is a cocycle that satisfies the twisted module condi\-tion and $E$ is an associative algebra with unit $1_E\coloneqq 1_A\times 1$. Moreover, by Remark~\ref{incluido} we know that $1_E\in K\ot_k H$. So, the hypotheses of \cite{GGV2}*{Section~3} are satisfied. For the sake of simplicity in the sequel we will write $\ot$ instead of~$\ot_K$. Let $M$ be an $E$-bi\-mod\-ule. By definition the Hochschild homology $\Ho^{\hs K}_*(E,M)\index{hg@$\Ho^{\hs K}_*(E,M)$|dotfillboldidx}$, of the $K$-algebra $E$ with coefficients in $M$, is the homology of the normalized Hochschild chain complex $\cramped{\bigl(M\ot\ov{E}^{\ot^*}\ot,b_*\bigr)}$, where $b_*$ is the canonical Hochschild boundary map. In \cite{GGV2}*{Section~3} a chain complex $(\wh{X}_*(M),\wh{d}_*)$
was obtained, simpler than the canon\-ical one, that gives the Hochschild homology $\Ho^{\hs K}_*(E,M)\index{hg@$\Ho^{\hs K}_*(E,M)$|dotfillboldidx}$, of the $K$-algebra $E$ with coefficients in $M$. From now on we assume that the cocycle $f$ is convolution invertible. In this Section we prove that $(\wh{X}_*(M),\wh{d}_*)$ is isomorphic to a simpler complex $(\ov{X}_*(M),\ov{d}_*)$. If $K$ is separable (for instance when $K\coloneqq H\xcdot 1_A$), then $(\ov{X}_*(M),\ov{d}_*)$ gives the absolute Hochschild homology of $E$ with coefficients in $M$. Recall that $M$ is an $A$-bi\-module via the map $\jmath_{\nu}\colon A\to E$. Let $\ov{A}\coloneqq A/K$\index{ak@$\ov{A}$|dotfillboldidx}, $\ov{E}\coloneqq E/\jmath_{\nu}(K)$\index{ec@$\ov{E}$|dotfillboldidx} and $\wt{E}\coloneqq E/\jmath_{\nu}(A)$.\index{ed@$\wt{E}$|dotfillboldidx} We recall from \cite{GGV2}*{Section~3} that
$$
\wh{X}_n(M) = \bigoplus_{\substack{r,s\ge 0\\ r+s = n}} \wh{X}_{rs}(M),\qquad\text{where $\wh{X}_{rs}(M)\coloneqq  M\ot_{\hs A}\wt{E}^{\ot_{\hs A}^s}\ot\ov{A}^{\ot^r}\ot,$}\index{xxc@$\wh{X}_n(M)$|dotfillboldidx}\index{xxa@$\wh{X}_{rs}(M)$|dotfillboldidx}
$$
and that there exist maps $\wh{d}^l_{rs}\colon \wh{X}_{rs}(M)\to \wh{X}_{r-l,s+l-1}(M)$\index{dd@$\wh{d}^l_{rs}$|dotfillboldidx} such that
$$
\wh{d}_n(\bx)\coloneqq \begin{cases} \displaystyle{\sum_{l=1}^n \wh{d}^l_{0n}(\bx)} &\text{if $\bx\in \wh{X}_{0n}$,}\\ \displaystyle{\sum^{n-r}_{l=0} \wh{d}^l_{r,n-r}(\bx)} &\text{if $\bx\in \wh{X}_{r,n-r}$ with $r>0$.}\end{cases}\index{de@$\wh{d}_n$|dotfillboldidx}
$$

\begin{notations}\label{not 7.1} We will use the following notations:

\begin{enumerate}[itemsep=0.7ex, topsep=1.0ex, label=(\arabic*)]

\item For each $a\in A$ we let $\ov{a}$\index{aka@$\ov{a}$|dotfillboldidx} denote its class in $\ov{A}$.  For each $x\in E$, we let $\ov{x}\index{eg@$\ov{x}$|dotfillboldidx}$ and $\wt{x}\index{eh@$\wt{x}$|dotfillboldidx}$ denote its class in $\ov{E}$ and $\wt{E}$, respectively.

\item Given $a_1,\dots, a_r\in A$ and $1\le i\le j\le r$, we set $\ba_{ij}\coloneqq  a_i\ot\cdots\ot a_j\index{as@$\ba_{ij}$|dotfillboldidx}$, and we let $\ov{\ba}_{ij}\index{at@$\ov{\ba}_{ij}$|dotfillboldidx}$ denote the class of $\ba_{ij}$ in $\ov{A}^{\ot^{j-i+1}}$.

\item Given $h_1,\dots, h_s\in H$ we set $\mathfrak{h}_{1s}\coloneqq h_1\cdots h_s$\index{hii@$\mathfrak{h}_{1s}$|dotfillboldidx} and $\brh_{1s}\coloneqq h_1\ot_k\cdots \ot_{\hs k} h_s$.\index{hj@$\brh_{1s}$|dotfillboldidx}

\item Given $h_1,\dots,h_s\in H$ we set
\begin{align*}
&\brh_{1s}^{(1)}\ot_k \brh_{1s}^{(2)} \coloneqq \bigl(h_1^{(1)}\ot_k \cdots \ot_k h_s^{(1)}\bigr) \ot_k \bigl(h_1^{(2)}\ot_k \cdots \ot_k h_s^{(2)}\bigr)\index{hja@$\brh_{1s}^{(1)}\ot_k \brh_{1s}^{(2)}$|dotfillboldidx}\\
\shortintertext{and}
&\ov{\bh}_{1s}^{(1)}\ot_k \ov{\bh}_{1s}^{(2)}\coloneqq \bigl(\ov{h}_1^{(1)}\ot_{H^{\!L}} \cdots \ot_{H^{\!L}} \ov{h}_s^{(1)}\bigr) \ot_k \bigl(\ov{h}_1^{(2)}\ot_{H^{\!L}} \cdots \ot_{H^{\!L}} \ov{h}_s^{(2)}\bigr).\index{hjaa@$\ov{\bh}_{1s}^{(1)}\ot_k \ov{\bh}_{1s}^{(2)}$|dotfillboldidx}
\end{align*}

\item Given $h\in H$ and $a_1,\dots,a_r\in A$ we set $h\xcdot \ov{\ba}_{1r}\coloneqq \ov{h^{(1)}\xcdot a_1}\ot\cdots\ot \ov{h^{(r)}\xcdot a_r}$.\index{hlla@$h\xcdot \ov{\ba}_{1r}$|dotfillboldidx}

\item We let $\ov{\gamma}\colon H\to\ov{E}\index{zc@$\gamma$!zc1@$\ov{\gamma}$|dotfillboldidx}$ and $\wt{\gamma}\colon H\to \wt{E}\index{zc@$\gamma$!zc2@$\wt{\gamma}$|dotfillboldidx}$ denote the maps induced by $\gamma$.

\item Given $h_1,\dots, h_s\in H$ we set $\wt{\gamma}_{\hs A}(\brh_{1s})\coloneqq \wt{\gamma}(h_1)\ot_{\hs A}\cdots \ot_{\hs A}\wt{\gamma}(h_s)$.\index{zc@$\gamma$!zc64@$\wt{\gamma}_{\hs A}(\brh_{1s})$|dotfillboldidx}

\item Given $h_1,\dots, h_s\in H$ we set $\gamma_{_{\!\times}}(\brh_{1s}) \coloneqq \gamma(h_1)\cdots \gamma(h_s)\index{zc@$\gamma$!zc65@$\gamma_{{}_{\times}}(\brh_{1s})$|dotfillboldidx}$.

\item Given $h_1,\dots, h_s\in H$ we set $\gamma_{_{\!\times}}^{-1}(\brh_{1s})\coloneqq \gamma^{-1}(h_s)\cdots \gamma^{-1}(h_1), \index{zc@$\gamma$!zc7@$\gamma_{{}_{\times}}^{-1}(\brh_{1s})$|dotfillboldidx}$ where $\gamma^{-1}$ is as in~\eqref{gamma{-1}}.

\end{enumerate}

\end{notations}

\subsection{Technical results}

\begin{lemma}\label{propiedad 4} Let $a,a'\in A$ and $h\in H$. The following equalities hold:

\begin{enumerate}[itemsep=0.7ex, topsep=1.0ex, label=\emph{(\arabic*)}]



\item $\delta_E\bigl(\jmath_{\nu}(a)\gamma(l)\jmath_{\nu}(a')\gamma(h)\bigr) = \jmath_{\nu}(a)\gamma(l) \jmath_{\nu}(a') \gamma(h^{(1)})\ot_k h^{(2)}$, for all $l\in H^{\!L}$.

\item $\delta_E\bigl(\jmath_{\nu}(a')\gamma(h)\jmath_{\nu}(a)\gamma(l)\bigr) = \jmath_{\nu}(a')\gamma(h^{(1)})\jmath_{\nu}(a)\gamma(l) \ot_k h^{(2)}$, for all $l\in H^{\!L}$.

\item $\delta_E\bigl(\jmath_{\nu}(A)\gamma(H^{\!L})\bigr)\subseteq E\ot_k H^{\!L}$.

\item $\jmath_{\nu}(a)\gamma^{-1}(h) = \gamma^{-1}(h^{(1)})\jmath_{\nu}(h^{(2)}\xcdot a)$.

\end{enumerate}

\end{lemma}

\begin{proof} 1)\enspace We have
\begin{align*}
\delta_E\bigl(\jmath_{\nu}(a)\gamma(l)\jmath_{\nu}(a')\gamma(h)\bigr) & = \delta_E\bigl(\jmath_{\nu}(a(l^{(1)}\cdot a'))\gamma(l^{(2)})\gamma(h)\bigr)\\
& = \delta_E\bigl(\jmath_{\nu}(a(l^{(1)}\cdot a'))\gamma(l^{(2)}h)\bigr)\\
& = \jmath_{\nu}(a(l^{(1)}\cdot a'))\gamma(l^{(2)}h^{(1)})\ot_k l^{(3)}h^{(2)} \\
& = \jmath_{\nu}(a(l^{(1)}\cdot a'))\gamma(l^{(2)}h^{(1)})\ot_k h^{(2)} \\
& = \jmath_{\nu}(a(l^{(1)}\cdot a'))\gamma(l^{(2)})\gamma(h^{(1)})\ot_k h^{(2)} \\
& = \jmath_{\nu}(a)\gamma(l)\jmath_{\nu}(a')\gamma(h^{(1)})\ot_k h^{(2)},
\end{align*}
where the first and the last equality hold by the first identity in~\eqref{gama iota y gama gama} and Theorem~\ref{weak crossed prod}(d); the second~one, by \cite{BNS1}*{(2.6a)} and~\cite{GGV1}*{Proposition~2.22}; the third one because, by equality~\eqref{equacion1}, the definition of $\delta_E$, and the fact that $\jmath_{\nu}(A)\gamma(H)\subseteq E$, for all $a\in A$ and $h\in H$, we have
\begin{equation}\label{calculo de coaccion}
\delta_E(\jmath_{\nu}(a)\gamma(h))= \nabla_{\!\rho}\bigl(a\times h^{(1)}\bigr)\ot_k h^{(2)}= \nabla_{\!\rho}\bigl(\jmath_{\nu}(a)\gamma(h^{(1)})\bigr)\ot_k h^{(2)}= \jmath_{\nu}(a)\gamma(h^{(1)})\ot_k h^{(2)};
\end{equation}
the fourth one, by~\cite{BNS1}*{(2.6a)} and the fact that by~\cite{BNS1}*{(2.4), (2.7a) and (2.10)}
\begin{equation}\label{aux2}
l^{(1)}h^{(1)}\ot_k l^{(2)}h^{(2)}= l h^{(1)}\ot_k h^{(2)}\quad\text{for all $h\in H$ and $l\in H^{\!L}$;}
\end{equation}
and the fifth one, by~\cite{GGV1}*{Proposition~2.22}.

\noindent 2)\enspace Mimic the proof of item~(1).

\smallskip

\noindent 3)\enspace This follows from equality~\eqref{calculo de coaccion} and~\cite{BNS1}*{(2.6a)}.

\smallskip

\noindent 4)\enspace This is \cite{GGV1}*{Proposition~5.22}.
\end{proof}

\begin{remark} Since $\jmath_{\nu}(1_A)=1_E$, from equality~\eqref{calculo de coaccion} it follows that $\gamma$ is $H$-colinear.
\end{remark}

\begin{lemma}\label{prop esp''} For all $h\in H$ and $l\in H^{\!L}$, we have $\gamma^{-1}(hl) = \gamma \bigl(S(l)\bigr) \gamma^{-1}(h)$ and $\gamma^{-1}(lh) = \gamma^{-1}(h)\gamma\bigl(S(l)\bigr)$.
\end{lemma}

\begin{proof} We prove the first equality and leave the second one, which is similar, to the reader. We have
\begin{align*}
\gamma^{-1}(hl) & = \jmath_{\nu}\bigl(f^{-1}(S(h^{(2)}l^{(2)})\ot_k h^{(3)}l^{(3)}\bigr)\gamma\bigl(S(h^{(1)}l^{(1)})\bigr)\\
& = \jmath_{\nu}\bigl(f^{-1}(S(h^{(2)})\ot_k h^{(3)}\bigr)\gamma\bigl(S(h^{(1)}l)\bigr)\\
& = \jmath_{\nu}\bigl(f^{-1}(S(h^{(2)})\ot_k h^{(3)}\bigr)\gamma\bigl(S(l)\bigr)\gamma\bigl(S(h^{(1)}\bigr)\\
& =\gamma\bigl(S(l)\bigr) \jmath_{\nu}\bigl(f^{-1}(S(h^{(2)})\ot_k h^{(3)}\bigr)\gamma\bigl(S(h^{(1)}\bigr)\\
& =\gamma\bigl(S(l)\bigr)\gamma^{-1}(h),
\end{align*}
where the first and last equality hold by the definition of $\gamma^{-1}$; the second one, by~\cite{BNS1}*{(2.6a)} and identity~\eqref{aux1}; the third one, by~\cite{GGV1}*{Proposition~2.22} and the fact that $S$ is antimultiplicative; and the fourth one, by~\cite{GGV1}*{Proposition~4.6}.
\end{proof}

\begin{lemma}\label{prop esp'} Let $h_1,\dots,h_s\in H$ and $a\in A$. The following equality holds:
$$
\gamma_{_{\!\times}}^{-1}(\brh_{1s}^{(1)})\ot_{\hs A} \wt{\gamma}_{\hs A}(\brh_{1s}^{(2)})\xcdot \jmath_{\nu}(a) = \jmath_{\nu}(a)\gamma_{_{\!\times}}^{-1}(\brh_{1s}^{(1)})\ot_{\hs A} \wt{\gamma}_{\hs A}(\brh_{1s}^{(2)}).
$$
\end{lemma}

\begin{proof} We proceed by induction on $s$. By the first identity in~\eqref{gama iota y gama gama} and Lemma~\ref{propiedad 4}(4), for $s = 1$ we have
$$
\gamma^{-1}(h_1^{(1)})\ot_{\hs A} \wt{\gamma}(h_1^{(2)})\jmath_{\nu}(a)  = \gamma^{-1}(h_1^{(1)})\jmath_{\nu}(h_1^{(2)}\xcdot a)\ot_{\hs A} \wt{\gamma}(h_1^{(3)}) = \jmath_{\nu}(a)\gamma^{-1}(h_1^{(1)})\ot_{\hs A} \wt{\gamma}(h_1^{(2)}).
$$
Assume $s>1$ and that the result is valid for $s-1$. Let $T\coloneqq \gamma_{_{\!\times}}^{-1}(\brh_{1s}^{(1)})\ot_{\hs A} \wt{\gamma}_{\hs A} (\brh_{1s}^{(2)})\xcdot \jmath_{\nu}(a)$. By the first identity in~\eqref{gama iota y gama gama},
$$
T\! =\! \gamma_{_{\!\times}}^{-1}(\brh_{1s}^{(1)})\ot_{\hs A} \wt{\gamma}_{\hs A}(\brh_{1,s-1}^{(2)}) \ot_{\hs A} \jmath_{\nu}(h_s^{(2)}\xcdot a) \wt{\gamma}(h_s^{(3)})\! =\! \gamma_{_{\!\times}}^{-1}(\brh_{1s}^{(1)})\ot_{\hs A} \wt{\gamma}_{\hs A}(\brh_{1,s-1}^{(2)})\xcdot \jmath_{\nu}(h_s^{(2)}\xcdot a) \ot_{\hs A} \wt{\gamma}(h_s^{(3)}).
$$
Consequently, by the inductive hypothesis and Lemma~\ref{propiedad 4}(4),
$$
T\! =\! \gamma^{-1}(h_s^{(1)})\jmath_{\nu}(h_s^{(2)}\xcdot a) \gamma_{_{\!\times}}^{-1}(\brh_{1,s-1}^{(1)}) \ot_{\hs A} \wt{\gamma}_{\hs A} (\brh_{1,s-1}^{(2)}) \ot_{\hs A}\wt{\gamma}(h_s^{(3)})\! =\!\jmath_{\nu}(a)\gamma_{_{\!\times}}^{-1}(\brh_{1s}^{(1)})\ot_{\hs A} \wt{\gamma}_{\hs A} (\brh_{1s}^{(2)}),
$$
as desired.
\end{proof}

\begin{lemma}\label{auxiliar 5} Let $s\ge 1$. for all $h_1,\dots, h_s\in H$, we have $\gamma_{_{\!\times}}(\brh_{1s}^{(1)}) \gamma_{_{\!\times}}^{-1}(\brh_{1s}^{(2)})\ot_{\hs A} \wt{\gamma}_{\hs A}(\brh_{1s}^{(3)}) = 1_E \ot_{\hs A} \wt{\gamma}_{\hs A}(\brh_{1s})$.
\end{lemma}

\begin{proof} We proceed by induction on $s$. Using the first equality in~\eqref{pepe} and arguing as in~\cite{BNS1}*{equality~(2.8b)}, we obtain that
\begin{equation}\label{auxiliar 4}
\Pi^L(h^{(1)})\ot_k h^{(2)} = S(\ov{\Pi}^L(h^{(1)}))\ot_k h^{(2)} = S(1^{(1)})\ot_k 1^{(2)}h\qquad\text{for all $h\in H$.}
\end{equation}
Using this, \cite{GGV1}*{Lemma~2.20 and Propositions~4.5 and~5.19}, we obtain
$$
\gamma(h_1^{(1)})\gamma^{-1}(h_1^{(2)})\ot_{\hs A} \wt{\gamma}(h_1^{(3)})\! =\! \gamma(\Pi^L(h_1^{(1)})) \ot_{\hs A} \wt{\gamma}(h_1^{(2)})\! =\! \gamma(S(1^{(1)})) \ot_{\hs A} \wt{\gamma}(1^{(2)}h_1)\! =\! 1_E \ot_{\hs A} \wt{\gamma}(h_1).
$$
This proves the case $s=1$. Assume that the result is true for $s$ and set $T\coloneqq \gamma_{_{\!\times}}(\brh_{1,s+1}^{(1)}) \gamma_{_{\!\times}}^{-1}(\brh_{1,s+1}^{(2)}) \ot_{\hs A} \wt{\gamma}_{\hs A}(\brh_{1,s+1}^{(3)})$. By equality~\eqref{auxiliar 4}, \cite{GGV1}*{Propositions~4.5 and~5.19},
$$
\gamma(h_{s+1}^{(1)})\gamma^{-1}(h_{s+1}^{(2)}) \ot_k h_{s+1}^{(3)} = \gamma(\Pi^L(h_{s+1}^{(1)})) \ot_k h_{s+1}^{(2)} = \jmath_{\nu}(\Pi^L(h_{s+1}^{(1)})\xcdot 1_A) \ot_k h_{s+1}^{(2)} = \jmath_{\nu}(S(1^{(1)})\xcdot 1_A) \ot_k 1^{(2)}h_{s+1}.
$$
Hence
$$
T = \gamma_{_{\!\times}}(\brh_{1s}^{(1)}) \jmath_{\nu}(S(1^{(1)})\xcdot 1_A)\gamma_{_{\!\times}}^{-1}(\brh_{1s}^{(2)}) \ot_{\hs A} \wt{\gamma}_{\hs A}(\brh_{1s}^{(3)}) \ot_{\hs A} \wt{\gamma}(1^{(2)}h_{s+1}).
$$
Using now the first identity in~\eqref{gama iota y gama gama} and Definition~\ref{def modulo algebra debil}(3) again and again, we obtain
$$
T = \jmath_{\nu}(\mathfrak{h}_{1s}^{(1)} S(1^{(1)})\xcdot 1_A) \gamma_{_{\!\times}}(\brh_{1s}^{(2)}) \gamma_{_{\!\times}}^{-1}(\brh_{1s}^{(3)}) \ot_{\hs A} \wt{\gamma}_{\hs A}(\brh_{1s}^{(4)}) \ot_{\hs A} \wt{\gamma}(1^{(2)}h_{s+1}).
$$
Consequently, by the inductive hypothesis,
$$
T = \jmath_{\nu}(\mathfrak{h}_{1s}^{(1)} S(1^{(1)})\xcdot 1_A)\ot_{\hs A} \wt{\gamma}_{\hs A}(\brh_{1s}^{(2)}) \ot_{\hs A} \wt{\gamma}(1^{(2)}h_{s+1}) = 1_E \ot_{\hs A} \wt{\gamma}_{\hs A}(\brh_{1s})\ot_{\hs A}\jmath_{\nu}(S(1^{(1)})\xcdot 1_A) \wt{\gamma}(1^{(2)}h_{s+1}),
$$
where the last equality follows using Definition~\ref{def modulo algebra debil}(3) and the first identity in~\eqref{gama iota y gama gama} again and again. Hence, by \cite{GGV1}*{Lemma~2.20 and~Proposition~4.5}, we have $T = 1_E \ot_{\hs A} \wt{\gamma}_{\hs A}(\brh_{1,s+1})$, as desired.
\end{proof}

\begin{lemma}\label{auxiliar 6} Let $s\ge 1$. For all $z\in H^{\!R}$ and $h_1,\dots, h_s\in H$, we have
$$
\gamma_{_{\!\times}}^{-1}(\mathrm{h}_{1s}^{(1)})\gamma(z)\gamma_{_{\!\times}}(\mathrm{h}_{1s}^{(2)})\ot_k \ov{\bh}_{1s}^{(3)} = \gamma(1^{(1)})\ot_k \ov{\Pi}^R(z)\xcdot \ov{\bh}_{1s}\xcdot 1^{(2)}.
$$
\end{lemma}

\begin{proof} Set $T\coloneqq \gamma_{_{\!\times}}^{-1}(\mathrm{h}_{1s}^{(1)})\gamma(z)\gamma_{_{\!\times}}(\mathrm{h}_{1s}^{(2)})\ot_k \ov{\bh}_{1s}^{(3)}$. We proceed by induction on $s$. Let $s=1$. By~\cite{BNS1}*{equal\-ity~(2.7b)}, \cite{GGV1}*{Propositions~2.22 and~5.19}, we have
$$
T = \gamma^{-1}(h_1^{(1)})\gamma(zh_1^{(2)})\ot_k \ov{h_1^{(3)}} = \gamma^{-1}(z^{(1)}h_1^{(1)})\gamma(z^{(2)}h_1^{(2)})\ot_k \ov{h_1^{(3)}} = \gamma\bigl(\Pi^R(zh_1^{(1)})\bigr)\ot_k \ov{h_1^{(2)}}.
$$
Consequently, by the definition of $\Pi^R$, \cite{BNS1}*{(2.6a)} and identities~\eqref{propiedad de epsilon} and~\eqref{aux1},
$$
T = \gamma(1^{(1)})\ot_k \epsilon(zh_1^{(1)}1^{(2)}) \ov{h_1^{(2)}} = \gamma(1^{(1)})\ot_k \epsilon(z1^{(1')})\epsilon(1^{(2')} h_1^{(1)}1^{(2)}) \ov{1^{(3')}h_1^{(2)}1^{(3)}}.
$$
So, by the definition of $\ov{\Pi}^R$,
$$
T = \gamma(1^{(1)})\ot_k \epsilon(z1^{(1')})\ov{1^{(2')} h_11^{(2)}}= \gamma(1^{(1)})\ot_k \ov{\ov{\Pi}^R(z)h_11^{(2)}}.
$$
Assume now that $s>1$ and the result holds by $s-1$. By the inductive hypothesis, we have
$$
T = \gamma^{-1}(h_s^{(1)})\gamma(1^{(1)})\gamma(h_s^{(2)})\ot_k \bigl(\ov{\Pi}^R(z)\xcdot \ov{\bh}_{1,s-1}\xcdot 1^{(2)}\ot_{\! H^{\!L}} \ov{h_s^{(3)}}\bigr).
$$
Thus, by the case $s=1$, \cite{BNS1}*{(2.4)} and~\cite{GGV1}*{Proposition~2.22}, we have
\begin{align*}
T & = \gamma^{-1}(h_s^{(1)})\gamma(1^{(1)}h_s^{(2)})\ot_k \bigl(\ov{\Pi}^R(z)\xcdot  \ov{\bh}_{1,s-1}\ot_{\! H^{\!L}} \ov{1^{(2)}h_s^{(3)}}\bigr)\\
& = \gamma^{-1}(h_s^{(1)})\gamma(h_s^{(2)})\ot_k \bigl(\ov{\Pi}^R(z)\xcdot \ov{\bh}_{1,s-1}\ot_{\! H^{\!L}} \ov{h_s^{(3)}}\bigr)\\
& = \gamma(1^{(1)})\ot_k \ov{\Pi}^R(z)\xcdot \ov{\bh}_{1s}\xcdot 1^{(2)},
\end{align*}
as desired.
\end{proof}

\subsection{Main results} Let $r,s\ge 0 $. By~\cite{GGV1}*{Propositions~2.22 and~4.6}, we know that $M\ot \ov{A}^{\ot^r}\ot$ is a left $H^{\!L}$-module via
\begin{equation}\label{acciones}
l \xcdot [m\ot \ov{\ba}_{1r}] \coloneqq [m\cdot \gamma(S(l))\ot\ov{\ba}_{1r}],
\end{equation}
where $[m\ot\ov{\ba}_{1r}]$ denotes the class of $m\ot\ov{\ba}_{1r}$ in $M\ot\ov{A}^{\ot^r}\ot$ (see Notation~\ref{tensor circular} and remember that $\ot$ stands for $\ot_K$). Write
$$
\ov{X}_{rs}(M)\coloneqq \ov{H}^{\ot_{\! H^{\!L}}^s} \ot_{\! H^{\!L}} \bigl(M\ot\ov{A}^{\ot^r}\ot\bigr)\index{xxb@$\ov{X}_{rs}(M)$|dotfillboldidx}.
$$
Since $\ov{H}^{\ot_{\! H^{\!L}}^0}  = H^{\!L}$ and $\ov{A}^{\ot^0} = K$, we have
\begin{equation}\label{ec10}
\ov{X}_{r0}(M)\simeq M\ot\ov{A}^{\ot^r}\ot\qquad\text{and}\qquad \ov{X}_{0s}(M) \simeq \ov{H}^{\ot_{\! H^{\!L}}^s} \ot_{\! H^{\!L}} \bigl(M\ot \bigr).
\end{equation}
Let
$\Theta'_{rs}\colon M\otimes_k E^{\otimes_k^s} \otimes_k \ov{A}^{\otimes^r} \longrightarrow \ov{X}_{rs}(M)$ and $\Lambda'_{rs}\colon H^{\otimes_k^s}\otimes_k M\otimes_k\ov{A}^{\ot^r}\longrightarrow \wh{X}_{rs}(M)$ be the maps defined by
\begin{align*}
&\Theta'(\bx) \coloneqq (-1)^{rs}\,\ov{\bh}_{1s}^{(2)} \ot_{\! H^{\!L}} \bigl[m\xcdot \jmath_{\nu}(a_1)\gamma(h_1^{(1)})\cdots \jmath_{\nu}(a_s)\gamma(h_s^{(1)})\ot \ov{\ba}_{s+1,s+r}\bigr]
\shortintertext{and}
&\Lambda'(\byy) \coloneqq (-1)^{rs} \bigl[m\xcdot\gamma_{_{\!\times}}^{-1}(\mathrm{h}_{1s}^{(1)})\ot_{\hs A} \wt{\gamma}_{\hs A}(\brh_{1s}^{(2)})\ot \ov{\ba}_{1r}\bigr] ,
\end{align*}
where $\bx\coloneqq m\ot_k \jmath_{\nu}(a_1)\gamma(h_1)\ot_k\cdots\ot_k\jmath_{\nu}(a_s)\gamma(h_s)\ot_k\ov{\ba}_{s+1,s+r}$ and $\byy\coloneqq \brh_{1s}\ot_k m\ot_k \ov{\ba}_{1r}$.

\begin{proposition}\label{const de aplicaciones} For each $r,s\ge 0$ the maps $\Theta'_{rs}$ and $\Lambda'_{rs}$ induce morphisms
$$
\Theta_{rs}\colon \wh{X}_{rs}(M) \longrightarrow \ov{X}_{rs}(M)\index{zz@$\Theta_{rs}$|dotfillboldidx}\qquad\text{and}\qquad \Lambda_{rs}\colon \ov{X}_{rs}(M)\longrightarrow \wh{X}_{rs}(M),\index{zmm@$\Lambda_{rs}$|dotfillboldidx}
$$
which are inverse one of each other.
\end{proposition}

\begin{proof} First we show that $\Theta_{rs}$ is well defined. Let $\bx$ be as in the definition of $\Theta'_{rs}$. We must prove that

\begin{enumerate}[itemsep=0.7ex, topsep=1.0ex, label= (\arabic*)]

\item If there exists $i$ such that $h_i\in H^{\!L}$, then $\Theta'(\bx) = 0$,

\item $\Theta'$ is $A$-balanced in the first $s$ tensors of $M\otimes_k E^{\otimes_k^s} \otimes_k A^{\otimes^r}$,

\item $\Theta'$ is $K$-balanced in the $(s+1)$-tensor of $M\otimes_k E^{\otimes_k^s} \otimes_k \ov{A}^{\otimes^r}$,

\item $\Theta'(\jmath_{\nu}(\lambda)\xcdot \bx) = \Theta'(\bx\xcdot \jmath_{\nu}(\lambda))$ for all $\lambda\in K$.

\end{enumerate}
Condition~(1) follows from Lemma~\ref{propiedad 4}(3). Next we prove that Condition~(2) is satisfied at the first tensor. Let $a\in A$. By Lemma~\ref{propiedad 4}(1), we have
\begin{align*}
& \Theta'(m\cdot \jmath_{\nu}(a)\ot_k \jmath_{\nu}(a_1)\gamma(h_1)\ot_k\cdots\ot_k \jmath_{\nu}(a_s)\gamma(h_s)\ot_k \ov{\ba}_{s+1,s+r})\\
& = (-1)^{rs}\,\ov{\bh}_{1s}^{(2)} \ot_{\! H^{\!L}} \bigl[m\xcdot \jmath_{\nu}(a)\jmath_{\nu}(a_1) \gamma(h_1^{(1)}) \cdots\jmath_{\nu}(a_s) \gamma(h_s^{(1)}) \ot \ov{\ba}_{s+1,s+r}\bigr]\\
& = \Theta'(m\ot_k  \jmath_{\nu}(a)\jmath_{\nu}(a_1)\gamma(h_1)\ot_k\cdots\ot_k \jmath_{\nu}(a_s)\gamma(h_s)\ot_k \ov{\ba}_{s+1,s+r}),
\end{align*}
A similar argument using items~(1) and~(2) of Lemma~\ref{propiedad 4} proves Condition~(2) at the $i$-th tensor with $2\le i\le s$. We now prove Condition~(3). Let $\lambda\in K$. By Lemma~\ref{propiedad 4}(2), we have
\begin{align*}
& \Theta'(m\ot_k \jmath_{\nu}(a_1)\gamma(h_1)\ot_k\cdots\ot_k\jmath_{\nu}(a_s)\gamma(h_s)\jmath_{\nu}(\lambda)\ot_k \ov{\ba}_{s+1,s+r})\\
& = (-1)^{rs}\,\ov{\bh}_{1s}^{(2)} \ot_{\! H^{\!L}} \bigl[m\xcdot\jmath_{\nu}(a_1)\gamma(h_1^{(1)})\cdots \jmath_{\nu}(a_s) \gamma(h_s^{(1)})\jmath_{\nu}(\lambda) \ot \ov{\ba}_{s+1,s+r}\bigr]\\
& = (-1)^{rs}\,\ov{\bh}_{1s}^{(2)} \ot_{\!H^{\!L}} \bigl[m\xcdot\jmath_{\nu}(a_1) \gamma(h_1^{(1)})\cdots\jmath_{\nu}(a_s)\gamma(h_s^{(1)})\ot \ov{\lambda a_{s+1}} \ot \ov{\ba}_{s+2,s+r}\bigr]\\
& = \Theta'(m\ot_k \jmath_{\nu}(a_1)\gamma(h_1)\ot_k\cdots\ot_k\jmath_{\nu}(a_s)\gamma(h_s)\ot_k\ov{\lambda a_{s+1}}\ot_k \ov{\ba}_{s+2,s+r}).
\end{align*}
Finally, when $r\ge 1$ the fourth assertion is trivial, while, when $r=0$ it follows from Lemma~\ref{propiedad 4}(2).

\smallskip

We next show that $\Lambda_{rs}$ is well defined. Let $\byy$ be as in the definition of $\Lambda'_{rs}$. We must prove that

\begin{enumerate}[itemsep=0.7ex, topsep=1.0ex, label=(\arabic*)]

\item If some $h_i\in H^{\!L}$, then $\Lambda'(\byy) = 0$,

\item $\Lambda'$ is $H^{\!L}$-balanced in the first $(s-1)$-th tensors of $\cramped{H^{\otimes_k^s}\otimes_k M\otimes_k \ov{A}^{\ot^r}}$,

\item If $r>0$, then $\Lambda'$ is $K$-balanced in the $(s+1)$-th tensor of $\cramped{H^{\otimes_k^s}\otimes_k M\otimes_k \ov{A}^{\ot^r}}$,

\item $\Lambda'(\mathrm{h}_{1s}\ot_k \jmath_{\nu}(\lambda)\xcdot m\ot_k \ov{\ba}_{1r}) = \Lambda'_{rs}(\mathrm{h}_{1s}\ot_k m \ot_k \ov{\ba}_{1r}\xcdot \lambda)$ for all $\lambda \in K$,

\item If $s>0$, then $\Lambda'(\mathrm{h}_{1s} \ot_k m \xcdot \gamma(S(l))\ot_k \ov{a}_{1r}) = \Lambda'(\mathrm{h}_{1s}\xcdot l\ot_k m \ot_k \ov{\ba}_{1r})$ for all $l \in H^{\!L}$,

\end{enumerate}
Item~(1) holds since, by~\cite{BNS1}*{(2.6a)} and~\cite{GGV1}*{Proposition~4.5}, $\gamma^{-1}(l^{(1)})\ot_k \gamma(l^{(2)}) \in E\ot_k \jmath_{\nu}(A)$ for all $l\in H^{\!L}$.~In~or\-der to prove item~(2) we must check that $\Lambda'(\mathrm{h}_{1i}\xcdot l\ot_k h_{i+1,s} \ot_k m\ot_k \ov{\ba}_{1r}) = \Lambda'(\mathrm{h}_{1i}\ot_k l\xcdot h_{i+1,s}\ot_k m \ot_k \ov{\ba}_{1r})$, for all $i\!<\!s$ and $l\in H^{\!L}$. But this follows from Lemma~\ref{prop esp''} and identities~\eqref{aux1} and~\eqref{aux2}. Item~(3) holds since, by Lemma~\ref{prop esp'},
\begin{align*}
\Lambda'(\mathrm{h}_{1s} \ot_k m\xcdot\jmath_{\nu}(\lambda) \ot_k \ov{\ba}_{1r}) & = \bigl[m\xcdot\jmath_{\nu}(\lambda) \gamma_{_{\!\times}}^{-1}(\mathrm{h}_{1s}^{(1)}) \ot_{\hs A} \wt{\gamma}_{\hs A}(\brh_{1s}^{(2)})\ot \ov{\ba}_{1r}\bigr]\\
& = \bigl[m\xcdot\gamma_{_{\!\times}}^{-1}(\brh_{1s}^{(1)})\ot_{\hs A} \wt{\gamma}_{\hs A}(\brh_{1s}^{(2)})\xcdot \jmath_{\nu}(\lambda)\ot \ov{\ba}_{1r}\bigr]\\
& = \bigl[m\xcdot \gamma_{_{\!\times}}^{-1}(\brh_{1s}^{(1)})\ot_{\hs A} \wt{\gamma}_{\hs A}(\brh_{1s}^{(2)})\ot \lambda \xcdot \ov{\ba}_{1r}\bigr]\\
& = \Lambda'(\brh_{1s}\ot_k m\ot_k \lambda\xcdot \ov{\ba}_{1r}).
\end{align*}
When $r\ge 1$, Item~(4) is trivial, while, when $r=0$, it holds since, by Lemma~\ref{prop esp'},
\begin{align*}
\Lambda'(\brh_{1s}\ot_k \jmath_{\nu}(\lambda)\xcdot m) & = \bigl[\jmath_{\nu}(\lambda)\xcdot m\xcdot  \gamma_{_{\!\times}}^{-1}(\brh_{1s}^{(1)}) \ot_{\hs A} \wt{\gamma}_{\hs A}(\brh_{1s}^{(2)})\bigr]\\
& = \bigl[m\xcdot\gamma_{_{\!\times}}^{-1}(\brh_{1s}^{(1)})\ot_{\hs A} \wt{\gamma}_{\hs A}(\brh_{1s}^{(2)})\xcdot \jmath_{\nu}(\lambda)\bigr]\\
& = \bigl[m\xcdot \jmath_{\nu}(\lambda)\gamma_{_{\!\times}}^{-1}(\mathrm{h}_{1s}^{(1)})\ot_{\hs A}\wt{\gamma}_{\hs A}(\brh_{1s}^{(2)})\bigr]\\
& = \Lambda'_{rs}(\brh_{1s}\ot_k m\xcdot\jmath_{\nu}(\lambda)),
\end{align*}
Finally, for Item~(5) we have
\begin{align*}
\Lambda'_{rs}(\brh_{1s}\ot_k m \xcdot \gamma(S(l))\ot_k \ov{\ba}_{1r}) & = \bigl[m\xcdot \gamma(S(l)) \gamma_{_{\!\times}}^{-1}(\brh_{1s}^{(1)}) \ot_{\hs A} \wt{\gamma}_{\hs A}(\brh_{1s}^{(2)})\ot \ov{\ba}_{1r}\bigr]\\
& = \bigl[m\xcdot\gamma^{-1}(h_s^{(1)}l)\gamma_{_{\!\times}}^{-1}(\brh_{1,s-1}^{(1)})\ot_{\hs A} \wt{\gamma}_{\hs A}(\brh_{1s}^{(2)})\ot \ov{\ba}_{1r}\bigr]\\
& = \bigl[m\xcdot\gamma^{-1}(h_s^{(1)}l^{(1)})\gamma_{_{\!\times}}^{-1}(\brh_{1,s-1}^{(1)})\ot_{\hs A} \wt{\gamma}_{\hs A}(\brh_{1,s-1}^{(2)}) \ot_{\hs A} \wt{\gamma}(h_s^{(2)}l^{(2)}) \ot \ov{\ba}_{1r}\bigr]\\
& = \Lambda'_{rs}(\brh_{1s}\xcdot l\ot_k m\ot_k \ov{\ba}_{1r}),
\end{align*}
where the second equality holds  by Lemma~\ref{prop esp''}; and the third one, by identity~\eqref{aux1}.

\smallskip

We next prove that $\Theta_{rs}$ and $\Lambda_{rs}$ are inverse one of each other. To begin with, note that under the first identifications in~\cite{GGV2}*{(3.1)} and~\eqref{ec10},
\begin{equation}\label{pepe4}
\Theta_{*0}=\ide\qquad\text{and}\qquad \Lambda_{*0} = \ide,
\end{equation}
which proves the case $s=0$. Assume $s\ge 1$ and let $m\in M$, $h_1,\dots,h_s\in H$ and $a_1,\dots,a_r\in A$. Set
$$
\bx\coloneqq [m\ot_{\hs A} \wt{\gamma}_{\hs A}(\brh_{1s})\ot \ov{\ba}_{1r}]\in \wh{X}_{rs}(M)\quad\text{and}\quad \byy\coloneqq \ov{\bh}_{1s} \ot_{H^{\!L}} [m\ot \ov{\ba}_{1r}]\in \ov{X}_{rs}(M).
$$
By Lemma~\ref{auxiliar 5}
\begin{align*}
\Lambda\bigl(\Theta(\bx)\bigr) & = (-1)^{rs} \Lambda\bigl(\ov{\bh}_{1s}^{(2)} \ot_{H^{\!L}} [m\xcdot \gamma_{_{\!\times}}(\brh_{1s}^{(1)})\ot \ov
{\ba}_{1r}]\bigr)\\
& = \bigl[m\xcdot \gamma_{_{\!\times}}(\brh_{1s}^{(1)})\gamma_{_{\!\times}}^{-1}(\brh_{1s}^{(2)})\ot_{\hs A} \wt{\gamma}_{\hs A}(\brh_{1s}^{(3)})\ot \ov{\ba}_{1r}\bigr]\\
& = [m \ot_{\hs A} \wt{\gamma}_{\hs A}(\brh_{1s})\ot \ov{\ba}_{1r}],
\end{align*}
as desired. For the other composition, by Lemma~\ref{auxiliar 6} and \cite{GGV1}*{Lemma~2.20}, we have
\begin{align*}
\Theta\bigl(\Lambda(\byy)\bigr) & = (-1)^{rs} \Theta\bigl(\bigl[m\xcdot \gamma_{_{\!\times}}^{-1}(\brh_{1s}^{(1)})\ot_{\hs A} \wt{\gamma}_{\hs A} (\brh_{1s}^{(2)})\ot \ov{\ba}_{1r}\bigr]\bigr)\\
& = \ov{\bh}_{1s}^{(3)} \ot_{H^{\!L}} \bigl[ m\xcdot \gamma_{_{\!\times}}^{-1}(\mathrm{h}_{1s}^{(1)}) \gamma_{_{\!\times}}(\brh_{1s}^{(2)})\ot \ov{\ba}_{1r}\bigr]\\
& =  \ov{\bh}_{1s} \xcdot  1^{(2)} \ot_{H^{\!L}} [m\xcdot \gamma(1^{(1)})\ot \ov{\ba}_{1r}]\\
& = \ov{\bh}_{1s} \ot_{H^{\!L}} [m\xcdot \gamma(1^{(1)})\gamma(S(1^{(2)}))\ot \ov{\ba}_{1r}]\\
& = \ov{\bh}_{1s} \ot_{H^{\!L}} [m\ot\ov{\ba}_{1r}],
\end{align*}
which finishes the proof.
\end{proof}

For each $0\le l\le s$ and $r\ge 0$ such that $r+l\ge 1$, let $\ov{d}^l_{rs}\colon \ov{X}_{rs}(M)\longrightarrow \ov{X}_{r+l-1,s-l}(M)$\index{de1@$\ov{d}^l_{rs}$|dotfillboldidx} be the map defined by $\ov{d}^l_{rs}\coloneqq \Theta_{r+l-1,s-l}\xcirc \wh{d}^l_{rs} \xcirc \Lambda_{rs}$.

\begin{theorem}\label{ppalhom} The Hochschild homology of the $K$-algebra $E$ with coefficients in $M$ is the homology of the chain complex $(\ov{X}_*(M),\ov{d}_*)$, where
$$
\ov{X}_n(M)\coloneqq \bigoplus_{r+s = n} \ov{X}_{rs}(M)\qquad\text{and}\qquad \ov{d}_n(\bx)\coloneqq \begin{cases} \displaystyle{\sum_{l=1}^n \ov{d}^l_{0n}(\bx)} &\text{if $\bx\in \ov{X}_{0n}$,}\\ \displaystyle{\sum^{n-r}_{l=0} \ov{d}^l_{r,n-r}(\bx)} &\text{if $\bx\in \ov{X}_{r,n-r}$ with $r>0$.}\end{cases}\index{xxd@$\ov{X}_n(M)$|dotfillboldidx}\index{de2@$\ov{d}_n$|dotfillboldidx}
$$
\end{theorem}

\begin{proof} By Proposition~\ref{const de aplicaciones} and the definition of $(\ov{X}_*(M),\ov{d}_*)$, the maps
$$
\Theta_*\colon (\wh{X}_*(M),\wh{d}_*)\longrightarrow (\ov{X}_*(M),\ov{d}_*)\quad\text{and}\quad \Lambda_*\colon (\ov{X}_*(M),\ov{d}_*)\longrightarrow (\wh{X}_*(M),\wh{d}_*),
$$
given by $\Theta_n\coloneqq \bigoplus_{r+s = n} \Theta_{rs}$ and $\Lambda_n\coloneqq \bigoplus_{r+s = n} \Lambda_{rs}$, are inverse one of each other.
\end{proof}

By \cite{GGV2}*{Remark~3.2} if $f$ takes its values in $K$, then $(\ov{X}_*(M),\ov{d}_*)$ is the total chain complex of the double~com\-plex $(\ov{X}_{**}(M),\ov{d}^0_{**},\ov{d}^1_{**})$; while if $K=A$, then $(\ov{X}_*(M),\ov{d}_*) = (\ov{X}_{0*}(M),\ov{d}^1_{0*})$.

\begin{lemma}\label{auxiliar 6'} Let $m\in M$, $a,a_1\dots,a_r\in A$, $h_1,\dots, h_s\in H$ and $z\in H^{\!R}$.

\begin{enumerate}[itemsep=0.7ex, topsep=1.0ex, label=\emph{(\arabic*)}]

\item For $\bx\coloneqq \bigl[m\xcdot \gamma_{_{\!\times}}^{-1}(\brh_{1s}^{(1)})\ot_{\hs A} \wt{\gamma}_{\hs A}(\brh_{1,s-1}^{(2)}) \ot_{\hs A} \hspace{-0.5pt}\stackon[-8pt]{$\gamma(h_s^{(2)}) \jmath_{\nu}(a)$}{\vstretch{1.5}{\hstretch{3.3} {\widetilde{\phantom{\;\;\;\;\;\;}}}}} \ot \ov{\ba}_{1r}\bigr]$, we have
$$
\Theta(\bx) = (-1)^{rs}\, \ov{\bh}_{1s} \ot_{H^{\!L}} [m\xcdot \jmath_{\nu}(a)\ot \ov{\ba}_{1r}].
$$

\item For $\bx\coloneqq \bigl[m\xcdot \gamma_{_{\!\times}}^{-1}(\brh_{1s}^{(1)})\ot_{\hs A} \wt{\gamma}_{\hs A}(\brh_{1,i-1}^{(2)}) \ot_{\hs A} \hspace{-0.5pt}\stackon[-8pt]{$\gamma(h_i^{(2)})\gamma(h_{i+1}^{(2)})$} {\vstretch{1.5}{\hstretch{3.2} {\widetilde{\phantom{\;\;\;\;\;\;\;\,}}}}} \ot_{\hs A} \wt{\gamma}_{\hs A}(\brh_{i+2,s}^{(2)}) \ot \ov{\ba}_{1r}\bigr]$, we have
$$
\Theta(\bx) = (-1)^{rs}\,\bigl(\ov{\bh}_{1,i-1}\ot_{H^{\!L}}\ov{h_ih_{i+1}}\ot_{H^{\!L}}\ov{\bh}_{i+2,s}\bigr)\ot_{H^{\!L}}[m\ot\ov{\ba}_{1r}]
$$
(of course, we are assuming that $s\ge 2$ and $1\le i<s$).

\item For $\bx\coloneqq \bigl[m\xcdot \gamma_{_{\!\times}}^{-1}(\brh_{1s}^{(1)})\gamma(z)\ot_{\hs A} \wt{\gamma}_{\hs A}(\brh_{1s}^{(2)}) \ot\ov{\ba}_{1r}\bigr]$, we have $\Theta(\bx) = (-1)^{rs}\,\ov{\Pi}^R(z)\xcdot \ov{\bh}_{1s}\ot_{H^{\!L}}[m\ot \ov{\ba}_{1r}]$.
\end{enumerate}
\end{lemma}

\begin{proof} 1)\enspace Under the first identifications in~\cite{GGV2}*{(3.1)} and~\eqref{ec10}, for $s=0$ the equality in item~(1) becomes $\Theta([m\xcdot \jmath_{\nu}(a)\ot \ov{\ba}_{1r}]) = [m\xcdot \jmath_{\nu}(a)\ot \ov{\ba}_{1r}]$, which follows immediately from equality~\eqref{pepe4}.
Assume now that $s\ge 1$. By~\cite{GGV1}*{Proposition~2.21}, we know that
\begin{equation}\label{auxiliar3}
\jmath_{\nu}(a)\gamma(h1^{(1)})\jmath_{\nu}(b)\gamma(S(1^{(2)})l) = \jmath_{\nu}(a)\gamma(h)\jmath_{\nu}(b)\gamma(l)\qquad\text{for all $a,b\in A$ and $l,h\in H$.}
\end{equation}
By this, the first identity in~\eqref{gama iota y gama gama}, Lemma~\ref{auxiliar 6}, \cite{BNS1}*{(2.4)} and the definition of the action in~\eqref{acciones},
\begin{align*}
\Theta(\bx) & = \Theta\Bigl(\Bigl[m\xcdot \gamma_{_{\!\times}}^{-1}(\brh_{1s}^{(1)})\ot_{\hs A} \wt{\gamma}_{\hs A}(\brh_{1,s-1}^{(2)})\ot_{\hs A} \hspace{-0.5pt}\stackon[-8pt]{$\jmath_{\nu}(h_s^{(2)}\xcdot a)\gamma(h_s^{(3)})$}{\vstretch{1.5}{\hstretch{3.2} {\widetilde{\phantom{\;\;\;\;\;\;\;\;\,}}}}}\! \ot \ov{\ba}_{1r}\Bigr]\Bigr)\\
& = (-1)^{rs}\,  \ov{\bh}_{1s}^{(3)} \ot_{H^{\!L}} \bigl[m \xcdot\gamma_{_{\!\times}}^{-1}(\brh_{1s}^{(1)})\gamma_{_{\!\times}}(\brh_{1s}^{(2)}) \jmath_{\nu}(a) \ot \ov{\ba}_{1r}\bigr] \\
& = (-1)^{rs}\,\ov{\bh}_{1s}\xcdot 1^{(2)} \ot_{H^{\!L}} \bigl[m\xcdot \gamma(1^{(1)})\jmath_{\nu}(a)\ot \ov{\ba}_{1r}\bigr]\\
& = (-1)^{rs}\,\ov{\bh}_{1s} \ot_{H^{\!L}} \bigl[m\xcdot \gamma(1^{(1)})\jmath_{\nu}(a) \gamma(S(1^{(2)}))\ot \ov{\ba}_{1r}\bigr] \\
& = (-1)^{rs}\, \ov{\bh}_{1s} \ot_{H^{\!L}} \bigl[m\xcdot \jmath_{\nu}(a)\ot \ov{\ba}_{1r}\bigr],
\end{align*}
as desired.

\smallskip

\noindent 2)\enspace By the second identity in~\eqref{gama iota y gama gama}, Lemma~\ref{auxiliar 6}, \cite{BNS1}*{(2.4)}, the definition of the action in~\eqref{acciones} and equality~\eqref{auxiliar3},
\begin{align*}
\Theta(\bx)\! & =\! \Theta\Bigl(\!\Bigl[m\cdot\! \gamma_{_{\!\times}}^{-1}(\brh_{1s}^{(1)})\!\ot_{\hs A} \wt{\gamma}_{\hs A}(\brh_{1,i-1}^{(2)})\! \ot_{\hs A}\! \stackon[-8pt]{$\jmath_{\nu}\bigl(f\bigl(h_i^{(2)}\!\ot_k\! h_{i+1}^{(2)}\bigr) \bigr)\gamma\bigl(h_i^{(3)}h_{i+1}^{(3)}\bigr)$} {\vstretch{1.5}{\hstretch{5.5}{\widetilde{\phantom {\,\;\;\;\;\;\;\;\;}}}}} \!\ot_{\hs A} \wt{\gamma}_{\hs A}(\brh_{i+2,s}^{(2)})\! \ot \ov{\ba}_{1r}\Bigr]\!\Bigr)\\
& = (-1)^{rs} \bigl(\ov{\bh}_{1,i-1}^{(3)}\ot_{H^{\!L}} \ov{h_i^{(3)}h_{i+1}^{(3)}} \ot_{H^{\!L}} \ov{\bh}_{i+2,s}^{(3)}\bigr)\ot_{H^{\!L}} \bigl[m\xcdot \gamma_{_{\!\times}}^{-1}(\brh_{1s}^{(1)})\gamma_{_{\!\times}}(\brh_{1s}^{(2)})\ot \ov{\ba}_{1r}\bigr]\\
& = (-1)^{rs}\,\bigl(\ov{\bh}_{1,i-1}\ot_{H^{\!L}} \ov{h_ih_{i+1}} \ot_{H^{\!L}}\ov{\bh}_{i+2,s}\xcdot  1^{(2)}\bigr) \ot_{H^{\!L}}  [m\xcdot \gamma(1^{(1)})\ot \ov{\ba}_{1r}]\\
& = (-1)^{rs}\,\bigl(\ov{\bh}_{1,i-1}\ot_{H^{\!L}} \ov{h_ih_{i+1}} \ot_{H^{\!L}}\ov{\bh}_{i+2,s}\bigr) \ot_{H^{\!L}} \bigl[m\xcdot \gamma(1^{(1)})\gamma(S(1^{(2)}))\ot \ov{\ba}_{1r}\bigr]\\
& = (-1)^{rs}\, \bigl(\ov{\bh}_{1,i-1}\ot_{H^{\!L}} \ov{h_ih_{i+1}} \ot_{H^{\!L}}\ov{\bh}_{i+2,s}\bigr)\ot_{H^{\!L}}[m\ot \ov{\ba}_{1r}],
\end{align*}
as desired.

\smallskip

\noindent 3)\enspace Under the first identifications in~\cite{GGV2}*{(3.1)} and~\eqref{ec10}, for $s=0$ item~(3) be\-comes
$$
\Theta([m\xcdot\gamma(z)\ot \ov{\ba}_{1r}]) = [m\xcdot\gamma(S(\ov{\Pi}^R(z)))\ot \ov{\ba}_{1r}],
$$
which follows immediately from identities~\eqref{pepe} and~\eqref{pepe4}. Assume now that $s\ge 1$. We have
\begin{align*}
\Theta(\bx) 
& = (-1)^{rs} \ov{\Pi}^R(z)\xcdot \ov{\bh}_{1s} \xcdot  1^{(2)}\ot_{H^{\!L}} \bigl[m\xcdot \gamma(1^{(1)}) \ot \ov{\ba}_{1r}\bigr] \\
& = (-1)^{rs} \ov{\Pi}^R(z)\xcdot \ov{\bh}_{1s}\ot_{H^{\!L}} \bigl[m\xcdot \gamma(1^{(1)})\gamma(S(1^{(2)})) \ot \ov{\ba}_{1r}\bigr] \\
& = (-1)^{rs} \ov{\Pi}^R(z)\xcdot \ov{\bh}_{1s}\ot_{H^{\!L}}\bigl[m\ot \ov{\ba}_{1r}\bigr],
\end{align*}
where the first equality holds by the definition of $\Theta$ and Lemma~\ref{auxiliar 6}; the second one, by~\cite{BNS1}*{(2.4)} and the definition of the action in~\eqref{acciones}; and the last one, by equality~\eqref{auxiliar3}.
\end{proof}

\begin{notation} Given a $k$-subalgebra $R$ of $A$ and $0\!\le\! u\!\le\! r$, we let $\ov{X}^u_{rs}(R,M)$\index{xxt@$\ov{X}^u_{rs}(R,M)$|dotfillboldidx} denote the~$k$-sub\-mo\-dule of $\ov{X}_{rs}(M)$ generated by all the elements $\ov{\bh}_{1s} \ot_{H^{\!L}}\hs [m\hs\ot\hs \ov{\ba}_{1r}]$ with $m\in M$, $a_1,\dots,a_r\in A$, $h_1,\dots,h_s\in H$, and at least $u$ of the $a_j$'s in $R$.
\end{notation}

\begin{theorem}\label{ppalhom1} Let $\byy\coloneqq \ov{\bh}_{1s} \ot_{H^{\!L}} [m\ot \ov{\ba}_{1r}]\in\ov{X}_{rs}(M)$, where $m\in M$, $a_1,\dots,a_r\in A$ and $h_1,\dots,h_s\in H$. The following assertions hold:

\begin{enumerate}[itemsep=0.7ex, topsep=1.0ex, label=\emph{(\arabic*)}]

\item For $r\ge 1$ and $s\ge 0$, we have
\begin{align*}
\ov{d}^0(\byy) & = \ov{\bh}_{1s} \ot_{H^{\!L}} [m\xcdot \jmath_{\nu}(a_1)\ot \ov{\ba}_{2r}]\\
& + \sum_{i=1}^{r-1} (-1)^i\,\ov{\bh}_{1s} \ot_{H^{\!L}} [m\ot \ov{\ba}_{1,i-1}\ot \ov{a_ia_{i+1}}\ot \ov{\ba}_{i+2,r}]\\
& + (-1)^r\, \ov{\bh}_{1s} \ot_{H^{\!L}} [\jmath_{\nu}(a_r)\xcdot m\ot \ov{\ba}_{1,r-1}].
\end{align*}

\item For $r\ge 0$ and $s = 1$, we have
\begin{align}
\qquad\quad \ov{d}^1(\byy) &= (-1)^r \bigl[m\xcdot \gamma\bigl(\Pi^R(h_1)\bigr)\ot \ov{\ba}_{1r}\bigr] - (-1)^r \bigl[\gamma(h_1^{(3)})\xcdot m\xcdot\gamma^{-1}(h_1^{(1)})\ot h_1^{(2)}\xcdot \ov{\ba}_{1r}\bigr],\label{formula alternativa}
\shortintertext{while, for $r\ge 0$ and $s>1$, we have}
\ov{d}^1(\byy) & = (-1)^r\,\ov{\Pi}^R(h_1)\xcdot \ov{\bh}_{2s}\ot_{H^{\!L}} [m\ot \ov{\ba}_{1r}]\notag\\
& + \sum_{i=1}^{s-1} (-1)^{r+i}\,\bigl(\ov{\bh}_{1,i-1}\ot_{H^{\!L}} \ov{h_ih_{i+1}} \ot_{H^{\!L}} \ov{\bh}_{i+2,s}\bigr) \ot_{H^{\!L}} [m \ot \ov{\ba}_{1r}]\notag\\
& + (-1)^{r+s}\,\ov{\bh}_{1,s-1} \ot_{H^{\!L}} [\gamma(h_s^{(3)})\xcdot m\xcdot\gamma^{-1}(h_s^{(1)})\ot h_s^{(2)}\xcdot \ov{\ba}_{1r}].\notag
\end{align}

\item For $r\ge 0$ and $s\ge 2$, we have
\begin{align*}
\quad\qquad \ov{d}^2(\byy) & = - \ov{\bh}_{1,s-2}\ot_{H^{\!L}} \bigl[\gamma(h_{s-1}^{(3)}h_s^{(3)})\xcdot m\xcdot \gamma^{-1}(h_s^{(1)}) \gamma^{-1}(h_{s-1}^{(1)})\ot \mathfrak{T}(h^{(2)}_{s-1},h^{(2)}_s, \ov{\ba}_{1r})\bigr] ,
\end{align*}
where
$$
\qquad\quad \mathfrak{T}(h_{s-1},h_s,\ov{\ba}_{1r})\coloneqq \sum_{i=0}^r (-1)^i  h_{s-1}^{(1)}\xcdot (h_s^{(1)}\xcdot \ov{\ba}_{1i})\ot f(h_{s-1}^{(2)}\ot_k h_s^{(2)}) \ot h_{s-1}^{(3)}h_s^{(3)}\xcdot \ov{\ba}_{i+1,r}.
$$

\item Let $R$ be a $k$-subalgebra of $A$. If $R$ is stable under $\rho$ and $f$ takes its values in $R$, then
$$
\ov{d}^l\bigl(\ov{X}_{rs}(M)\bigr)\subseteq \ov{X}^{l-1}_{r+l-1,s-l}(R,M)
$$
for each $r\ge 0$ and $1< l\le s$.
\end{enumerate}
\end{theorem}

\begin{proof} 1)\enspace By the definition of $\Lambda$ and \cite{GGV2}*{Theorem~3.5(1)}, we have
\begin{align*}
\ov{d}^0(\byy) 
& = (-1)^{rs} \Theta \xcirc \wh{d}^0\bigl(\bigl[m\xcdot \gamma_{_{\!\times}}^{-1}(\brh_{1s}^{(1)})\ot_{\hs A} \wt{\gamma}_{\hs A}(\brh_{1s}^{(2)})\ot \ov{\ba}_{1r}\bigr]\bigr)\\
& = (-1)^{rs+s} \Theta \bigl(\bigl[m\xcdot \gamma_{_{\!\times}}^{-1}(\brh_{1s}^{(1)})\ot_{\hs A} \wt{\gamma}_{\hs A}(\brh_{1s}^{(2)})\xcdot \jmath_{\nu}(a_1)\ot \ov{\ba}_{2r}\bigr]\bigr)\\
& + \sum_{i=1}^{r-1} (-1)^{rs+s+i} \Theta \bigl(\bigl[m\xcdot \gamma_{_{\!\times}}^{-1}(\brh_{1s}^{(1)})\ot_{\hs A} \wt{\gamma}_{\hs A} (\brh_{1s}^{(2)})\ot \ov{\ba}_{1,i-1}\ot \ov{a_ia_{i+1}}\ot \ov{\ba}_{i+2,r}\bigr]\bigr)\\
& + (-1)^{rs+s+r} \Theta \bigl(\bigl[\jmath_{\nu}(a_r)\xcdot m\xcdot \gamma_{_{\!\times}}^{-1}(\brh_{1s}^{(1)})\ot_{\hs A} \wt{\gamma}_{\hs A} (\brh_{1s}^{(2)})\ot \ov{\ba}_{1,r-1}\bigr]\bigr).
\end{align*}
The formula for $\ov{d}^0$ follows from this using Lemma~\ref{auxiliar 6'}(1).

\smallskip

\noindent 2)\enspace By the definition of $\Lambda$, \cite{GGV1}*{Proposition~5.19} and~\cite{GGV2}*{Theorem~3.5(2)}, we have
\begin{align*}
\ov{d}^1(\byy) 
\! & = (-1)^{rs} \Theta \xcirc \wh{d}^1\bigl(\bigl[m\xcdot \gamma_{_{\!\times}}^{-1}(\brh_{1s}^{(1)})\ot_{\hs A} \wt{\gamma}_{\hs A} (\brh_{1s}^{(2)})\ot \ov{\ba}_{1r}\bigr]\bigr)\\
& = (-1)^{rs} \Theta \bigl(\bigl[m\xcdot \gamma_{_{\!\times}}^{-1}(\brh_{2s}^{(1)})\gamma(\Pi^R(h_1))\ot_{\hs A} \wt{\gamma}_{\hs A}(\brh_{2s}^{(2)}) \ot \ov{\ba}_{1r}\bigr]\bigr)\\
& +\hs \sum_{i=1}^{r-1} (-1)^{rs+i} \Theta \bigl(\!\bigl[m\xcdot \gamma_{_{\!\times}}^{-1}(\brh_{1s}^{(1)})\hs \ot_{\hs A}  \wt{\gamma}_{\hs A} (\brh_{1,i-1}^{(2)})\hs \ot_{\hs A} \!\stackon[-8pt]{$\gamma(h_i^{(2)})\gamma(h_{i+1}^{(2)})$} {\vstretch{1.5}{\hstretch{4.1}{\widetilde{\phantom{\,\;\;\;\;\;}}}}}\!\ot_{\hs A}\hs \wt{\gamma}_{\hs A}(\brh_{i+2,s}^{(2)})\hs \ot\hs \ov{\ba}_{1r}\bigr]\!\bigr)\\
& + (-1)^{rs+r} \Theta \bigl(\bigl[\gamma(h_s^{(3)})\xcdot m\xcdot \gamma_{_{\!\times}}^{-1}(\brh_{1s}^{(1)})\ot_{\hs A} \wt{\gamma}_{\hs A} (\brh_{1,s-1}^{(2)})\ot h_s^{(2)}\xcdot \ov{\ba}_{1r} \bigr]\bigr).
\end{align*}
The formula for $\ov{d}^1$ follows from this using Lemma~\ref{auxiliar 6'} and the fact that $\ov{\Pi}^R\xcirc \Pi^R = \ov{\Pi}^R$.

\smallskip

\noindent 3)\enspace By the definition of $\Lambda$ and \cite{GGV2}*{Theorem~3.6}, we have
\begin{align*}
\ov{d}^2(\byy) 
& = (-1)^{rs} \Theta \xcirc \wh{d}^2\bigl(\bigl[m\xcdot \gamma_{_{\!\times}}^{-1}(\brh_{1s}^{(1)})\ot_{\hs A} \wt{\gamma}_{\hs A}(\brh_{1s}^{(2)})\ot \ov{\ba}_{1r}\bigr]\bigr)\\
& = (-1)^{rs+s+1} \Theta \bigl(\bigl[\gamma(h_{s-1}^{(3)}h_s^{(3)})\xcdot m\xcdot \gamma_{_{\!\times}}^{-1}(\brh_{1s}^{(1)})\ot_{\hs A} \wt{\gamma}_{\hs A}(\brh_{1,s-2}^{(2)})\ot \mathfrak{T}(h^{(2)}_{s-1},h^{(2)}_s,\ov{\ba}_{1r})\bigr]\bigr).
\end{align*}
The formula for $\ov{d}^2$ follows from this using Lemma~\ref{auxiliar 6'}(1).

\smallskip

\noindent 4)\enspace Let $\wh{X}^{l-1}_{r+l-1,s-l}(R,M)$ be as in \cite{GGV2}*{Notation~3.4}. By Remark~\ref{estable bajo rho = estable bajo chi} and \cite{GGV2}*{Theorem~3.5(3)} this item follows from the fact that $\ov{X}^{l-1}_{r+l-1,s-l}(R,M) = \Theta\bigl(\wh{X}^{l-1}_{r+l-1,s-l}(R,M)\bigr)$.
\end{proof}

\begin{remark}\label{para el calculo de ss} By the second equality in~\eqref{pepe}, we have
$$
\ov{\Pi}^R(h_1) \xcdot \bigl[m \ot \ov{\ba}_{1r}\bigr]  = \bigl[m\xcdot \gamma\bigl(\Pi^R(h)\bigr)\ot \ov{\ba}_{1r}\bigr]\qquad\text{for all $h\in H$, $m\in M$ and $a_1,\dots,a_r\in A$.}
$$
This gives an alternative formula for $\ov{d}^1(\byy)$ in~\eqref{formula alternativa}. We will use this fact in the proof of Proposition~\ref{calculo de ss}.
\end{remark}

\begin{proposition}\label{F^h} For each $h\in H$, the map $\cramped{F_*^h\colon \bigl(M\ot \ov{A}^{\ot^*}\ot, b_*\bigr)\longrightarrow \bigl(M\ot \ov{A}^{\ot^*}\ot, b_*\bigr)}$\index{fr@$F_*^h$|dotfillboldidx}, defined by
$$
F_r^h([m \ot \ov {\ba}_{1r}])\coloneqq \bigl[\gamma(h^{(3)})\xcdot m\xcdot \gamma^{-1}(h^{(1)})\ot  h^{(2)}\xcdot \ov{\ba}_{1r}\bigr],
$$
is a morphism of complexes. Moreover the following facts hold:

\begin{enumerate}[itemsep=0.7ex, topsep=1.0ex, label=\emph{(\arabic*)}]

\item For each $h,l\in H$, the endomorphisms of $\Ho^K_*(A,M)$ induced by $F_*^h\xcirc F_*^l$ and $F_*^{hl}$ coincide.

\item $F_r^l([m\ot \ov{\ba}_{1r}]) = l\xcdot [m\ot \ov{\ba}_{1r}]$, for all $l\in H^{\!L}$ (see~\eqref{acciones}). In particular $F_*^1$ is the identity map.

\end{enumerate}
Conse\-quently, $\Ho^K_*(A,M)$ is a left $H$-module.
\end{proposition}

\begin{proof} Using the first equality in Proposition~\ref{gama iota y gama gama} and Lemma~\ref{propiedad 4}(4), it is easy to see that the maps $F_*^h$ are~well defined. Moreover by Definition~\ref{def modulo algebra debil}(2), the first identity in~\eqref{gama iota y gama gama} and Lemma~\ref{propiedad 4}(4),  they are morphisms of com\-plexes.

\smallskip

For $h,l\in H$, let $\cramped{\bigl(\mathfrak{h}_r\colon M\ot \ov{A}^{\ot^r}\ot \longrightarrow M \ot \ov{A}^{\ot^{r+1}}\ot \bigr)_{r\ge 0}}$ be the family of maps, defined by
$$
\mathfrak{h}_r([m\ot \ov{\ba}_{1r}]) \coloneqq - \bigl[ \gamma(h^{(3)}l^{(3)})\xcdot m\xcdot \gamma^{-1}(l^{(1)})\gamma^{-1}(h^{(1)})\ot  \mathfrak{T}(h^{(2)}),l^{(2)}),\ov {\ba}_{1r})\bigr],
$$
where $\mathfrak{T}(h,l,\ov{\ba}_{1r})$ is as in Theorem~\ref{ppalhom1}(3). In order to prove item~(1) it suffices to show that $\bigl(\mathfrak{h}_r\bigr)_{r\ge 0}$ is~a~ho\-mo\-topy from $F^{hl}_*$ to $F^h_*\xcirc F^l_*$. For this we must check that
\begin{equation}\label{pepe6}
\bigl(F^{hl}_r-F^h_r\xcirc F^l_r\bigr)([m\ot \ov {\ba}_{1r}]) = \begin{cases} \bigl(b \xcirc\mathfrak{h}_0\bigr)([m]) & \text{if $r = 0$,}\\ \bigl(b \xcirc \mathfrak{h}_r+ \mathfrak{h}_{r-1}\xcirc b \bigr)([m\ot\ov {\ba}_{1r}]) & \text{if $r>0$.}\end{cases}
\end{equation}
Let $\byy\coloneqq (\ov{h}\ot_{H^{\!L}} \ov{l}) \ot_{H^{\!L}} \bigl[m\ot \ov{\ba}_{1r}\bigr] \in\ov{X}_{r2}(M)$. Since $(\ov{X}_*(M),\ov{d}_*)$ is a chain complex,
$$
\ov{d}^1(\ov{d}^1(\byy)) = \begin{cases} \ov{d}^0(\ov{d}^2(\byy)) & \text{if $r = 0$,}\\ \ov{d}^0(\ov{d}^2(\byy)) + \ov{d}^2(\ov{d}^0(\byy)) & \text{if $r>0$.}\end{cases}
$$
By this fact and Theorem~\ref{ppalhom1}, to prove equality~\eqref{pepe6} it suffices to show that
\begin{equation}\label{igualdad}
\begin{aligned}
\ov{d}^1(\ov{d}^1(\byy)) & = \bigl[\gamma\bigl(h^{(3)}l^{(3)}\bigr)\xcdot m\xcdot \gamma^{-1}\bigl(h^{(1)}l^{(1)}\bigr)\ot h^{(2)}l^{(2)}\xcdot \ov{\ba}_{1r}\bigr]\\
& - \bigl[\gamma(h^{(3)})\gamma(l^{(3)})\xcdot m\xcdot\gamma^{-1}(l^{(1)})\gamma^{-1}(h^{(1)})\ot h^{(2)}\xcdot \bigl(l^{(2)}\xcdot \ov {\ba}_{1r}\bigr)\bigr].
\end{aligned}
\end{equation}
Now, a direct computation shows that
\begin{align*}
\ov{d}^1\bigl(\ov{d}^1(\byy)\bigr) & = (-1)^r \ov{d}^1\bigl(\ov{\Pi}^R(h)\,\ov{l}\ot_{H^{\!L}} [m\ot \ov{\ba}_{1r}]\bigr) - (-1)^r \ov{d}^1\bigl(\ov{hl}\ot_{H^{\!L}} [m \ot \ov{\ba}_{1r}]\bigr)\\
& + (-1)^r \ov{d}^1\bigl(\ov{h} \ot_{H^{\!L}}[\gamma(l^{(3)})\xcdot m\xcdot\gamma^{-1}(l^{(1)})\ot  l^{(2)}\xcdot \ov {\ba}_{1r}]\bigr)\\
& = \bigl[m\xcdot \gamma\bigl(\Pi^R\bigl(\ov{\Pi}^R(h)l\bigr)\bigr)\ot \ov{\ba}_{1r}\bigr]\\
& - \bigl[\gamma\bigl(\ov{\Pi}^R(h)^{(3)}l^{(3)}\bigr)\xcdot m\xcdot\gamma^{-1}\bigl(\ov{\Pi}^R(h)^{(1)}l^{(1)}\bigr)\ot \ov{\Pi}^R(h)^{(2)}l^{(2)}\xcdot \ov{\ba}_{1r}\bigr]\\
& - \bigl[m\xcdot \gamma\bigl(\Pi^R(hl)\bigr)\ot \ov{\ba}_{1r}\bigr]
+ \bigl[\gamma\bigl(h^{(3)}l^{(3)}\bigr)\xcdot m\xcdot \gamma^{-1}\bigl(h^{(1)}l^{(1)}\bigr)\ot h^{(2)}l^{(2)}\xcdot \ov {\ba}_{1r}\bigr]\\
& + \bigl[ \gamma(l^{(3)})\xcdot m\xcdot\gamma^{-1}(l^{(1)})\gamma\bigl(\Pi^R(h)\bigr) \ot  l^{(2)}\xcdot \ov{\ba}_{1r}\bigr]\\
& - \bigl[\gamma(h^{(3)})\gamma(l^{(3)})\xcdot m\xcdot\gamma^{-1}(l^{(1)})\gamma^{-1}(h^{(1)})\ot h^{(2)}\xcdot (l^{(2)}\xcdot \ov {\ba}_{1r})\bigr].
\end{align*}
Using the second equality in~\eqref{pepe}, and the fact that, by Lemma~\ref{prop esp''} and identity~\eqref{aux2},
\begin{align*}
\bigl[\gamma\bigl(l^{(3)}\bigr)\xcdot m\xcdot\gamma^{-1}\bigl(l^{(1)}\bigr)\gamma(\Pi^R(h))\ot l^{(2)}\xcdot \ov{\ba}_{1r}\bigr] & = \bigl[\gamma\bigl(l^{(3)}\bigr)\xcdot m\xcdot\gamma^{-1}\bigl(l^{(1)}\bigr)\gamma(S(\ov{\Pi}^R(h)))\ot l^{(2)}\xcdot \ov{\ba}_{1r}\bigr]\\
& = \bigl[\gamma\bigl(l^{(3)}\bigr)\xcdot m\xcdot\gamma^{-1}\bigl(\ov{\Pi}^R(h) l^{(1)}\bigr)\ot l^{(2)}\xcdot \ov{\ba}_{1r}\bigr]\\
& = \bigl[\gamma\bigl(\ov{\Pi}^R(h)^{(3)}l^{(3)}\bigr)\xcdot m\xcdot\gamma^{-1}\bigl(\ov{\Pi}^R(h)^{(1)}l^{(1)}\bigr)\ot  \bigl(\ov{\Pi}^R(h)^{(2)}l^{(2)}\bigr)\xcdot \ov{\ba}_{1r}\bigr],
\end{align*}
we obtain that equality~\eqref{igualdad} holds.

We  next prove item~(2). Let $l\in H^{\!L}$. For $r=1$, we have
$$
F_0^l([m])=\bigl[\gamma(l^{(2)})\xcdot m\xcdot \gamma^{-1}(l^{(1)})\bigr]=\bigl[m\xcdot \gamma^{-1}(l^{(1)})\gamma(l^{(2)})\bigr]=\bigl[m\xcdot \gamma(\Pi^R(l))\bigr] = \bigl[m\xcdot \gamma(S(l))\bigr],
$$
where the second equality holds by~\cite{BNS1}*{(2.6a)}, \cite{GGV1}*{Proposition~4.5} and Remark~\ref{incluido}; the third one, by~\cite{GGV1}*{Proposition~5.19}; and the fourth one, by the second identity in~\eqref{pepe} and the fact that $\ov{\Pi}^R(l) = l$. Assume now $r\ge 1$. To begin note that, by Remark~\ref{incluido}, ~\cite{GGV1}*{Proposition~4.5} and De\-fi\-ni\-tion~\ref{def modulo algebra debil}(2),
$$
F_r^l([m\ot \ov {\ba}_{1r}])=\bigl[m\xcdot \gamma^{-1}(l^{(1)})\ot l^{(2)}\xcdot \ov{\ba}_{1,r-1}\ot (l^{(3)}\xcdot a_r)(l^{(4)}\xcdot 1_A)\bigr] = \bigl[m\xcdot \gamma^{-1}(l^{(1)})\ot l^{(2)}\xcdot \ov{\ba}_{1r}\bigr].
$$
We claim that
\begin{equation}
l^{(1)}\ot_k l^{(2)}\xcdot \ba_{1r} = l^{(1)}\ot_k (l^{(2)}\xcdot 1_A)a_1\ot \ba_{2r}.
\end{equation}
By~\cite{BNS1}*{(2.6a)} and Proposition~\ref{modulo algebra debil}(1),
$$
l^{(1)}\ot_k l^{(2)}\xcdot \ba_{1r} = l^{(1)}\ot_k l^{(2)}\xcdot \ba_{1,r-1}\ot l^{(3)}\xcdot a_r = l^{(1)}\ot_k l^{(2)}\xcdot \ba_{1,r-1}\ot (l^{(3)}\xcdot 1_A) a_r.
$$
If $r=1$ this ends the proof of the claim. Assume that $r>1$. In this case, by Remark~\ref{incluido} and De\-fi\-ni\-tion~\ref{def modulo algebra debil}(2),
$$
l^{(1)}\ot_k l^{(2)}\xcdot \ba_{1r} = l^{(1)}\ot_k l^{(2)}\xcdot \ba_{1,r-2}\ot (l^{(3)}\xcdot a_{r-1})(l^{(4)}\xcdot 1_A) \ot a_r = l^{(1)}\ot_k l^{(2)}\xcdot \ba_{1,r-1}\ot a_r,
$$
and the claim follows from an evident inductive argument. Thus, by~\cite{GGV1}*{Proposition~4.5}  and Remark~\ref{incluido},
$$
F_r^l([m\ot\ov {\ba}_{1r}]) = \bigl[m\xcdot \gamma^{-1}(l^{(1)})\ot (l^{(2)}\xcdot 1_A)a_1\ot\ba_{2r}\bigr] = \bigl[m\xcdot \gamma^{-1}(l^{(1)})\jmath_{\nu}(l^{(2)}\xcdot 1_A)\ot \ba_{1r}\bigr],
$$
and so, by~\cite{GGV1}*{Propositions~4.5 and~5.19}, the second identity in~\eqref{pepe} and the fact that $\ov{\Pi}^R(l) = l$,
$$
F_r^l([m\ot\ov {\ba}_{1r}]) = \bigl[m\xcdot \gamma(\Pi^R(l))\ot \ba_{1r}\bigr] = \bigl[m\xcdot \gamma(S(l))\ot\ba_{1r}\bigr],
$$
as desired.
\end{proof}

\begin{example}\label{A es K homologia} If $A=K$, then $\Ho^K_*(E,M) = \Ho_*\bigl(H,M\ot\bigr)$, where $M\ot$ is considered as a left $H$-module via the action given by $h\xcdot [m]\coloneqq [\gamma(h^{(2)})\xcdot m\xcdot \gamma^{-1}(h^{(1)})]$.
\end{example}

In the following proposition, for each $r\ge 0$, we consider $M\ot \ov{A}^{\ot^*}\ot$ as a left $H$-module via the action introduced in Proposition~\ref{F^h}, and we consider $M\ot \ov{A}^{\ot^*}\ot$ as a left $H^{\! L}$-module via the canonical inclusion of $H^{\! L}$ into $H$. In Proposition~\ref{F^h}(2) we prove that these structures coincide with the ones introduced in~\eqref{acciones} (which are the ones used in Theorem~\ref{ppalhom1}).

\begin{proposition}\label{calculo de ss} The spectral sequence of~\cite{GGV2}*{(3.3)} satisfies
$$
E^1_{rs} = \ov{H}^{\ot_{\!H^{\!L}}^s}\ot_{\!H^{\!L}} \Ho^K_r(A,M) \qquad\text{and}\qquad E^2_{rs} = \Ho_s\bigl(H, \Ho^K_r(A,M)\bigr).
$$
\end{proposition}

\begin{proof} For each $i,n\ge 0$, let $F^i(\ov{X}_n(M))\coloneqq \bigoplus_{s=0}^i \ov{X}_{n-s,s}(M)$ \index{fz@$F^i(\ov{X}_n(M))$|dotfillboldidx}. The chain complex $(\ov{X}_*(M),\ov{d}_*)$ is filtrated by
\begin{equation}\label{ss1}
0=F^{-1}(\ov{X}_*(M))\subseteq F^0(\ov{X}_*(M))\subseteq F^1(\ov{X}_*(M))\subseteq F^2(\ov{X}_*(M))\subseteq F^3(\ov{X}_*(M))\subseteq \dots.
\end{equation}
Moreover the isomorphism $\Theta_*\colon (\wh{X}_*(M),\wh{d}_*)\to (\ov{X}_*(M),\ov{d}_*)$ preserves filtrations, where we consider the chain complex $(\wh{X}_*(M),\wh{d}_*)$ endowed with the filtration introduced at the beginning of \cite{GGV2}*{Subsection~3.2}. So, the spectral sequence of \cite{GGV2}*{(3.3)} coincides with the spectral sequence determined by the filtration~\eqref{ss1}. Thus, the formula for $E^1_{rs}$ follows from Theorem~\ref{ppalhom1}(1) and the formula for $E^2_{rs}$ follows from Theorem~\ref{ppalhom1}(2), Remark~\ref{para el calculo de ss} and Proposition~\ref{hom}(1).
\end{proof}

\section{Hochschild cohomology of cleft extensions}\label{Hochschild cohomology of cleft extensions}
Let $H$, $A$, $\rho$, $\chi_{\rho}$, $f$, $\mathcal{F}_f$, $E$, $K$, $M$, $\nu$, $\jmath_{\nu}$ and $\gamma$ be as in the previous section. Assume that the hypotheses of that section are fulfilled. In particular $H$ is a weak Hopf algebra, $A$ is a weak module algebra and $f$ is convolution invertible. Let $\gamma^{-1}$ be as in equality~\eqref{gamma{-1}}. By definition the Hochschild cohomology $\Ho_{\hs K}^*(E,M)\index{hh$@$\Ho_{\hs K}^*(E,M)$|dotfillboldidx}$, of the $K$-algebra $E$ with co\-e\-fficients in an $E$-bimodule $M$, is the cohomology of the normalized Hochschild cochain complex $\cramped{\bigl(\Hom_{K^e}\bigl(\ov{E}^{\ot *},M\bigr),b^*\bigr)}$, where $b^*$ is the canonical Hochschild boundary map. In \cite{GGV2}*{Section~4} a cochain complex $(\wh{X}^*(M),\wh{d}^*)$\index{xxg@$\wh{X}^n(M)$|dotfillboldidx} was obtained, simpler than the canonical one, that gives the Hochschild cohomology of $E$ with coefficients in $M$. In this section we prove that $(\wh{X}^*(M),\wh{d}^*)$ is isomorphic to a simpler complex $(\ov{X}^*(M),\ov{d}^*)$. When $K$ is separable, the complex $(\ov{X}^*(M),\ov{d}^*)$ gives the absolute Hochschild cohomology of $E$ with coefficients in $M$. We recall from \cite{GGV2}*{Section~4} that
$$
\wh{X}^n(M) = \bigoplus_{\substack{r,s\ge 0\\ r+s = n}} \wh{X}^{rs}(M),\qquad\text{where $\wh{X}^{rs}(M)\coloneqq \Hom_{(A,K)}\bigl(\wt{E}^{\ot_{\hs A}^s}\ot\ov{A}^{\ot^r},M\bigr)$,}\index{xxe@$\wh{X}^{rs}(M)$|dotfillboldidx}
$$
and that there exist maps $\wh{d}_l^{rs}\colon\wh{X}^{r+l-1,s-l}(M)\longrightarrow \wh{X}^{rs}(M)$\index{df@$\wh{d}_l^{rs}$|dotfillboldidx} such that
$$
\wh{d}^n(\alpha)\coloneqq \sum^{r+1}_{l=0} \wh{d}_l^{r-l+1,n-r+l-1} (\alpha)\index{dg@$\wh{d}^n$|dotfillboldidx} \qquad\text{for all $\alpha\in \wh{X}^{r,n-r-1}(M)$.}
$$
By~\cite{GGV1}*{Propositions~2.22 and~4.6}, we know that $M$ is a $(K,K\ot_k H^{\!L})$-bimodule via
$$
\lambda \xcdot m\xcdot (\lambda'\ot_k l)\coloneqq \jmath_{\nu}(\lambda)\gamma(S(l))\xcdot m\xcdot \jmath_{\nu}(\lambda').
$$
For each $r,s\ge 0$, we set $\ov{X}^{rs}(M)\coloneqq\Hom_{(K,K\ot_k H^{\!L})}\bigl(\ov{H}^{\ot_{H^{\!L}}^s}\ot_k \ov{A}^{\ot^r},M\bigr)$ \index{xxf@$\ov{X}^{rs}(M)$|dotfillboldidx}, where we consider $\ov{H}^{\ot_{\! H^{\!L}}^s}\ot_k\ov{A}^{\ot^r}$ as a $(K,K\ot_k H^{\!L})$-bimodule via
$$
\lambda \xcdot \bigl(\ov{\bh}_{1s}\ot_k\ov{\ba}_{1r}\bigr)\xcdot (\lambda'\ot_k l) \coloneqq \ov{\bh}_{1s}\xcdot l \ot_k \lambda \xcdot  \ov{\ba}_{1r}\xcdot \lambda'.
$$
Let $M^{\!K}\coloneqq \{m\in M: \lambda\xcdot m = m\xcdot \lambda\text{ for all $\lambda\in K$}\}$. Since $\ov{H}^{\ot_{\! H^{\!L}}^0}  = H^{\!L}$ and $\ov{A}^{\ot^0} = K$, we have
\begin{equation}\label{ec10'}
\ov{X}^{r0}(M)\simeq \Hom_{K^e} \bigl(\ov{A}^{\ot ^r},M \bigr)\qquad\text{and}\qquad \ov{X}^{0s}(M) \simeq \Hom_{\!H^{\!L}}\bigl(\ov{H}^{\ot_{\! H^{\!L}}^s},M^{\!K}\bigr),
\end{equation}
where $M^{\!K}$ is considered as a right $H^{\!L}$-module via $m\xcdot l\coloneqq \gamma(S(l))\xcdot m$.

\begin{remark}\label{pepitito3} For each $r,s\ge 0$, we have $\ov{X}^{rs}(M)\simeq \Hom_{\!H^{\!L}}\bigl(\ov{H}^{\ot_{\! H^{\!L}}^s},\Hom_{K^e}\bigl(\ov{A}^{\ot ^r},M \bigr)\bigr)$, where $\Hom_{K^e}\bigl(\ov{A}^{\ot^r},M\bigr)$ is considered as right $H^{\!L}$-module via $(\beta\xcdot l)(\ov{\ba}_{1r})\coloneqq \gamma(S(l))\xcdot \beta(\ov{\ba}_{1r})$.
\end{remark}

\begin{proposition}\label{const de aplicaciones en coh} For each $r,s\ge 0$ there exist maps
$$
\Theta^{rs}\colon \ov{X}^{rs}(M)\longrightarrow \wh{X}^{rs}(M)\index{zza@$\Theta^{rs}$|dotfillboldidx}\qquad\text{and}\qquad \Lambda^{rs}\colon \wh{X}^{rs}(M)\longrightarrow \ov{X}^{rs}(M),\index{zmm@$\Lambda^{rs}$|dotfillboldidx}
$$
such that for
$$
\bx\coloneqq \stackon[-8pt]{$\jmath_{\nu}(a_1)\gamma(h_1)$}{\vstretch{1.5}{\hstretch{2.8} {\widetilde{\phantom{\;\;\;\;\;\;}}}}} \ot_{\hs A}\cdots \ot_{\hs A} \stackon[-8pt]{$\jmath_{\nu}(a_s)\gamma(h_s)$} {\vstretch{1.5}{\hstretch{2.8} {\widetilde{\phantom{\;\;\;\;\;\;}}}}}\ot \ov{\ba}_{s+1,s+r} \in \wt{E}^{\ot_{\hs A}^s}\ot\ov{A}^{\ot^r} \quad\text{and}\quad \byy\coloneqq  \ov{\bh}_{1s}\ot_k \ov{\ba}_{1r}\in \ov{H}^{\ot_{\! H^{\!L}}^s}\ot_k \ov{A}^{\ot^r},
$$
we have
\begin{align*}
&\Theta^{rs}(\beta)(\bx)\coloneqq (-1)^{rs}\jmath_{\nu}(a_1)\gamma(h_1^{(1)})\cdots\jmath_{\nu}(a_s)\gamma(h_s^{(1)})\xcdot \beta\bigl(\ov{\bh}_{1s}^{(2)} \ot_k \ov{\ba}_{s+1,s+r}\bigr)\\
\shortintertext{and}
&\Lambda^{rs}(\alpha)(\byy)\coloneqq (-1)^{rs} \gamma_{_{\!\times}}^{-1}(\brh_{1s}^{(1)})\cdot \alpha\bigl(\wt{\gamma}_{\hs A}(\brh_{1s}^{(2)}) \ot \ov{\ba}_{1r}\bigr).
\end{align*}
Moreover the maps $\Theta^{rs}$ and $\Lambda^{rs}$ are inverse one of each other.
\end{proposition}

\begin{proof} Mimic the proof of Proposition~\ref{const de aplicaciones} .
\end{proof}

For each $0\le l\le s$ and $r\ge 0$ such that $r+l\ge 1$, let $\ov{d}_l^{rs}\colon \ov{X}^{r+l-1,s-l}(M)\to \ov{X}^{rs}(M)$ be the map $\ov{d}_l^{rs}\coloneqq \Lambda^{r+l-1,s-l}\xcirc \wh{d}_l^{rs} \xcirc \Theta^{rs}$.\index{dg1@$\ov{d}_l^{rs}$|dotfillboldidx}

\begin{theorem}\label{ppalchom} The Hochschild cohomology of the $K$-algebra $E$ is the cohomology of $(\ov{X}^*(M),\ov{d}^*)$, where
$$
\ov{X}^n(M)\coloneqq \bigoplus_{r+s=n}\ov{X}^{rs}(M)\index{xxk@$\ov{X}^n(M)$|dotfillboldidx}\quad\text{and}\quad \ov{d}^n(\beta)\coloneqq \sum^{r+1}_{l=0} \ov{d}_l^{r-l+1,n-r+l-1}(\beta)\quad\text{for all $\beta\in \ov{X}^{r,n-r-1}(M)$.}\index{dg2@$\ov{d}^n$|dotfillboldidx}
$$
\end{theorem}

\begin{proof} By Proposition~\ref{const de aplicaciones en coh} and the definition of $(\ov{X}_*(M),\ov{d}_*)$, the map $\Theta^*\colon (\ov{X}^*(M),\ov{d}^*)\longrightarrow (\wh{X}^*(M),\wh{d}^*)$, given by $\Theta^n\coloneqq \bigoplus_{r+s = n} \Theta^{rs}$, is an isomorphism of complexes.
\end{proof}

By~\cite{GGV2}*{Remark~4.2}, if $f$ takes its values in $K$, then $(\ov{X}^*(M),\ov{d}^*)$ is the total cochain com\-plex of the double complex $(\ov{X}^{**}(M),\ov{d}_0^{**},\ov{d}_1^{**})$; while, if $K=A$, then $(\ov{X}^*(M),\ov{d}^*) = (\ov{X}^*(M),\ov{d}_1^{0*})$.

\begin{lemma}\label{auxiliar 6''} Let $\beta\in\ov{X}^{rs}(M)$, $a,a_1\dots,a_r\in A$, $h_1,\dots, h_s\in H$ and $z\in H^{\!R}$.

\begin{enumerate}[itemsep=0.7ex, topsep=1.0ex, label=\emph{(\arabic*)}]

\item We have $\gamma_{_{\!\times}}^{-1}(\brh_{1s}^{(1)})\xcdot\Theta(\beta) \bigl(\wt{\gamma}_{\hs A}(\brh_{1s}^{(2)})\xcdot \jmath_{\nu}(a)\ot \ov{\ba}_{1r}\bigr) = (-1)^{rs} \jmath_{\nu}(a)\xcdot\beta\bigl(\ov{\bh}_{1s}\ot_k \ov{\ba}_{1r}\bigr)$.

\item For $s\ge 2$ and $1\le i<s$, we have
\begin{align*}
& \gamma_{_{\!\times}}^{-1}(\brh_{1s}^{(1)})\xcdot\Theta(\beta)\Bigl(\wt{\gamma}_{\hs A}(\brh_{1,i-1}^{(2)})\ot_{\hs A}\!\! \hspace{-0.5pt}\stackon[-8pt] {$\gamma(h_i^{(2)}) \gamma(h_{i+1}^{(2)})$} {\vstretch{1.5}{\hstretch{3.2} {\widetilde{\phantom{\;\;\;\;\;\;\;\,}}}}}\!\! \ot_{\hs A} \wt{\gamma}_{\hs A}(\brh_{i+2,s}^{(2)})\ot \ov{\ba}_{1r}\Bigr)\\
&= (-1)^{rs} \beta\bigl(\ov{\bh}_{1,i-1}\ot_{H^{\!L}} \ov{h_ih_{i+1}} \ot_{H^{\!L}}\ov{\bh}_{i+2,s}\ot_k \ov{\ba}_{1r} \bigr).
\end{align*}

\item We have $\gamma_{_{\!\times}}^{-1}(\brh_{1s}^{(1)})\gamma(z)\xcdot\Theta(\beta)\bigl(\wt{\gamma}_{\hs A}(\brh_{1s}^{(2)})\ot\ov{\ba}_{1r}\bigr) =  (-1)^{rs} \beta \bigl(\ov{\Pi}^R(z)\xcdot \ov{\bh}_{1s}\ot_k \ov{\ba}_{1r}\bigr)$.

\end{enumerate}

\end{lemma}

\begin{proof} Mimic the proof of Lemma~\ref{auxiliar 6'}.
\end{proof}

\begin{notation} For each $k$-subalgebra $R$ of $A$ and $0\!\le\! u\!\le\! r$, we set $\ov{X}_u^{rs}(R,M)\coloneqq \Lambda\bigl(\wh{X}_u^{rs}(R,M)\bigr)$\index{xxu@$\ov{X}_u^{rs}(R,M)$|dotfillboldidx}, where $\wh{X}_u^{rs}(R,M)$ is as in \cite{GGV2}*{Notation~4.4}
\end{notation}

\begin{theorem}\label{ppalchom2} Let $\byy \coloneqq \ov{\bh}_{1s}\ot_k \ov{\ba}_{1r}$, where $h_1,\dots,h_s\in H$ and $a_1,\dots,a_r\in A$. The following assertions hold:

\begin{enumerate}[itemsep=0.7ex, topsep=1.0ex, label=\emph{(\arabic*)}]

\item For $r\ge 1$ and $s\ge 0$, we have
$$
\qquad\quad\,\,\,\, d_0(\beta)(\byy)\! =\! \jmath_{\nu}(a_1)\beta\bigl(\ov{\bh}_{1s}\ot_k \ov{\ba}_{2r}\bigr)+\sum_{i=1}^{r-1} (-1)^i \beta\bigl(\ov{\bh}_{1s}\ot_k (\ov {\ba}_{1,i-1}\ot \ov{a_ia_{i+1}}\ot \ov {\ba}_{i+2,r})\bigr) + \beta\bigl(\ov{\bh}_{1s} \ot_k \ov{\ba}_{1,r-1}\bigr)\xcdot\jmath_{\nu}(a_r).
$$

\item For $r\ge 0$ and $s = 1$, we have
\begin{align}
\qquad\quad\,\,\,\ov{d}^1(\beta)(\byy) &= (-1)^r\Bigl(\gamma(\Pi^R(h_1))\xcdot\beta\bigl(\ov{\ba}_{1r}\bigr)- \gamma^{-1}(h_1^{(1)})\xcdot \beta\bigl(h_1^{(2)}\xcdot \ov{\ba}_{1r}\bigr)\xcdot \gamma(h_1^{(3)})\Bigr),\label{formula alternativac}\\
\shortintertext{while for $r\ge 0$ and $s> 1$, we have}
\ov{d}_1(\beta)(\byy) & = (-1)^r \beta\bigl((\ov{\Pi}^R(h_1)h_2 \ot_{H^{\!L}} \ov{\bh}_{3s})\ot_k \ov{\ba}_{1r}\bigr)\notag\\
& + \sum_{i=1}^{s-1} (-1)^{r+i} \beta\bigl((\ov{\bh}_{1,i-1}\ot_{H^{\!L}} \ov{h_ih_{i+1}} \ot_{H^{\!L}} \ov{\bh}_{i+2,s})\ot_k\ov{\ba}_{1r}\bigr)\notag\\
& + (-1)^{r+s} \gamma^{-1}(h_s^{(1)})\xcdot \beta\bigl(\ov{\bh}_{1,s-1}\ot_k h_s^{(2)}\xcdot \ov{\ba}_{1r}\bigr) \xcdot \gamma(h_s^{(3)}).\notag
\end{align}

\item For $r\ge 0$ and $s\ge 2$, we have
\begin{align*}
\quad\qquad\,\,\, \ov{d}_2(\beta)(\byy) & = - \gamma^{-1}(h_s^{(1)})\gamma^{-1}(h_{s-1}^{(1)})\xcdot \beta \bigl(\ov{\bh}_{1,s-2}\ot_k \mathfrak{T}(h^{(2)}_{s-1},h^{(2)}_s,\ov{\ba}_{1r})\bigr)\xcdot \gamma(h_{s-1}^{(3)}h_s^{(3)}),
\end{align*}
where $\mathfrak{T}(h_{s-1},h_s,\ov{\ba}_{1r})$ is as in Theorem~\ref{ppalhom1}(3).

\item Let $R$ be a $k$-subalgebra of $A$. If $R$ is stable under $\rho$ and $f$ takes its values in $R$, then
$$
\quad\qquad\,\,\,\ov{d}_l\bigl(\ov{X}^{r+l-1,s-l}(M)\bigr)\subseteq\ov{X}_{l-1}^{rs}(R,M),
$$
for each $r\ge 0$ and $1< l\le s$.

\end{enumerate}

\end{theorem}

\begin{proof} Mimic the proof of Theorem~\ref{ppalhom1}.
\end{proof}

\begin{remark}\label{para el calculo de ssc} By the second equality in~\eqref{pepe}, we have
$$
\bigl(\beta\xcdot \ov{\Pi}^R(h)\bigr)(\ov{\ba}_{1r})  = \gamma(\Pi^R(h_1)\xcdot\beta\bigl(\ov{\ba}_{1r}\bigr)\qquad\text{for all $h\in H$, $\beta\in \Hom_{K^e}\bigl(\ov{A}^{\ot^r},M \bigr)$ and $a_1,\dots,a_r\in A$.}
$$
This gives an alternative formula for $\ov{d}^1(\beta)(\byy)$ in~\eqref{formula alternativac}. We will use this fact in the proof of Proposition~\ref{calculo de ssc}.
\end{remark}

\begin{proposition} For each $h\in H$ the map $F^*_h\colon \bigl(\Hom_{K^e}\bigl(\ov{A}^{\ot^r},M\bigr),b^*\bigr)\longrightarrow \bigl(\Hom_{K^e}\bigl(\ov{A}^{\ot^r},M\bigr),b^*\bigr)$, defined by
$$
F^r_h(\beta)\bigl(\ov {\ba}_{1r}\bigr)\coloneqq \gamma^{-1}(h^{(1)})\xcdot \beta\bigl(h^{(2)} \xcdot \ov {\ba}_{1r}\bigr)\xcdot \gamma(h^{(3)}),
$$
is a morphism of complexes. Moreover the following facts hold

\begin{enumerate}[itemsep=0.7ex, topsep=1.0ex, label=\emph{(\arabic*)}]

\item  For each $h,l\in H$, the endomorphisms of $\Ho_{\hs K}^*(A,M)$ induced by $F^*_l\xcirc F^*_h$ and $F^*_{hl}$ coincide.

\item $F^r_l(\beta)(\ov{\ba}_{1r}) = (\beta\xcdot l)\bigl(\ov{\ba}_{1r}\bigr)$, for all $l\in H^{\!L}$ (see Remark~\ref{pepitito3}). In particular $F_1^*$ is the identity map.

\end{enumerate}
Consequently $\Ho_{\hs K}^*(A,M)$ is a right $H$-module.
\end{proposition}

\begin{proof} Mimic the proof of Proposition~\ref{F^h}.
\end{proof}

\begin{example}\label{A es K cohomologia} If $A=K$, then $\Ho_K^*(E,M) = \Ho^*\bigl(H,M^{\!K})$, where $M^{\!K} $ is considered as a right $H$-module via $m\xcdot h\coloneqq \gamma(h^{(2)})\xcdot m\xcdot \gamma^{-1}(h^{(1)})$ (for the well definition use the first identity in~\eqref{gama iota y gama gama} and Lemma~\ref{propiedad 4}(4)).
\end{example}

\begin{proposition}\label{calculo de ssc} The spectral sequence of~\cite{GGV2}*{(4.3)} satisfies
$$
E_1^{rs}=\Hom_{\!H^{\!L}}\bigl(\ov{H}^{\ot_{\!H^{\!L}}^s},\Ho_{\hs K}^r(A,M)\bigr)\qquad\text{and}\qquad E_2^{rs}=\Ho^s\bigl(H,\Ho_{\hs K}^r(A,M)\bigr).
$$
\end{proposition}

\begin{proof} For each $s,n\ge 0$, let $F_i(\ov{X}^n(M))\coloneqq \bigoplus_{s=i}^n \ov{X}^{n-s,s}(M)$\index{fza@$F_i(\ov{X}^n(M))$|dotfillboldidx}. The cochain complex $(\ov{X}^*(M),\ov{d}^*)$ is filtrated by
\begin{equation}\label{ss2}
F_0(\ov{X}^*(M))\supseteq F_1(\ov{X}^*(M))\supseteq F_2(\ov{X}^*(M))\supseteq F_2(\ov{X}^*(M))\supseteq \dots.
\end{equation}
Since the isomorphism $\Theta^*\colon (\ov{X}^*(M),\ov{d}^*)\to (\wh{X}^*(M),\wh{d}^*)$ preserves filtrations, the spectral sequence of \cite{GGV2}*{(4.3)} coincides with the one determined by the filtration~\eqref{ss2}. Thus, the formula for $E_1^{rs}$ follows from Theorem~\ref{ppalchom2}(1) and the formula for $E_2^{rs}$ follows from Theorem~\ref{ppalchom2}(2), Remark~\ref{para el calculo de ssc} and Proposition~\ref{hom}(2).
\end{proof}

\section{The cup and cap products for cleft extensions}
Let $H$, $A$, $\rho$, $f$, $E$, $K$, $M$, $\nu$, $\gamma$ and $\gamma^{-1}$ be as in Section~\ref{Hochschild homology of cleft extensions}. Thus $H$ is a weak Hopf algebra, $A$ is a weak module algebra and $f$ is convolution invertible. In this section we obtain formulas involving the complexes $(\ov{X}^*(E),\ov{d}^*)$ and $(\ov{X}_*(M),\ov{d}_*)$ that induce the cup product of $\HH_{\hs K}^*(E)$ and the cap product of $\Ho^{\hs K}_*(E,M)$. We will use freely the operators $\bullet$ and $\diamond$ introduced in \cite{GGV2}*{Section~5}.

\begin{notation} Given $h_1,\dots,h_s\in H$ and $a_1,\dots, a_r\in A$ we set $\brh_{1s}\xcdot \ov{\ba}_{1r}\coloneqq  h_1\xcdot (h_2\xcdot (\dots \xcdot(h_s\xcdot \ov{\ba}_{1r}))\dots)$\index{hlll@$\bh_{1s}\cdot \ov{\ba}_{1r}$|dotfillboldidx}.
\end{notation}

\begin{definition}\label{producto star} Let $\beta\in \ov{X}^{rs}(E)$ and $\beta'\in \ov{X}^{r's'}(E)$. We define $\beta\centerdot \beta'\in \ov{X}^{r'',s''}(E)$ by
$$
(\beta\centerdot\beta')(\byy)\coloneqq (-1)^{r's}\gamma_{_{\!\times}}^{-1}(\brh_{s+1,s''}^{(1)})\beta\bigl(\ov{\bh}_{1s}^{(1)}\ot_k \brh_{s+1,s''}^{(2)}\xcdot\ov{\ba}_{1r}\bigr)\gamma_{_{\!\times}}(\brh_{s+1,s''}^{(3)})\beta'\bigl(\ov{\bh}_{s+1,s''}^{(4)}\ot_k\ov{\ba}_{r+1,r''} \bigr),
\index{zzmb@$\centerdot$|dotfillboldidx}
$$
where $r'' \coloneqq r+r'$, $s'' \coloneqq s+s'$, $h_1,\dots,h_{s''}\in H$, $a_1,\dots, a_{r''}\in A$ and $\byy\coloneqq \ov{\bh}_{1s''}\ot_k \ov{\ba}_{1r''}$.
\end{definition}

\begin{proposition}\label{cor cup product'} If $f$ takes its values in $K$, then the cup product in $\HH_{\hs K}^*(E)$ is induced by  $\centerdot$.
\end{proposition}

\begin{proof} By~\cite{GGV2}*{Corollary~5.3} it suffices to prove that
$$
\Theta \bigl(\beta\centerdot \beta'\bigr) = \Theta(\beta)\bullet \Theta(\beta')\qquad\text{for each $\beta\in\ov{X}^{rs}(E)$ and $\beta'\in\ov{X}^{r's'}(E)$.}
$$
For this, take $\bx\!\coloneqq\! \wt{\gamma}_{\hs A}\bigl(\brh_{1s''})\ot \ov{\ba}_{1,r''}$, where $r''\!\coloneqq\! r+r'$, $s''\!\coloneqq\! s+s'$, $h_1,\dots,h_{s''}\!\in\! H$ and $a_1,\dots, a_{r''}\!\in\! A$. By Lemma~\ref{auxiliar 5} and the definitions of $\Theta$, $\bullet$ and $\centerdot$,
\begin{align*}
\bigl(\Theta(\beta&) \bullet \Theta(\beta')\bigr)(\bx) = (-1)^{rs'} \Theta(\beta)\bigl(\wt{\gamma}_{\hs A}\bigl(\brh_{1s})\ot\brh_{s+1,s''}^{(1)} \xcdot \ov{\ba}_{1r}\bigr)\Theta(\beta')\bigl(\wt{\gamma}_{\hs A}\bigl(\brh_{s+1,s''}^{(2)}) \ot \ov{\ba}_{r+1,r''}\bigr)\\
& = (-1)^{rs'+rs+r's'} \gamma_{_{\!\times}}(\brh_{1s}^{(1)})\beta\bigl(\ov{\bh}_{1s}^{(2)} \ot_k \brh_{s+1,s''}^{(1)}\xcdot\ov{\ba}_{1r}\bigr) \gamma_{_{\!\times}}(\brh_{s+1,s''}^{(2)})\beta\bigl(\ov{\bh}_{s+1,s''}^{(3)}\ot_k\ov{\ba}_{r+1,r''}\bigr)\\
& = (-1)^{r''s''}\gamma_{_{\!\times}}(\brh_{1s''}^{(1)})\bigl(\beta\centerdot\beta'\bigr)\bigl(\bh_{1s''}^{(2)} \ot_k \ov{\ba}_{1r''} \bigr)\\
& = \Theta \bigl(\beta\centerdot \beta'\bigr)(\bx),
\end{align*}
as desired.
\end{proof}

\begin{definition}\label{producto star cap} Let $\beta\in \ov{X}^{r's'}(E)$ and let $\byy\coloneqq \ov{\bh}_{1s}\ot_{H^{\!L}} [m\ot \ov{\ba}_{1r}]\in \ov{X}_{rs}(M)$, where $m\in M$, $a_1,\dots,a_r\in A$ and $h_1,\dots,h_s\in H$. Assume that $r\ge r'$ and $s\ge s'$. We define $\byy\ast \beta$ by
$$
\byy\ast \beta\coloneqq (-1)^{rs'+r's'} \ov{\bh}_{s'+1,s}^{(4)} \ot_{H^{\!L}}\Bigl[m\xcdot \gamma_{_{\!\times}}^{-1}(\brh_{s'+1,s}^{(1)}) \beta\bigl(\ov{\bh}_{1s'}^{(1)}\ot_k \brh_{s'+1,s}^{(2)}\xcdot \ov{\ba}_{1r'} \bigr)\gamma_{_{\!\times}}(\brh_{s'+1,s}^{(3)}) \ot
\ov{\ba}_{r'+1,r}\Bigr].\index{zzmc@$\ast$|dotfillboldidx}
$$
If $r'>r$ or $s'>s$, then we set $\byy\ast \beta = 0$.
\end{definition}

\begin{proposition}\label{cap product caso simple'} If $f$ takes its values in $K $, then in terms of the complexes $(\ov{X}_*(M),\ov{d}_*)$ and $(\ov{X}^*(E),\ov{d}^*)$, the cap product is induced by the operation $\ast$.
\end{proposition}

\begin{proof} By~\cite{GGV2}*{Corollary~5.7} it suffices to prove that
$$
\Lambda\bigl(\byy\ast \beta\bigr) = \Lambda(\byy)\diamond \Theta(\beta) \qquad\text{for each $\byy\in \ov{X}_{rs}(M)$ and $\beta\in\ov{X}^{r's'}(E)$.}
$$
If $r'>r$ or $s'>s$, then both sides of this equality are zero. Then we can assume that $r'\le r$, $s'\le s$ and $\byy$ is as in Definition~\ref{producto star cap}. We have
\begin{align*}
& \Lambda(\byy)\diamond \Theta(\beta) = (-1)^{rs}\bigl[m\xcdot\gamma_{_{\!\times}}^{-1}(\mathrm{h}_{1s}^{(1)})\ot_{\hs A} \wt{\gamma}_{\hs A} (\brh_{1s}^{(2)}) \ot \ov{\ba}_{1r}\bigr] \diamond \Theta(\beta)\\
& = (-1)^{rs+r'(s-s')}\bigl[m\xcdot \gamma_{_{\!\times}}^{-1}(\brh_{1s}^{(1)})\Theta(\beta)\bigl(\wt{\gamma}_{\hs A} (\brh_{1s'}^{(2)})\ot \brh_{s'+1,s}^{(2)}\xcdot \ov{\ba}_{1r'}\bigr) \ot_{\hs A}\wt{\gamma}_{\hs A}(\brh_{s'+1,s}^{(3)})\ot\ov{\ba}_{r'+1,r}\bigr]\\
& = (-1)^{rs+r's}\bigl[m\xcdot \gamma_{_{\!\times}}^{-1}(\brh_{1s}^{(1)})\gamma_{_{\!\times}}(\brh_{1s'}^{(2)}) \beta\bigl(\ov{\bh}_{1s'}^{(3)} \ot_k \brh_{s'+1,s}^{(2)}\xcdot \ov{\ba}_{1r'}\bigr) \ot_{\hs A}\wt{\gamma}_{\hs A}(\brh_{s'+1,s}^{(3)})\ot\ov{\ba}_{r'+1,r}\bigr]\\
& = (-1)^{rs+r's}\bigl[m\xcdot \gamma_{_{\!\times}}^{-1}(\brh_{s'+1,s}^{(1)})\gamma(1^{(1)})\beta\bigl(\ov{\bh}_{1s'}^{(1)}\xcdot 1^{(2)}\ot_k \brh_{s'+1,s}^{(2)}\xcdot \ov{\ba}_{1r'}\bigr) \ot_{\hs A}\wt{\gamma}_{\hs A}(\brh_{s'+1,s}^{(3)})\ot\ov{\ba}_{r'+1,r}\bigr]\\
& = (-1)^{rs+r's}\bigl[m\xcdot \gamma_{_{\!\times}}^{-1}(\brh_{s'+1,s}^{(1)})\beta\bigl(\ov{\bh}_{1s'}^{(1)}\ot_k \brh_{s'+1,s}^{(2)}\xcdot \ov{\ba}_{1r'}\bigr) \ot_{\hs A}\wt{\gamma}_{\hs A} (\brh_{s'+1,s}^{(3)})\ot\ov{\ba}_{r'+1,r}\bigr]\\
& = (-1)^{rs'+r's'}\Lambda\Bigl(\ov{\bh}_{s'+1,s}^{(4)}\ot_{H^{\!L}}\Bigl[m\xcdot\gamma_{_{\!\times}}^{-1}(\brh_{s'+1,s}^{(1)})\beta\bigl( \ov{\bh}_{1s'}^{(1)}\ot_k \!\brh_{s'+1,s}^{(2)}\xcdot \ov{\ba}_{1r'}\bigr)\gamma_{_{\!\times}}(\brh_{s'+1,s}^{(3)})\!\ot \ov{\ba}_{r'+1,r}\Bigr]\! \Bigr)\\
& = \Lambda(\byy\ast \beta),
\end{align*}
as desired. Here the first equality holds by the definition of $\Lambda$; the second one, by the definition of $\diamond$; the third one, by the definition of $\Theta(\beta)$; the fourth one, by Lemma~\ref{auxiliar 6}; the fifth one, since $\beta$ is a right $H^{\!L}$-module morphism and $\gamma(1^{(1)})\gamma(S(1^{(2)}))=1_E$; the fifth one, by the definition of $\Lambda$ and Lemma~\ref{auxiliar 5}; and the last one, by the definition of $\ast$.
\end{proof}

\section{Cyclic homology of cleft extensions}\label{cyclic homology of cleft extensions}
Let $H$, $A$, $\rho$, $f$, $E$, $K$, $M$, $\nu$, $\gamma$ and $\gamma^{-1}$ be as in Section~\ref{Hochschild homology of cleft extensions}. Thus $H$ is a weak Hopf algebra, $A$ is a weak module algebra and $f$ is convolution invertible. In this section we prove that the mixed complex $(\wh{X}_*(E),\wh{d}_*,\wh{D}_*)$\index{dh@$\wh{D}_*$|dotfillboldidx} of \cite{GGV2}*{Section~6} is isomorphic to a simpler mixed complex $(\ov{X}_*(E),\ov{d}_*,\ov{D}_*)$. Let
$$
\Theta_*\colon (\wh{X}_*(E),\wh{d}_*)\longrightarrow (\ov{X}_*(E),\ov{d}_*)\quad\text{and}\quad \Lambda_*\colon (\ov{X}_*(E),\ov{d}_*) \longrightarrow (\wh{X}_*(E),\wh{d}_*),
$$
be the maps introduced in the proof of Theorem~\ref{ppalhom}. For each $n\!\ge\! 0$, let $\ov{D}_n\colon \ov{X}_n(E)\to \ov{X}_{n+1}(E)$ be the map defined by $\ov{D}_n\coloneqq \Theta_{n+1}\xcirc \wh{D}_n\xcirc \Lambda_n$.\index{di@$\ov{D}_n$|dotfillboldidx}

\begin{theorem}\label{complejo mezclado que da la homologia ciclica caso cleft} The triple $\bigl(\ov{X}_*(E),\ov{d}_*,\ov{D}_*\bigr)$ is a mixed complex that gives the Hochschild and that cyclic type~ho\-mol\-ogies of the $K$-algebra $E$. More precisely, the mixed complexes $\bigl(\ov{X}_*(E),\ov{d}_*,\ov{D}_*\bigr)$ and $\cramped{\bigl(E\ot \ov{E}^{\ot^*},b_*,B_*\bigr)}$ are homotopically equivalent.
\end{theorem}

\begin{proof} Since $\Theta_*$ and $\Lambda_*$ are inverse one of each other it is clear that $\bigl(\ov{X}_*(E),\ov{d}_*,\ov{D}_*\bigr)$ is a mixed complex and that $\Theta_*\colon \bigl(\wh{X}_*(E),\wh{d}_*,\wh{D}_*\bigr)\to \bigl(\ov{X}_*(E),\ov{d}_*,\ov{D}_*\bigr)$ is an isomorphism of mixed complexes. Consequently the result follows from \cite{GGV2}*{Theorem~6.3}.
\end{proof}

\begin{definition}\label{def ov{D}0, etc} For each $r,s\ge 0$, let $\wh{D}^0_{rs}\colon \wh{X}_{rs}(E)\to \wh{X}_{r,s+1}(E)$\index{dj@$\wh{D}^0_{rs}$|dotfillboldidx} and $\wh{D}^1_{rs}\colon \wh{X}_{rs}(E)\to \wh{X}_{r+1,s}(E)$\index{dk@$\wh{D}^1_{rs}$|dotfillboldidx} be the maps introduced in \cite{GGV2}*{Definition~6.5}. We define
$\ov{D}^0_{rs}\colon \ov{X}_{rs}(E)\to \ov{X}_{r,s+1}(E)$\index{djj@$\ov{D}^0_{rs}$|dotfillboldidx} and $\ov{D}^1_{rs}\colon \ov{X}_{rs}(E)\to\ov{X}_{r+1,s}(E)$\index{dkk@$\ov{D}^1_{rs}$|dotfillboldidx} by $\ov{D}^0_{rs}\coloneqq  \Theta_{r,s+1}\xcirc \wh{D}^0_{rs}\xcirc \Lambda_{rs}$ and $\ov{D}^1_{rs}\coloneqq  \Theta_{r+1,s}\xcirc \wh{D}^1_{rs}\xcirc \Lambda_{rs}$, respectively.
\end{definition}

\begin{proposition}\label{Connes operator'} let $R$ be a subalgebra of $A$ and let $\byy \coloneqq \ov{\bh}_{1s} \ot_{H^{\!L}} [a_0\xcdot \gamma(h_0)\ot\ov{\ba}_{1r}]\in \ov{X}_{rs}(E)$. If $R$ is stable under $\rho$ and $f$ takes its values in $R$, then $\ov{D}(\byy) = \ov{D}^0(\byy) + \ov{D}^1(\byy)$ module $\bigoplus_{s=0}^i \ov{X}^1_{s,n+1-s}(R,M)$.
\end{proposition}

\begin{proof} This follows by Remark~\ref{estable bajo rho = estable bajo chi} and \cite{GGV2}*{Proposition~6.6(1)}.
\end{proof}

\begin{corollary} If $f$ takes its values in $K$, then $\ov{D} = \ov{D}^0 + \ov{D}^1$.
\end{corollary}

We next compute the maps $\ov{D}^0$ and $\ov{D}^1$. We will need the following proposition.

\begin{notation} Given $h_1,\dots,h_s\in S$ and $1\le j\le s$, we set $S_{_{\!\times}}(\brh_{1s}) \coloneqq S(h_s)S(h_{s-1})\cdots S(h_1)\index{sc@$S_{_{\times}}(\brh_{1s})$|dotfillboldidx}$.
\end{notation}

\begin{proposition}\label{priequ'} For each $s\in \mathds{N}$ there exists a map $T_s\colon H^{\ot_k^{s+1}}\to A$ such that
\begin{equation*}
\gamma(h_0)\gamma_{_{\!\times}}^{-1}(\brh_{1s}) = \jmath_{\nu}\bigl(T_s(h_0^{(1)}\!\ot_k\! \brh_{1s}^{(2)})\bigr) \gamma\bigl(h_0^{(2)}S_{_{\!\times}}(\brh_{1s}^{(1)})\bigr)\quad\text{and}\quad T_s(h_0^{(1)}\!\ot_k\! \brh_{1s}^{(2)})\bigl(h_0^{(2)}S_{_{\!\times}}(\brh_{1s}^{(1)})\xcdot 1_A\bigr) = T_s(\brh_{0s}).
\end{equation*}
\end{proposition}

\begin{proof} We will proceed by an inductive argument. Assume first that $s=1$. By equality~\eqref{gamma{-1}}, the first identity in~\eqref{gama iota y gama gama} and Theorem~\ref{weak crossed prod}(d), we have
\begin{align*}
\gamma(h_0)\gamma^{-1}(h_1) & = \gamma(h_0)\jmath_{\nu}\bigl(f^{-1}\bigl(S(h_1^{(2)})\ot_k h_1^{(3)}\bigr)\bigr)\gamma\bigl(S(h_1^{(1)})\bigr)\\
& = \jmath_{\nu}\bigl(h_0^{(1)}\!\cdot\! f^{-1}\bigl(S(h_1^{(2)})\ot_k h_1^{(3)}\bigr)\bigr)\gamma(h_0^{(2)})\gamma\bigl(S(h_1^{(1)})\bigr)\\
& = \jmath_{\nu}\Bigl(\bigl(h_0^{(1)}\!\cdot\! f^{-1}\bigl(S(h_1^{(3)})\ot_k h_1^{(4)}\bigr)\bigr)f\bigl(h_0^{(2)}\ot_k S(h_1^{(2)})\bigr) \Bigr)\gamma\bigl(h_0^{(3)}S(h_1^{(1)})\bigr).
\end{align*}
So, we have $T_1(h_{01}) = \bigl(h_0^{(1)}\!\cdot\! f^{-1}\bigl(S(h_1^{(2)})\ot_k h_1^{(3)}\bigr)\bigr)f\bigl(h_0^{(2)}\ot_k S(h_1^{(1)})\bigr)$. The first equality in the statement follows immediately from the equality in~The\-orem~\ref{weak crossed prod}(1). Assume now that $s\ge 1$ and there exists $T_s$ as in the statement. Then
\begin{align*}
& \gamma(h_0)\gamma_{_{\!\times}}^{-1}(\brh_{1,s+1}) = \jmath_{\nu}\bigl(T_s(h_0^{(1)}\!\ot_k\! \brh_{2,s+1}^{(2)})\bigr)\gamma\bigl(h_0^{(2)} S_{_{\!\times}}(\brh_{2,s+1}^{(1)})\bigr) \jmath_{\nu}\bigl(f^{-1}\bigl(S(h_1^{(2)})\ot_k h_1^{(3)}\bigr)\bigr)\gamma\bigl(S(h_1^{(1)})\bigr)\\
& = \jmath_{\nu}\Bigl(T_s(h_0^{(1)}\!\ot_k\!\brh_{2,s+1}^{(3)}) h_0^{(2)} S_{_{\!\times}}(\brh_{2,s+1}^{(2)})\!\cdot\! f^{-1}\bigl(S(h_1^{(2)})\ot_k h_1^{(3)}\bigr)\Bigr)\gamma\bigl(h_0^{(3)} S_{_{\!\times}}(\brh_{2,s+1}^{(1)})\bigr)\gamma\bigl(S(h_1^{(1)})\bigr)\\
& = \jmath_{\nu}\Bigl(T_s(h_0^{(1)}\!\ot_k\! \brh_{2,s+1}^{(4)}) h_0^{(2)} S_{_{\!\times}}(\brh_{2,s+1}^{(3)})\!\cdot\! f^{-1}\bigl(S(h_1^{(2)})\ot_k h_1^{(3)}\bigr) f\bigl(h_0^{(3)} S_{_{\!\times}}(\brh_{2,s+1}^{(2)})\ot_k S(h_1^{(1)})\bigr)\Bigr)\gamma\bigl(h_0^{(4)} S_{_{\!\times}}(\brh_{1,s+1}^{(1)})\bigr),
\end{align*}
which provides the recursive step. The same argument as above proves the second equality in the statement.
\end{proof}

\begin{remark} By Theorem~\ref{weak crossed prod}(d), Propositions~\ref{modulo algebra debil}(3) and~\ref{priequ'}, and \cite{GGV1}*{Propositions~4.5 and 5.19},
\begin{align*}
\gamma(h_0^{(1)})\gamma_{_{\!\times}}^{-1}(\brh_{1s}^{(2)})\gamma^{-1}\bigl(h_0^{(2)}S_{_{\!\times}}(\brh_{1s}^{(1)})\bigr) & = \jmath_{\nu}\bigl(T_s(h_0^{(1)}\!\ot_k\! \brh_{1s}^{(3)})\bigr) \gamma\bigl(h_0^{(2)}S_{_{\!\times}}(\brh_{1s}^{(2)})\bigr) \gamma^{-1}\bigl(h_0^{(3)}S_{_{\!\times}}(\brh_{1s}^{(1)})\bigr)\\
& = (-1)^{jr+r}\bigl(T_s(h_0^{(1)}\!\ot_k\! \brh_{1s}^{(2)})\bigr) \gamma\bigl(\Pi^{L}\bigl(h_0^{(2)}S_{_{\!\times}}(\brh_{1s}^{(1)})\bigr)\bigr)\\
& = \jmath_{\nu}\bigl(T_s(h_0^{(1)}\!\ot_k\! \brh_{1s}^{(2)})\bigr) \jmath_{\nu}\bigl(h_0^{(2)}S_{_{\!\times}}(\brh_{1s}^{(1)})\xcdot 1_A\bigr)\\
& = \jmath_{\nu}\bigl(T_s(h_0\!\ot_k\! \brh_{1s})\bigr).
\end{align*}
Since $(A\ot \epsilon)\xcirc \jmath_{\nu} = \ide_A$, this implies that
$$
T_s(h_0\!\ot_k\! \brh_{1s}) = (A\ot\epsilon)\Bigl(\gamma(h_0^{(1)})\gamma_{_{\!\times}}^{-1}(\brh_{1s}^{(2)})\gamma^{-1}\bigl(h_0^{(2)} S_{_{\!\times}}(\brh_{1s}^{(1)})\bigr)\Bigr).
$$
\end{remark}

\begin{proposition} For $\byy\coloneqq \ov{\bh}_{1s}\ot_{H^{\!L}} [\jmath_{\nu}(a_0)\gamma(h_0)\ot\ov{\ba}_{1r}]\in \ov{X}_{rs}(E)$, we have
\begin{align*}
\ov{D}^0(\byy) &\!=\! \sum_{j=0}^s \alpha_{jrs} \bigl(\ov{\bh}_{j+1,s}^{(5)} \ot_{H^{\!L}}\! \ov{h_0^{(2)}S_{_{\!\times}}(\brh_{1s}^{(1)})}\ot_{H^{\!L}}\! \bh_{1j}^{(2)}\bigr) \ot_{H^{\!L}}\! \bigl[\gamma_{_{\!\times}}(\brh_{j+1,s}^{(4)})\jmath_{\nu}(a_0)\gamma(h_0^{(1)})\gamma_{_{\!\times}}^{-1} (\brh_{j+1,s}^{(2)})\ot \brh_{j+1,s}^{(3)}\xcdot \ov{\ba}_{1r}\bigr]
%
\shortintertext{and}
\ov{D}^1(\byy) &\!=\!\! \sum_{j=0}^s \beta_{jr} \ov{\bh}_{1s}^{(6)}\! \ot_{H^{\!L}}\!\! \Bigl[\gamma\bigl(h_0^{(3)}S_{_{\!\times}}(\brh_{1s}^{(1)})\bigr) \gamma_{_{\!\times}}(\brh_{1s}^{(5)}) \!\ot\!\ov{\ba}_{j+1,r}\!\ot\!\ov{a_0T_s(h_0^{(1)}\!\ot_k\! \brh_{1s}^{(3)}\bigr)}\!\ot\! \bigl(h_0^{(2)}S_{_{\!\times}}(\brh_{1s}^{(2)})\bigr)\xcdot \bigl(\brh_{1s}^{(4)}\xcdot \ov{\ba}_{1j}\bigr)\Bigr],
\end{align*}
where $\alpha_{jrs}\coloneqq (-1)^{js+r+s}$ and $\beta_{jr}\coloneqq (-1)^{jr+r}$.
\end{proposition}

\begin{proof} By the definition of $\Lambda$ and the first equality in Proposition~\ref{priequ'},
\begin{align*}
\Lambda(\byy) & = (-1)^{rs} \bigl[\jmath_{\nu}(a_0)\gamma(h_0)\gamma_{_{\!\times}}^{-1}(\brh_{1s}^{(1)}) \ot_{\hs A} \wt{\gamma}_{\hs A}(\brh_{1s}^{(2)})\ot \ov{\ba}_{1r}\bigr]\\
& = (-1)^{rs} \Bigl[\jmath_{\nu}\bigl(a_0T_s(h_0^{(1)}\!\ot_k\!\brh_{1s}^{(2)}\bigr)\bigr) \gamma\bigl(h_0^{(2)}S_{_{\!\times}}(\brh_{1s}^{(1)})\bigr)\ot_{\hs A} \wt{\gamma}_{\hs A}(\brh_{1s}^{(3)})\ot \ov{\ba}_{1r}\Bigr].
\end{align*}
So
\begin{align*}
& \wh{D}^0\bigl(\Lambda(\byy)\bigr)\! =\! \sum_{j=0}^s \alpha'_{jrs} \Bigl[1_E\ot_{\hs A}\wt{\gamma}_{\hs A}\bigl(\brh_{j+1,s}^{(4)}\bigr)\!\ot_{\hs A} \! \stackon[-8pt]{$\jmath_{\nu}\bigl(a_0T_s(h_0^{(1)}\!\ot_k\!\brh_{1s}^{(2)}\bigr)\bigr) \gamma\bigl(h_0^{(2)}S_{_{\!\times}}(\brh_{1s}^{(1)})\bigr)$} {\vstretch{1.5}{\hstretch{6.5} {\widetilde{\phantom{\;\;\;\;\;\;\;\;\,}}}}} \!\ot_{\hs A}\! \wt{\gamma}_{\hs A}(\brh_{1j}^{(3)})\!\ot\! \brh_{j+1,s}^{(3)}\xcdot \ov{\ba}_{1r}\Bigr]\\
\shortintertext{and}
%
%
&\wh{D}^1\bigl(\Lambda(\byy)\bigr) = \sum_{j=0}^r \beta'_{jrs} \Bigl[\gamma\bigl(h_0^{(3)}S_{_{\!\times}}(\brh_{1s}^{(1)})\bigr)\ot_{\hs A}\wt{\gamma}_{\hs A} (\brh_{1s}^{(5)})\ot\ov{\ba}_{j+1,r}\ot\ov{a_0T_s(h_0^{(1)}\!\ot_k\! \brh_{1s}^{(3)}\bigr)}\ot \bigl(h_0^{(2)}S_{_{\!\times}}(\brh_{1s}^{(2)})\bigr)\xcdot \bigl(\brh_{1s}^{(4)}\xcdot \ov{\ba}_{1j}\bigr)\Bigr],
\end{align*}
where $\alpha'_{jrs}\coloneqq (-1)^{rs+js+s}$ and $\beta'_{jrs}\coloneqq (-1)^{rs+jr+r+s}$. Now, we have
\begin{align*}
%
& (-1)^{rs+r} \Theta\Bigl(\bigl[1_E\ot_{\hs A} \wt{\gamma}_{\hs A}\bigl(\brh_{j+1,s}^{(4)}\bigr) \ot_{\hs A} \stackon[-8pt]{$\jmath_{\nu}\bigl(a_0T_s(h_0^{(1)}\!\ot_k\! \brh_{1s}^{(2)})\bigr) \gamma\bigl(h_0^{(2)}S_{_{\!\times}}(\brh_{1s}^{(1)})\bigr)$}{\vstretch{1.5}{\hstretch{6.5}{\widetilde{\phantom{\;\;\;\;\;\;\;\;\,}}}}} \ot_{\hs A} \wt{\gamma}_{\hs A}(\brh_{1j}^{(3)})\ot \brh_{j+1,s}^{(3)}\xcdot \ov{\ba}_{1r}\bigr]\Bigr)\\
%
%
& = \bigl(\ov{\bh}_{j+1,s}^{(5)} \ot_{H^{\!L}}\! \ov{h_0^{(2)}S_{_{\!\times}}(\brh_{1s}^{(1)})}\ot_{H^{\!L}}\! \ov{\bh}_{1j}^{(4)}\bigr) \ot_{H^{\!L}}\! \bigl[\gamma_{_{\!\times}}(\brh_{j+1,s}^{(4)})\jmath_{\nu}(a_0)\gamma(h_0^{(1)})\gamma_{_{\!\times}}^{-1} (\brh_{1s}^{(2)})\gamma_{_{\!\times}} (\brh_{1j}^{(3)})\ot \brh_{j+1,s}^{(3)}\xcdot \ov{\ba}_{1r}\bigr] \\
& = \bigl(\ov{\bh}_{j+1,s}^{(5)} \ot_{H^{\!L}}\! \ov{h_0^{(2)}S_{_{\!\times}}(\brh_{1s}^{(1)})}\ot_{H^{\!L}}\! \bh_{1j}^{(2)}\xcdot 1^{(2)}\bigr) \ot_{H^{\!L}}\! \bigl[\gamma_{_{\!\times}}(\brh_{j+1,s}^{(4)})\jmath_{\nu}(a_0)\gamma(h_0^{(1)})\gamma_{_{\!\times}}^{-1} (\brh_{j+1,s}^{(2)})\gamma(1^{(1)})\ot \brh_{j+1,s}^{(3)}\xcdot \ov{\ba}_{1r}\bigr]\\
& = \bigl(\ov{\bh}_{j+1,s}^{(5)} \ot_{H^{\!L}}\! \ov{h_0^{(2)}S_{_{\!\times}}(\brh_{1s}^{(1)})}\ot_{H^{\!L}}\! \bh_{1j}^{(2)}\bigr) \ot_{H^{\!L}}\! \bigl[\gamma_{_{\!\times}}(\brh_{j+1,s}^{(4)})\jmath_{\nu}(a_0)\gamma(h_0^{(1)})\gamma_{_{\!\times}}^{-1} (\brh_{j+1,s}^{(2)})\gamma(1^{(1)}S(1^{(2)}))\ot \brh_{j+1,s}^{(3)}\xcdot \ov{\ba}_{1r}\bigr]\\
& = \bigl(\ov{\bh}_{j+1,s}^{(5)} \ot_{H^{\!L}}\! \ov{h_0^{(2)}S_{_{\!\times}}(\brh_{1s}^{(1)})}\ot_{H^{\!L}}\! \bh_{1j}^{(2)}\bigr) \ot_{H^{\!L}}\! \bigl[\gamma_{_{\!\times}}(\brh_{j+1,s}^{(4)})\jmath_{\nu}(a_0)\gamma(h_0^{(1)})\gamma_{_{\!\times}}^{-1} (\brh_{j+1,s}^{(2)})\ot \brh_{j+1,s}^{(3)}\xcdot \ov{\ba}_{1r}\bigr],
\end{align*}
where the first equality holds by the definition on $\Theta$ and the first identity in Proposition~\ref{priequ'}; the second one, by Lemma~\ref{auxiliar 6}; the third one, by the definition of the action in~\eqref{acciones} and~\cite{GGV1}*{Proposition~2.22(2)}; and the fourth one, since $\gamma(1^{(1)}S(1^{(2)})) = \gamma(1) = 1$. Similarly
\begin{align*}
& (-1)^{rs+s} \Theta\Bigl(\Bigl[\gamma\bigl(h_0^{(3)}S_{_{\!\times}}(\brh_{1s}^{(1)})\bigr)\ot_{\hs A}\wt{\gamma}_{\hs A} (\brh_{1s}^{(5)})\ot\ov{\ba}_{j+1,r}\ot\ov{a_0T_s(h_0^{(1)}\!\ot_k\! \brh_{1s}^{(3)}\bigr)}\ot \bigl(h_0^{(2)}S_{_{\!\times}}(\brh_{1s}^{(2)})\bigr)\xcdot \bigl(\brh_{1s}^{(4)}\xcdot \ov{\ba}_{1j}\bigr)\Bigr]\Bigr)\\
& = \ov{\bh}_{1s}^{(6)} \ot_{H^{\!L}}\! \Bigl[\gamma\bigl(h_0^{(3)}S_{_{\!\times}}(\brh_{1s}^{(1)})\bigr) \gamma_{_{\!\times}}(\brh_{1s}^{(5)}) \ot\ov{\ba}_{j+1,r}\ot\ov{a_0T_s(h_0^{(1)}\!\ot_k\! \brh_{1s}^{(3)}\bigr)}\ot \bigl(h_0^{(2)}S_{_{\!\times}}(\brh_{1s}^{(2)})\bigr)\xcdot \bigl(\brh_{1s}^{(4)}\xcdot \ov{\ba}_{1j}\bigr)\Bigr].
\end{align*}
From these facts it follows that the formulas in the statement are true.
\end{proof}
\printindex

%
%
%

\end{document}